\documentclass[11pt]{article}
\usepackage{amsmath,amssymb,amsthm,mathrsfs}
\usepackage{hyperref}
\hypersetup{colorlinks=true, allcolors=black}
\usepackage{graphicx}
\usepackage[a4paper,margin=1in]{geometry}
\usepackage{bm,mathtools}
\usepackage{enumitem}
\usepackage{xcolor}
\usepackage{etoolbox}
\usepackage[ruled,lined,linesnumbered]{algorithm2e}

\numberwithin{equation}{section}

\newtheorem{theorem}{Theorem}[section]
\newtheorem{lemma}[theorem]{Lemma}
\newtheorem{proposition}[theorem]{Proposition}
\newtheorem{corollary}[theorem]{Corollary}

\newtheorem{introtheorem}{Theorem}

\theoremstyle{definition}
\newtheorem{definition}[theorem]{Definition}
\newtheorem{remark}[theorem]{Remark}
\newtheorem{example}[theorem]{Example}

\graphicspath{{figures/}}

\newcommand{\R}{\mathbb{R}}
\newcommand{\C}{\mathbb{C}}
\newcommand{\T}{\mathbb{T}}

\newcommand{\Pcal}{\mathcal{P}}
\newcommand{\F}{\mathcal{F}}
\newcommand{\E}{\mathcal{E}}

\newcommand{\Wtwo}{\operatorname{W}_2}
\newcommand{\Ent}{\operatorname{Ent}}
\newcommand{\Nbhd}{\mathcal{N}}

\newcommand{\vol}{\mathrm{vol}}
\newcommand{\potential}[1][\mu]{\dot{H}^1_{#1}}
\newcommand{\dualspace}[1][\mu]{\dot{H}^{-1}_{#1}}
\newcommand{\rieszmap}[1][\mu]{\mathcal{I}_{#1}}
\newcommand{\dotHk}[1][k]{\dot{H}^{#1}}
\newcommand{\dotHminusone}{\dot{H}^{-1}}
\newcommand{\constants}{\mathbf{1}}

\begin{document}

\title{Feedback Control and Local Convexification of \\ Wasserstein Gradient Flows}
\author{Dante Kalise, Lucas M. Moschen, and Grigorios A. Pavliotis}
\date{}
\maketitle

\begin{abstract}
    Building on the feedback stabilization framework that we developed for McKean--Vlasov PDEs, we treat a class of entropy-regularized free energies $\F = \E + \sigma\Ent$ on the flat torus, with $\E$ possibly nonconvex and nonlocal, and show that the resulting controlled system admits a gradient flow formulation in Wasserstein space.
    Specifically, we realize the Wasserstein Hessian at a target stationary measure $\bar\mu$ as a self-adjoint operator with compact resolvent and show that its negative is unitarily equivalent to the generator of the linearized dynamics.
    The resulting feedback stabilization can then be interpreted as a local convexification of the free energy landscape: for any prescribed threshold $\delta > 0$, the control, obtained from an algebraic Riccati equation for the linearized problem, induces a finite-rank perturbation of the Wasserstein Hessian that lifts its spectrum above $\delta$, yielding local displacement convexity on a H\"older neighborhood of $\bar\mu$ in regimes where the uncontrolled energy is nonconvex or slowly contracting.
    Under a local well-posedness assumption, $\bar\mu$ is moreover a locally exponentially stable equilibrium of the closed-loop dynamics, with rate at least $\delta$.
    The framework covers McKean--Vlasov models and extends to moment-constrained Fokker--Planck dynamics and Fokker--Planck equations on closed Riemannian manifolds.
\end{abstract}

\section{Introduction}
\label{sec:introduction}

Nonlocal Fokker--Planck equations arising as the mean-field limit of interacting
particle systems~\cite{Sznitman1991}, and the mean-field Langevin dynamics used in optimization and machine learning~\cite{HuEtAl2021, Chizat2022}, are both prominent instances of gradient flows on the space of probability measures with respect to the quadratic Wasserstein distance. 
This formulation was introduced for the Fokker--Planck equation
in~\cite{JordanKinderlehrerOtto1998}, and the associated formal Riemannian structure, including tangent spaces and Wasserstein Hessians, was developed in~\cite{Otto2001} in the context of the porous-medium equation. 
Concretely, the Wasserstein gradient flow of a free energy $\F$ is the continuity equation
\begin{equation}
    \label{eq:intro_gradient_flow}
    \partial_t \mu_t = \nabla\cdot\left(\mu_t\nabla\frac{\delta\F}{\delta\mu}(\mu_t)\right), \qquad \mu_t \in \Pcal_2(\T^d),
\end{equation}
where $\frac{\delta\F}{\delta\mu}$ denotes the first variation of $\F$. 
In this paper, we consider a class of entropy-regularized free energies on $\Pcal_2(\T^d)$ of the form
\[
\F(\mu) = \E(\mu) + \sigma\Ent(\mu), \qquad \Ent(\mu) \coloneqq \int_{\T^d} \mu\log\mu\,dx, \qquad \sigma > 0,
\]
where $\T^d$ is the flat $d$-dimensional torus and $\E$ is a possibly nonconvex and nonlocal energy, such as the McKean--Vlasov energy $\int_{\T^d} V \, d\mu + \frac{1}{2} \iint_{\T^d\times\T^d} W(x-y) \, d\mu(x) \, d\mu(y)$~\cite{CarrilloMccannVillani2003}, or one of the form $\int_{\T^d} U(K*\mu) \, dx$ arising in mean-field Langevin training~\cite{Chizat2022}.

The long-time behavior of~\eqref{eq:intro_gradient_flow} is well understood under convexity assumptions: strong displacement convexity yields a unique equilibrium $\bar\mu$ together with exponential convergence to it~\cite{Santambrogio2017}.
The same conclusion follows from a global Log-Sobolev inequality~\cite{Chizat2022}, a condition that may hold beyond the displacement convex setting.
Even then, the rate can be slow, since it is governed by the spectral gap of the Wasserstein Hessian at $\bar\mu$.
This is typical of nonconvex interaction energies, where the dynamics may become metastable and remain near a local minimum for exponentially long times before transitioning to a stable state~\cite{BolleyGentilGuillin2012, CarrilloEtAl2020, Monmarche2025}.
Near such a metastable equilibrium, the local convergence rate is governed by the Poincar\'e constant of the corresponding stationary measure~\cite{WangChizat2026, MonmarcheReygner2025}.
The same energy landscape may also contain unstable equilibria.
Slow convergence and linear instability are local spectral phenomena encoded in a single operator, the Wasserstein Hessian at the equilibrium, and can therefore be addressed together by modifying its spectrum.
We do so by controlling the linearized dynamics.
Therefore, we fix a stationary state $\bar\mu$ and study how control can modify the Wasserstein Hessian in its neighborhood.

Our aim is to locally convexify the free energy around a prescribed stationary state by modifying the dynamics to strengthen or create local attraction toward $\bar\mu$.
Given a threshold $\delta > 0$, we construct a modification of the free energy that retains the gradient-flow structure of~\eqref{eq:intro_gradient_flow} and adds a finite-rank perturbation to its Wasserstein Hessian at $\bar\mu$, shifting the spectrum above $\delta$.
The starting point is the infinite-horizon linear-quadratic feedback control constructed in~\cite{KaliseMoschenPavliotis2025} for the linearized McKean--Vlasov equation.
In that work, the stabilization problem is formulated in a weighted $L^2$ space and its optimal feedback is characterized by an algebraic Riccati equation.
Here, we show that this Riccati equation leads to the finite-rank Hessian perturbation described above.
To establish this connection, we first realize the Hessian of $\F$ at $\bar\mu$ as a self-adjoint operator with compact resolvent, in the spirit of~\cite{denzler2005fast,seis2014long}, and then identify its negative with the linearization of~\eqref{eq:intro_gradient_flow} through the mapping $\rieszmap[\bar\mu] \phi \coloneqq -\nabla \cdot (\bar\mu\nabla\phi)$, so that the curvature of the energy at $\bar\mu$ and the decay of the linearized flow are governed by the same discrete spectrum.

\begin{introtheorem}
    \label{introthm:hessian}
    Under Assumptions~\ref{ass:H1}--\ref{ass:H3}, the Wasserstein Hessian of $\F$ at $\bar\mu$ admits a self-adjoint realization
    \[
    A : D(A) \subset X \to X, \qquad X \coloneqq \potential[\bar\mu]
    \]
    with compact resolvent. 
    The linearization of~\eqref{eq:intro_gradient_flow} around $\bar\mu$ admits a closed realization in $X'$, given by
    \[
    L = -\rieszmap[\bar\mu]A\rieszmap[\bar\mu]^{-1}.
    \]
    In $X$, the corresponding linearized dynamics are $\partial_t\phi = -A\phi$.
\end{introtheorem}

Here, the quadratic form represented by $A$ is
\[
a(\varphi,\psi) = H_\E(\varphi,\psi) + \sigma\int_{\T^d} \nabla^2\varphi:\nabla^2\psi\,d\bar\mu,
\]
where $H_\E$ denotes the Wasserstein Hessian of the energy $\E$ as a bilinear form in $X$.
The entropy contribution gives the coercive part of the form, while $H_\E$ is bounded on $X$.
The precise form domain and proof are given in Section~\ref{sec:hessian}.

Given any prescribed threshold $\delta > 0$, compactness of the resolvent implies that only finitely many eigenvalues of $A$ lie below $\delta$. 
The perturbation is therefore chosen on the finite-dimensional space of eigenfunctions associated with eigenvalues $\lambda \le \delta$.
More precisely, we define an operator $B : \R^m \to X$ whose range spans this subspace, and verify the Hautus condition
\[
\ker(A - \lambda I) \cap \ker B^* = \{0\} \quad \text{for every eigenvalue } \lambda \le \delta,
\]
so that the algebraic Riccati equation has a unique bounded nonnegative solution $\Pi$ under a detectability assumption.
The following result collects the consequences of this construction at the nonlinear level by proving that the closed-loop equation is again a Wasserstein gradient flow, of an explicitly modified energy that is locally displacement convex near $\bar\mu$, and $\bar\mu$ becomes locally exponentially stable.

\begin{introtheorem}
    \label{introthm:convexification}
    Let $A$ be the Hessian realization from Theorem~\ref{introthm:hessian} and fix $\delta > 0$.
    Let $\{e_1, \dots, e_m\}$ be an $X$-orthonormal basis of the subspace of eigenfunctions of $A$ associated with eigenvalues $\lambda \le \delta$, set $B u \coloneqq \sum_{j=1}^m u_j e_j$, and let $Q$ be a bounded, self-adjoint, nonnegative operator on $X$, diagonal in the eigenfunctions of $A$, such that $(-A + \delta I, Q^{1/2})$ is detectable.
    Then the algebraic Riccati equation, understood in the weak sense,
    \[
    (A - \delta I)^* \Pi + \Pi(A - \delta I) + \Pi BB^* \Pi = Q,
    \]
    has a unique self-adjoint, nonnegative, and bounded solution $\Pi$, which is diagonal. 
    The corresponding modified operator $A_\Pi \coloneqq A + BB^*\Pi$ is self-adjoint with a closed and coercive form.
    Moreover, the closed-loop equation is the Wasserstein gradient flow of a modified energy $\F_{\rm cl}$.
    Under Assumptions~\ref{ass:H4}--\ref{ass:H5}, the modified energy $\F_{\rm cl}$ is locally $\hat{\lambda}$-displacement convex for some $\hat{\lambda} > \delta$ on a H\"older neighborhood of $\bar\mu$.
    Finally, if Assumption~\ref{ass:H6} holds, $\bar\mu$ is a locally exponentially stable equilibrium of the closed-loop equation, with decay rate at least $\delta$.
\end{introtheorem}

The Riccati equation yields the bilinear form $a_\Pi(\phi,\psi) \coloneqq a(\phi,\psi) + \langle BB^*\Pi\phi, \psi\rangle_X$ satisfying 
\[
a_\Pi (\phi,\phi) \ge \lambda_\Pi \|\phi\|_X^2 \qquad \forall \phi\in D(a),
\]
for some $\lambda_\Pi > \delta$.
Under Assumptions~\ref{ass:H4}--\ref{ass:H5}, for every $\hat{\lambda} \in (0,\lambda_\Pi)$, there exists $r_{\hat{\lambda}} \in (0,1)$ such that
\[
\mathrm{Hess}_{\Wtwo} \F_{\rm cl}(\mu)(\nabla\phi,\nabla\phi) \ge \hat{\lambda} \int_{\T^d} |\nabla\phi|^2 \, d\mu
\]
for every probability density $\mu$ with $\|\mu/\bar\mu - 1\|_{L^\infty(\T^d)} \le r_{\hat{\lambda}}$ and every $\phi \in W^{2,\infty}(\T^d)$.
Combined with the regularity of the optimal transport map between densities close to $\bar\mu$ in $C^{0,\alpha}(\T^d)$, this estimate yields local $\hat\lambda$-displacement convexity of $\F_{\rm cl}$ on a H\"older neighborhood of $\bar\mu$.
Because the perturbation defined through $B$ and $Q$ is finite-rank and determined by finitely many eigenfunctions and eigenvalues of $A$, the feedback is computable from a Galerkin discretization of the Hessian. 
Section~\ref{sec:applications-numerics} and Appendix~\ref{app:numerics} illustrate this on double-well and Kuramoto interactions, on moment-constrained dynamics, and on the sphere.

\subsection{Related work}

The variational formulation of the Fokker--Planck equation as a Wasserstein gradient flow was introduced by Jordan--Kinderlehrer--Otto~\cite{JordanKinderlehrerOtto1998}. 
Otto's formal Riemannian calculus~\cite{Otto2001} was subsequently complemented by the metric theory of gradient flows developed in~\cite{AmbrosioGigliSavare2008}; see also~\cite{Santambrogio2017}.
The second-order calculus underlying the Wasserstein Hessian appears in~\cite{Otto2001, Lott2008, Gigli2012}, and this perspective has also been extended to total variation flow and related fourth-order evolutions~\cite{CarlierPoon2019}.
The most direct precedents for our construction are spectral realizations of displacement Hessians, or of the corresponding linearized operators, on weighted potential spaces.
In fast diffusion, the complete spectrum of the Hessian operator associated with the rescaled flow around the stationary state is computed and related to long-time asymptotics~\cite{denzler2005fast}.
For the confined porous-medium equation, the displacement Hessian is realized as a self-adjoint operator with purely discrete spectrum, and its eigenfunctions describe higher-order asymptotics~\cite{seis2014long}.
Those works use the realized operator to analyze the uncontrolled dynamics.
Here, this self-adjoint realization is used for entropy-regularized free energies and then modified by a perturbation defined through an algebraic Riccati equation.

In the context of nonconvex Wasserstein gradient flows, stability and metastability are studied at the local level. 
Classical spectral results for McKean--Vlasov dynamics go back to~\cite{Tamura1984}, and criteria for local exponential stability in the Wasserstein-$1$ distance have recently been obtained on both $\R^d$ and the torus~\cite{Cormier2025}.
More quantitatively, in the setting of~\cite{WangChizat2026}, the sharp local convergence rate is governed by the Poincar\'e spectral gap of the stationary measure and agrees with the rate predicted by a tight positivity estimate for the Wasserstein Hessian.
For double-well potentials, a local nonlinear Log-Sobolev inequality can also be established~\cite{MonmarcheReygner2025}, providing a complementary route to local exponential convergence. 
Related Kramers-type phase transitions for nonlocal Fokker--Planck equations with a time-dependent constraint were studied in~\cite{HerrmannNiethammerVelazquez2014}.
Rather than deriving stability from a functional inequality satisfied by the original energy, we modify the energy so as to change its Hessian at a prescribed equilibrium.

The finite-rank perturbation in this paper is chosen by Riccati stabilization of the linearized equation.
A nearby line of work is mean-field control, where a central planner acts on the population law~\cite{HuangCainesMalhame2012,BensoussanFrehseYam2013,FornasierSolombrino2014,AlbiEtAl2017,CarmonaDelarue2018I}.
The direct control of continuity equations, formulated at the level of measure-valued solutions and allowing nonlocal velocities, was introduced and developed in~\cite{PiccoliRossi2013,PiccoliRossi2014}.
Feedback laws have also been proposed for Fokker--Planck equations~\cite{BreitenKunischPfeiffer2018}, with nonlocal extensions in~\cite{KaliseMoschenPavliotis2025}.
Approximate controllability with localized vector fields is treated in~\cite{DuprezMoranceyRossi2019}, and stabilization around reference trajectories of nonlocal continuity equations in~\cite{PogodaevRossi2024}.
Our perturbation is instead finite-rank, selected from the Wasserstein Hessian spectrum, and designed through the Riccati equation to change the local curvature of the free energy.

From an optimal-control perspective, necessary conditions such as the Pontryagin Maximum Principle have been established for continuity equations without diffusion in Wasserstein space~\cite{BonnetRossi2019}, while the value function has been characterized through Hamilton--Jacobi--Bellman equations for multi-agent systems~\cite{JimenezMarigondaQuincampoix2020} and through dynamic programming for sparse mean-field problems~\cite{CavagnariMarigondaPiccoli2020}.
Optimal-control formulations for Fokker--Planck equations are developed in~\cite{AnnunziatoBorzi2010}, with second-order optimality conditions in~\cite{AronnaTroltzsch2021} and nonlinear extensions in~\cite{Barbu2023}.
A complementary perspective interprets Euclidean and Wasserstein gradient flows as optimal controlled evolutions for suitable action functionals~\cite{ChenGeorgiouPavon2025}.
These works are mainly variational, open-loop, or finite-horizon.
Sparse feedback stabilization has been studied for alignment models~\cite{CaponigroEtAl2013}, while model-predictive control with mean-field limits appears in~\cite{AlbiPareschi2018}.
Mean-field games~\cite{LasryLions2007}, by contrast, model competitive equilibria of optimizing agents and lead to coupled forward--backward systems.
Seeking optimal feedback for the nonlinear dynamics would instead require solving a Hamilton--Jacobi--Bellman equation on $\Pcal_2(\T^d)$, in the tradition of the infinite-dimensional theory of Crandall and Lions~\cite{CrandallLions1985}, for which viscosity solutions have been studied for gradient flows~\cite{ConfortiKraaijTonon2023, ConfortiKraaijTamaniniTonon2026} and for controlled McKean--Vlasov equations on the torus~\cite{SonerYan2024}.
We avoid this route by linearizing at $\bar\mu$, at the price of a feedback that is optimal only for the linearized problem.

Finally, the work is related to sampling and variational inference, where Wasserstein gradient flow formulations over Gaussian and Gaussian-mixture families appear in~\cite{LambertEtAl2022}.
Convergence of the unadjusted Langevin algorithm is governed by Log-Sobolev inequalities~\cite{BakryEmery1985}, with quantitative nonasymptotic rates in convex settings~\cite{Dalalyan2017,VempalaWibisono2023}.
For mean-field dynamics, analogous results under suitable Log-Sobolev inequalities are obtained in~\cite{Chizat2022}.
Nonasymptotic rates are also available for some nonconvex landscapes~\cite{RaginskyRakhlinTelgarsky2017}.
Complementary constructive directions include acceleration via nonreversible drifts~\cite{LelievreNierPavliotis2013,DuncanNuskenPavliotis2017} and consensus-based global optimization~\cite{CarrilloEtAl2018}.

\subsection{Organization}

Section~\ref{sec:setting} introduces the notation, main definitions, and assumptions.
Section~\ref{sec:hessian} proves Theorem~\ref{introthm:hessian}, while Section~\ref{sec:feedback} establishes the Riccati-based convexification and stabilization results of Theorem~\ref{introthm:convexification}.
Section~\ref{sec:applications-numerics} treats McKean--Vlasov models, constrained dynamics, and the extension to closed manifolds, including the sphere example.
The appendices contain the auxiliary Hessian and displacement calculations, the nonlinear stability estimates, and the numerical methodology.

\section{Preliminaries}
\label{sec:setting}

We write $\T^d = \R^d/\mathbb{Z}^d$, $|\cdot|$ for the Euclidean norm on $\R^d$, and $d_{\T^d}(x,y)$ for the distance on $\T^d$.
Let $\Pcal_2(\T^d)$ be the space of probability measures on $\T^d$ with finite second moments, endowed with the quadratic Wasserstein distance $\Wtwo$.
This metric underlies Otto's formal Riemannian viewpoint on Wasserstein space~\cite{Otto2001,AmbrosioGigliSavare2008}.
At a measure $\mu \in \Pcal_2(\T^d)$, tangent vectors are represented, in Otto's calculus, by gradient vector fields in
\[
T_\mu\Pcal_2(\T^d) \coloneqq \overline{\{\nabla\varphi : \varphi\in C^\infty(\T^d)\}}^{L^2(\mu)}.
\]
Following~\cite{denzler2005fast}, we represent tangent vectors by weakly differentiable functions through the map $\varphi \mapsto \nabla\varphi$, which leads to the pre-Hilbert space
\[
\potential \coloneqq \left\{ \varphi : \T^d \to \R \, \Big| \, \int_{\T^d} |\nabla\varphi|^2 \, d\mu < \infty\right\} \big/ \{\|\cdot\|_{\potential} = 0\},
\]
endowed with the inner product $\langle \varphi, \psi \rangle_{\potential} \coloneqq \int_{\T^d} \nabla\varphi \cdot \nabla\psi \, d\mu$.
The dual of $\potential$ is denoted by $\dualspace \coloneqq (\potential)^\prime$ with the inner product
\[
\langle \nu, \eta \rangle_{\dualspace} \coloneqq \left\langle \rieszmap[\mu]^{-1}\nu, \rieszmap[\mu]^{-1}\eta \right\rangle_{\potential}, \qquad \nu, \eta \in \dualspace.
\]
The associated Riesz map is
\[
\rieszmap : \potential \to \dualspace, \qquad
\left\langle \rieszmap \varphi, \psi\right\rangle_{\dualspace,\potential} \coloneqq \langle \varphi, \psi \rangle_{\potential},
\]
for $\varphi, \psi \in \potential$.
If $\mu$ has a density (also denoted by $\mu$) that satisfies $0 < c \le \mu \le C < +\infty$, then $\potential$ is a Hilbert space, as we prove in Lemma~\ref{lem:hilbert_structure}.
In this case, by the Riesz representation theorem, $\rieszmap$ is an isometric isomorphism from $\potential$ onto $\dualspace$.
Moreover, in the sense of distributions, $\rieszmap \varphi = -\nabla\cdot(\mu \nabla\varphi)$, as follows by integration by parts on the torus.

A curve $(\mu_s)_{s \in [0,1]}$ in $\Pcal_2(\T^d)$ is a constant speed $\Wtwo$-geodesic if, for every $s, t \in [0,1]$, 
\[
\Wtwo(\mu_s, \mu_t) = |s-t| \Wtwo(\mu_0,\mu_1).
\]
For a function $\varphi$, write $T_\varphi \coloneqq \mathrm{Id} + \nabla\varphi$, understood through its periodic extension to $\R^d$.
If $\mu \in \Pcal_2(\T^d)$ is absolutely continuous and $\nu \in \Pcal_2(\T^d)$, there is a $\mu$-a.e.\ unique optimal transport map from $\mu$ to $\nu$ for the quadratic cost, of the form $T_\psi = \mathrm{Id} + \nabla\psi$ with $\frac{1}{2}|\cdot|^2 + \psi$ convex on $\R^d$ and $\psi$ $\mathbb{Z}^d$-periodic.
Moreover $|T_\psi(x) - x| = d_{\T^d}(x, T_\psi(x))$ for $\mu$-a.e.\ $x$~\cite{cordero1999transport}.
Set $\mathscr{D} \coloneqq \left\{\nabla\varphi : \varphi \in W^{2,\infty}(\T^d)\right\}$, so that, by mollification, $\mathscr{D} \subset T_\mu\Pcal_2(\T^d)$ for every $\mu \in \Pcal_2(\T^d)$.
In particular, if $\mu$ is absolutely continuous and $\nabla\varphi \in \mathscr{D}$, then for $|s|$ sufficiently small the curve $\mu_s = (T_{s\varphi})_{\sharp}\mu$ is a local $\Wtwo$-geodesic starting from $\mu$ with initial velocity $\nabla\varphi$ (see Lemma~\ref{lem:local_displacement_interpolation}).

Let $\mathcal{F} : \Pcal_2(\T^d) \to \R \cup \{+\infty\}$ be a functional.
We say that $\mathcal{F}$ admits a {\em Wasserstein Hessian} at $\mu$ on $\mathscr{D}$ if, for every $\nabla\varphi \in \mathscr{D}$, the map $s \mapsto \mathcal{F}(\mu_s)$ is twice differentiable at $s = 0$ along a local $\Wtwo$-geodesic $(\mu_s)_{|s|< \varepsilon}$ satisfying $\mu_0 = \mu$ and having initial velocity $\nabla \varphi$, and if the resulting second variation defines a quadratic form on $\mathscr{D}$.
Then, the Wasserstein Hessian of $\mathcal{F}$ at $\mu$ is defined as
\[
\mathrm{Hess}_{\Wtwo} \mathcal{F}(\mu)(\nabla\varphi, \nabla\varphi) \coloneqq \frac{d^2}{ds^2} \mathcal{F}(\mu_s)\Big|_{s=0}.
\]
The associated symmetric bilinear form $\mathrm{Hess}_{\Wtwo} \mathcal{F}(\mu)(\nabla \varphi, \nabla\psi)$ on $\mathscr{D} \times \mathscr{D}$ is obtained by polarization.
In this paper, we use the following local version of displacement convexity.

\begin{definition}[Local $\lambda$-displacement convexity]
    \label{def:local-strong-displacement-convexity}
    Let $\F : \Pcal_2(\T^d) \to \R \cup \{+\infty\}$, $\lambda \in \R$, and $\mathcal{U} \subseteq \Pcal_2(\T^d)$.
    We say that $\F$ is {\em $\lambda$-displacement convex on $\mathcal{U}$} if, for every $\mu_0, \mu_1 \in \mathcal{U}$, the constant speed $\Wtwo$-geodesic $(\mu_s)_{s \in [0,1]}$ joining them satisfies
    \[
    \F(\mu_s) \le (1-s)\F(\mu_0) + s \F(\mu_1) - \lambda\frac{s(1-s)}{2} \Wtwo^2(\mu_0,\mu_1), \quad \forall s \in [0,1].
    \]
    We say that $\F$ is {\em locally $\lambda$-displacement convex near $\mu$} if it is $\lambda$-displacement convex on a set which is a neighborhood of $\mu$ in $C^{0,\alpha}(\T^d)$ for some $\alpha \in (0,1)$.
\end{definition}

Given a proper and lower semicontinuous function $\E : \Pcal_2(\T^d) \to \R \cup \{+\infty\}$, we consider entropy-regularized free energies of the form
\[
\F(\mu) = \E(\mu) + \sigma\Ent(\mu),
\]
where $\sigma > 0$ is a temperature parameter and $\Ent(\mu) \coloneqq \int_{\T^d} \mu\log\mu\,dx$ if $\mu$ is absolutely continuous with respect to Lebesgue measure and $+\infty$ otherwise.
We assume throughout that $\E$ admits a first variation
\[
\Theta_E(\mu) \coloneqq \frac{\delta\E}{\delta\mu}(\mu),
\]
so that the formal Wasserstein gradient flow is well-defined and written as
\[
\partial_t\mu_t = \nabla\cdot\left(\mu_t\nabla\frac{\delta\F}{\delta\mu}(\mu_t)\right) = \nabla\cdot\left(\mu_t \nabla \Theta_E(\mu_t)\right) + \sigma \Delta \mu_t.
\]

Let $\bar\mu \in \operatorname{dom}(\F)$ be a stationary point of $\F$ and write $\Phi(x) = \Theta_E(\bar\mu)(x)$.
The stationarity condition reads as
\begin{equation}
    \label{eq:stationarity}
    \Phi + \sigma(\log\bar\mu + 1) = c_{\bar\mu} \qquad \text{a.e.\ on } \T^d
\end{equation}
for some $c_{\bar\mu} \in \R$.
Since the analysis is local around $\bar\mu$, we work on the space $X \coloneqq \potential[\bar\mu]$, which, under the assumptions below, is a Hilbert space by Lemma~\ref{lem:hilbert_structure}.
Additionally, since $X$ is a quotient space, we identify each equivalence class with its unique representative satisfying
\begin{equation}
    \label{eq:mean_zero}
    \int_{\T^d} \varphi\,d\bar\mu = 0.
\end{equation}

We make the following structural assumptions near $\bar\mu$, which ensure that the Wasserstein Hessian of $\F$ at $\bar\mu$ admits an operator realization on $X$.
\begin{enumerate}[label=(H\arabic*),ref=(H\arabic*)]
    \item\label{ass:H1} The first variation of $\E$ at $\bar\mu$ satisfies $\Phi \in W^{1,\infty}(\T^d)$.

    \item\label{ass:H2} The Wasserstein Hessian of $\E$ at $\bar\mu$, initially defined on $\mathscr{D}$, extends to a bounded symmetric bilinear form on $X$.
    More precisely, for $\varphi, \psi \in W^{2,\infty}(\T^d)$, we set
    \[
    H_\E(\varphi,\psi) \coloneqq \mathrm{Hess}_{\Wtwo} \E(\bar\mu)(\nabla\varphi,\nabla\psi), 
    \]
    and assume that there exists $C_E > 0$ such that for all $\varphi, \psi \in W^{2,\infty}(\T^d)$, 
    \[
    |H_\E(\varphi,\psi)| \le C_E\|\varphi\|_X\|\psi\|_X. 
    \]
    Hence $H_\E$ extends uniquely by continuity to a bounded symmetric bilinear form on $X$.
\end{enumerate}

\subsection{Consequences of the structural assumptions}
\label{subsec:consequences-structural-assumptions}

For $k \ge 1$, we write $\dotHk(\T^d) \coloneqq H^k(\T^d)/\R$ and endow this quotient space with the norm
\[
\|\varphi\|_{\dotHk} \coloneqq \|\nabla^k \varphi\|_{L^2} = \left(\sum_{i_1,\ldots,i_k=1}^d \left\| \partial_{i_1} \cdots \partial_{i_k}\varphi \right\|_{L^2}^2\right)^{1/2},
\]
where, for matrix- and tensor-valued functions, the pointwise norm is the Frobenius norm.
This norm is well-defined on the quotient space because adding a constant does not change the derivatives.
Moreover, if $\|\nabla^k \varphi\|_{L^2} = 0$, then all $k$-th order weak derivatives of $\varphi$ vanish.
In Fourier variables this gives $\sum_{n \in \mathbb{Z}^d}|n|^{2k} |\widehat\varphi_n|^2 = 0$, and hence $\widehat\varphi_n = 0$ for every $n \neq 0$ and $\varphi$ is a.e.\ constant on $\T^d$.
This norm is equivalent to the canonical quotient norm $\|\varphi\|_{\dotHk}^{\rm quot}
\coloneqq \inf_{c \in \R} \|\varphi+c\|_{H^k}$.
Indeed, the infimum is attained by $\varphi_0 = \varphi - \int_{\T^d} \varphi \, dx$, and the Poincaré--Wirtinger inequality gives
\[
\|\varphi\|_{\dotHk}^{\rm quot} = \|\varphi_0\|_{H^k} \simeq \|\nabla^k \varphi_0\|_{L^2} = \|\nabla^k \varphi\|_{L^2}.
\]
Finally, we denote $\dot{H}^{-k} = (\dotHk[k](\T^d))'$.
For $k \ge 0$ and $\alpha \in (0,1]$, we write $C^{k,\alpha}(\T^d)$ for the space of $k$ times continuously differentiable functions on $\T^d$ whose derivatives of order $k$ are $\alpha$-Hölder continuous, endowed with the norm $\|f\|_{C^{k,\alpha}} \coloneqq \|f\|_{C^{k}} + [\nabla^k f]_\alpha$, where
\[
\|f\|_{C^{k}} \coloneqq \sum_{j=0}^{k} \|\nabla^j f\|_{L^\infty}, \qquad [g]_\alpha \coloneqq \sup_{x \neq y} \frac{|g(x)-g(y)|}{d_{\T^d}(x,y)^\alpha}.
\]
In particular, $\|f\|_{C^{0,\alpha}} = \|f\|_{L^\infty} + [f]_\alpha$.
We now record two lemmas that will be used throughout the paper.

\begin{lemma}
    \label{lem:mu_bar_bounds}
    Let $\bar\mu \in \operatorname{dom}(\F)$ be a stationary point of $\F$.
    Assume~\ref{ass:H1}.
    Then:
    \begin{itemize}
        \item[(i)] $\bar\mu(x) = Z^{-1}\exp(-\Phi(x)/\sigma)$ for a.e.\ $x\in\T^d$, where $Z \coloneqq \int_{\T^d} e^{-\Phi(y)/\sigma}\,dy$.
        \item[(ii)] There exist constants $0 < c_0 \le C_0 < \infty$ such that for a.e.\ $x \in \T^d$, $c_0 \le \bar\mu(x) \le C_0$.
        \item[(iii)] $\bar\mu\in W^{1,\infty}(\T^d)$ and for a.e.\ $x\in\T^d$,
        \begin{equation}
            \label{eq:derivative_barmu}
            \nabla\bar\mu(x) = -\frac{1}{\sigma} \bar\mu(x) \nabla\Phi(x). 
        \end{equation}
    \end{itemize}
\end{lemma}

\begin{proof}
    From \eqref{eq:stationarity}, $\bar\mu(x) = \exp(c_{\bar\mu} / \sigma - 1) e^{-\Phi(x)/\sigma}$.
    The normalization $\int_{\T^d} \bar\mu\,dx = 1$ gives
    \[
    \exp\left(\frac{c_{\bar\mu}}{\sigma}-1\right) = \left(\int_{\T^d} e^{-\Phi(y)/\sigma}\,dy\right)^{-1} = Z^{-1}.
    \]
    Since $\Phi \in W^{1,\infty}(\T^d)$ and $\T^d$ is compact, $\Phi$ is bounded. 
    Hence $e^{-\Phi/\sigma}$ is bounded above and below away from zero, which gives the claimed bounds on $\bar\mu$. 
    Finally, the chain rule gives $e^{-\Phi/\sigma} \in W^{1,\infty}(\T^d)$ and~\eqref{eq:derivative_barmu}.
\end{proof}

\begin{lemma}
\label{lem:hilbert_structure}
    Let $\mu \in L^1(\T^d)$ be a probability density with $0 < c \le \mu \le C < +\infty$.
    Then $\potential$ is a Hilbert space and the norm in $\potential$ is equivalent to the norm on $\dotHk[1](\T^d)$.
    In particular, $X$ is a Hilbert space and each equivalence class in $X$ has a unique representative satisfying~\eqref{eq:mean_zero}.
    Moreover, $X'$ is linearly isomorphic to
    \[
    \left\{f\in H^{-1}(\T^d) : \langle f, \constants \rangle_{H^{-1}(\T^d), H^1(\T^d)} = 0 \right\} 
    \]
    with equivalent norms, where $H^{-1}(\T^d) = [H^1(\T^d)]^\prime$.
\end{lemma}

\begin{proof}
    The bounds $0 < c \le \mu \le C$ imply
    \[ 
    c\int_{\T^d} |\nabla\varphi|^2\,dx \le \int_{\T^d} |\nabla\varphi|^2\,d\mu \le C\int_{\T^d} |\nabla\varphi|^2\,dx.
    \]
    Hence $\potential$ is linearly isomorphic to $\dotHk[1](\T^d)$ with equivalent norms. 
    Since $\dotHk[1](\T^d)$ is complete, $\potential$ is a Hilbert space.
    Applying this to $\mu = \bar\mu$, Lemma~\ref{lem:mu_bar_bounds} implies that $X$ is a Hilbert space.
    For each equivalence class in $X$, subtracting the constant $\int_{\T^d} \varphi \, d\bar\mu$ gives the unique representative with zero $\bar\mu$-mean.

    Since $X$ is linearly isomorphic to $\dotHk[1](\T^d)$ with equivalent norms, it is enough to identify the dual of $\dotHk[1](\T^d)$.
    Let $q : H^1(\T^d) \to \dotHk[1](\T^d)$ be the quotient map.
    If $\Lambda \in (\dotHk[1](\T^d))'$, then $f \coloneqq \Lambda \circ q \in H^{-1}(\T^d)$ and $\langle f, \constants \rangle_{H^{-1}(\T^d), H^1(\T^d)} = \Lambda(q(\constants)) = 0$.
    Thus $f\in \dotHminusone$.
    Conversely, if $f \in \dotHminusone$, then
    \[
    \Lambda_f([\varphi]) \coloneqq \langle f,\varphi\rangle_{H^{-1}(\T^d), H^1(\T^d)}
    \]
    is well-defined on $\dotHk[1](\T^d)$.
    The maps $\Lambda \mapsto \Lambda \circ q$ and $f \mapsto \Lambda_f$ are inverse to each other.
    Moreover, both maps are continuous, and the inverse map is continuous by the quotient norm characterization.
    Hence $(\dotHk[1](\T^d))'$ is identified with this subspace of $H^{-1}(\T^d)$.
    Since $X \simeq \dotHk[1](\T^d)$ with equivalent norms, it follows that $X' \simeq \dotHminusone$ with equivalent norms.
\end{proof}

We close this subsection with two model classes for which Assumptions~\ref{ass:H1}--\ref{ass:H2} can be verified directly.

\begin{example}[McKean--Vlasov]
    \label{ex:mckean-vlasov-energy}
    Let $V, W \in W^{2,\infty}(\T^d)$ and, for simplicity, $W$ even. 
    Set
    \begin{equation}
        \label{eq:mckean_vlasov_equation}
        \E(\mu) = \int_{\T^d} V(x) \, d\mu(x) + \frac{1}{2}\iint_{\T^d\times\T^d} W(x-y)\,d\mu(x)d\mu(y).
    \end{equation}
    Then $\Phi = V + W * \bar\mu$ and~\ref{ass:H1} holds.
    For $\varphi, \psi \in W^{2, \infty}(\T^d)$, the Wasserstein Hessian is
    \begin{align*}
        H_{\E}(\varphi,\psi) &= \int_{\T^d} \bar\mu\,\nabla^2 V\,\nabla\varphi \cdot \nabla\psi \, dx \\
        &\quad + \frac{1}{2}\iint_{\T^d\times\T^d} \nabla^2 W(x-y)\,(\nabla\varphi(x)-\nabla\varphi(y))\cdot(\nabla\psi(x)-\nabla\psi(y))\,d\bar\mu(x)\,d\bar\mu(y).
    \end{align*}
    Hence $|H_\E(\varphi,\psi)| \le \left(\|\nabla^2V\|_{L^\infty} + \|\nabla^2W\|_{L^\infty}\right)\|\varphi\|_X\|\psi\|_X$, so Assumption~\ref{ass:H2} holds.
\end{example}

\begin{example}[Nonlinear nonlocal interaction]
\label{ex:nonlinear-nonlocal-interaction}
    Let $K \in W^{2,\infty}(\T^d)$, $U \in C^3(\R)$, and
    \[ 
    \E(\mu) = \int_{\T^d} U((K*\mu)(x)) \, dx.
    \]
    Define $\widetilde{K}(z) \coloneqq K(-z)$. 
    Then $\Phi = \widetilde K*U'(K*\bar\mu) \in W^{1,\infty}(\T^d)$ by the regularity of $K$ and $U$ and Assumption~\ref{ass:H1} holds.
 
    For $\nu \in X'$, the derivative of the first variation is
    \[ 
    D\Theta_E[\bar\mu] \nu = \widetilde{K} * \left(U''(K * \bar\mu)(K * \nu)\right). 
    \]
    For $\varphi, \psi \in W^{2,\infty}(\T^d)$, the Wasserstein Hessian of $\E$ can be written in the form
    \[ 
    H_\E(\varphi,\psi) = \int_{\T^d} \nabla^2\Theta_E(\bar\mu)(x) \nabla\varphi(x) \cdot \nabla\psi(x) \, d\bar\mu(x) + \int_{\T^d} \nabla(D\Theta_E[\bar\mu] \rieszmap[\bar\mu]\varphi)(x) \cdot \nabla\psi(x) \, d\bar\mu(x). 
    \]
    The first term is bounded by $\|\nabla^2\Theta_E(\bar\mu)\|_{L^\infty}\|\varphi\|_X\|\psi\|_X$. 
    For the second term, integration by parts in the convolution gives bounds of the form $\|\nabla(D\Theta_E[\bar\mu]\rieszmap[\bar\mu]\varphi)\|_{L^\infty} \le C\|\varphi\|_X$, where $C$ depends only on the norms of $K$, the $C^3$-norm of $U$ on the range of $K*\bar\mu$, and $\bar\mu$. 
    Consequently, Assumption~\ref{ass:H2} holds.
\end{example}

\section{The Hessian realization and the linearized flow}
\label{sec:hessian}

This section realizes the Wasserstein Hessian of $\F$ at $\bar\mu$ as a self-adjoint operator on $X$ and identifies its negative with the linearization of the gradient flow through the Riesz map $\rieszmap[\bar\mu]$. 
We start by recalling the known expression for the Hessian of the entropy.
For a formal derivation, we refer the reader to~\cite[Section 9.1]{Villani2003}, whereas a rigorous framework can be found in~\cite{Lott2008, Gigli2012}.
See also Appendix~\ref{app:proof_prop_entropy} for a proof.

\begin{proposition}
    \label{prop:entropy_hessian}
    Let $\mu \in \Pcal_2(\T^d) \cap L^\infty(\T^d)$.
    For $\varphi, \psi\in W^{2,\infty}(\T^d)$,
    \begin{equation}\label{eq:Hess_ent}
        \mathrm{Hess}_{\Wtwo} \mathrm{Ent}(\mu) (\nabla \varphi, \nabla \psi) = \int_{\T^d} \mu(x)\,\nabla^2\varphi(x) : \nabla^2\psi(x) \, dx.
    \end{equation}
\end{proposition}

\subsection{Self-adjoint realization and spectral structure}

For $\varphi, \psi \in C^\infty(\T^d)$, we define
\begin{equation}
    \label{eq:q0_def}
    q_0(\varphi, \psi) \coloneqq \mathrm{Hess}_{\Wtwo}\mathrm{Ent}(\bar\mu)(\nabla \varphi, \nabla \psi).
\end{equation}

\begin{lemma}
    \label{lem:h2_coercivity}
    There exist constants $c, C > 0$ such that, for all $\varphi \in C^\infty(\T^d)$,
    \begin{equation}
        \label{eq:h2_inequality_q0}
        c\|\varphi\|_{\dotHk[2]}^2 \le q_0(\varphi, \varphi) \le C\|\varphi\|_{\dotHk[2]}^2.
    \end{equation}
    Consequently, $q_0$ extends uniquely to a closed symmetric form on $D(q_0) = \dotHk[2](\T^d)$, still given by~\eqref{eq:Hess_ent} with $\mu = \bar\mu$, and $q_0^{1/2}$ is an equivalent norm on $\dotHk[2](\T^d)$.
    Moreover, the form norm
    \[
    \|\varphi\|_\star^2 \coloneqq q_0(\varphi,\varphi) + \|\varphi\|_X^2
    \]
    is also equivalent to $\|\varphi\|_{\dotHk[2]}^2$.
\end{lemma}

\begin{proof}
    For $\varphi \in C^\infty(\T^d)$, Proposition~\ref{prop:entropy_hessian} and~\eqref{eq:q0_def} give the formula $q_0(\varphi,\varphi) = \int_{\T^d}|\nabla^2\varphi|^2\,d\bar\mu$.
    By Lemma~\ref{lem:mu_bar_bounds}(ii), $c_0 \le \bar\mu \le C_0$ a.e., and therefore
    \[
    c_0\|\nabla^2\varphi\|_{L^2}^2 \le q_0(\varphi,\varphi) \le C_0\|\nabla^2\varphi\|_{L^2}^2,
    \]
    which is precisely~\eqref{eq:h2_inequality_q0} for $c = c_0$ and $C = C_0$.
    Thus $q_0$ is continuous on $C^\infty(\T^d)$ with respect to the $\dotHk[2]$-norm.
    Since $C^\infty(\T^d)/\R$ is dense in $\dotHk[2](\T^d)$, $q_0$ extends uniquely to a continuous symmetric bilinear form on $\dotHk[2](\T^d)$.
    Additionally, as $\bar\mu \in L^\infty(\T^d)$, the right-hand side of~\eqref{eq:Hess_ent} is continuous on $\dotHk[2](\T^d)$ too, so the representation persists.

    Finally, by Lemma~\ref{lem:mu_bar_bounds}(ii) and the Poincaré inequality on the torus with constant $C_P$,
    \[
    \|\varphi\|_X^2 \le C_0\|\nabla\varphi\|_{L^2}^2 \le C_0 C_P \|\nabla^2\varphi\|_{L^2}^2 \le C_P C_0 c_0^{-1} q_0(\varphi,\varphi).
    \]
    Hence $q_0(\varphi,\varphi) \le q_0(\varphi,\varphi) + \|\varphi\|_X^2 \le (1 + C_P C_0 c_0^{-1}) q_0(\varphi,\varphi)$, and $\|\cdot\|_\star$ is equivalent to the $\dotHk[2]$-norm. 
    Since $\dotHk[2](\T^d)$ is complete, the extended form is closed.
\end{proof}

From now on, we denote by $q_0$ this closed extension with domain $D(q_0) = \dotHk[2](\T^d)$.

\begin{theorem}[Hessian as a closed form]
    \label{thm:hessian_form}
    Let $\bar\mu$ be a stationary point of $\F$ and assume~\ref{ass:H1}--\ref{ass:H2}. 
    Then, there exists a symmetric bilinear form
    \[
    a: D(a) \times D(a) \to \R, \qquad D(a) = \dotHk[2](\T^d),
    \]
    such that, for all $\varphi, \psi \in C^\infty(\T^d)$,
    \begin{equation}
        \label{eq:HessianF_formula}
        \mathrm{Hess}_{\Wtwo} \F(\bar\mu)(\nabla \varphi, \nabla \psi) = a(\varphi,\psi) \coloneqq H_\E(\varphi,\psi) + \sigma q_0(\varphi, \psi).
    \end{equation}
    Moreover, $a$ is densely defined, symmetric, closed, and bounded from below in the Hilbert space $(X, \langle\cdot,\cdot\rangle_X)$.
    Consequently, there exists a unique self-adjoint operator $A : D(A)\subset X \to X$ such that, for all $\varphi \in D(A)$ and $\psi \in D(a)$,
    \begin{equation}
        \label{eq:A_representation}
        a(\varphi, \psi) = \langle A\varphi, \psi\rangle_X.
    \end{equation}
    The domain of $A$ is given by
    \[
    D(A) = \big\{\varphi \in D(a) : \exists f \in X \text{ such that } a(\varphi,\psi) = \langle f,\psi\rangle_X \; \forall \psi \in D(a)\big\},
    \]
    and $A\varphi = f$.
\end{theorem}

\begin{proof}
    By~\ref{ass:H2} and Proposition~\ref{prop:entropy_hessian}, $H_\E$ is a bounded symmetric bilinear form on $X$ and
    \[
    \mathrm{Hess}_{\Wtwo}\F(\bar\mu)(\nabla \varphi,\nabla \psi) = H_\E(\varphi,\psi) + \sigma q_0(\varphi,\psi)
    \]
    for all $\varphi, \psi \in C^\infty(\T^d)$.
    Since $C^\infty(\T^d)/\R$ is dense in $\dotHk[2](\T^d)$ and $\dotHk[2](\T^d)$ is dense in $X$ through the equivalence $X \simeq \dotHk[1](\T^d)$ of Lemma~\ref{lem:hilbert_structure}, the form $a$ defined in~\eqref{eq:HessianF_formula} is densely defined on $X$.
    It is symmetric because both $q_0$ and $H_\E$ are symmetric.

    By~\ref{ass:H2}, $|H_\E(\varphi,\varphi)| \le C_E\|\varphi\|_X^2$.
    Hence $a(\varphi,\varphi) \ge -C_E\|\varphi\|_X^2$ and $a$ is bounded from below.
    Choose $\gamma > C_E$ and set
    \[
    \|\varphi\|_{a,\gamma}^2 \coloneqq a(\varphi,\varphi) + \gamma\|\varphi\|_X^2.
    \]
    Then, we have $\|\varphi\|_{a,\gamma}^2 \ge \sigma q_0(\varphi,\varphi) + (\gamma-C_E)\|\varphi\|_X^2$ and $\|\varphi\|_{a,\gamma}^2 \le \sigma q_0(\varphi,\varphi) + (\gamma + C_E)\|\varphi\|_X^2$.
    Therefore, $\|\cdot\|_{a,\gamma}$ is equivalent on $D(a)$ to $\|\cdot\|_\star$, hence to the $\dotHk[2]$-norm by Lemma~\ref{lem:h2_coercivity}.
    Since $\dotHk[2](\T^d)$ is complete, $a$ is closed.

    The representation theorem for densely defined symmetric closed forms bounded from below (e.g., \cite[Ch.~VI, Thm.~2.6]{Kato1980}) implies the existence of a unique self-adjoint operator $A$ satisfying~\eqref{eq:A_representation}. 
    Moreover, $A$ is bounded from below with the same lower bound as $a$.
\end{proof}

\begin{theorem}[Compact resolvent of $A$]
    \label{thm:compact_resolvent_a}
    Under Assumptions~\ref{ass:H1}--\ref{ass:H2}, the self-adjoint operator $A : D(A) \subset X \to X$ has compact resolvent.
\end{theorem}

\begin{proof}
    Choose $\gamma > C_E$.
    As in the proof of Theorem~\ref{thm:hessian_form}, $\|\cdot\|_{a,\gamma}$ is equivalent to $\|\cdot\|_\star$ on $D(a)$, and by Lemma~\ref{lem:h2_coercivity}, it is equivalent to the $\dotHk[2]$-norm on $D(a) = \dotHk[2](\T^d)$.
    Moreover, since $A$ is bounded from below by $-C_E$, the operator $A + \gamma I$ is invertible on $X$.

    Let $(f_n)_{n \ge 1}$ be a bounded sequence in $X$ and define $u_n \coloneqq (A + \gamma I)^{-1}f_n$.
    Then $u_n \in D(A) \subset D(a)$ and, for all $\psi \in D(a)$, we have $a(u_n, \psi) + \gamma\langle u_n, \psi\rangle_X = \langle f_n, \psi\rangle_X$.
    Taking $\psi = u_n$ gives $\|u_n\|_{a,\gamma}^2 = \langle f_n,u_n\rangle_X \le \|f_n\|_X \|u_n\|_X$.
    On the other hand, from the lower bound of $a$,
    \[
    \|u_n\|_{a,\gamma}^2 = a(u_n,u_n)+\gamma\|u_n\|_X^2 \ge (\gamma-C_E)\|u_n\|_X^2.
    \]
    Since $(f_n)$ is bounded in $X$, these two inequalities imply that $(u_n)$ is bounded in the norm $\|\cdot\|_{a, \gamma}$.
    By the equivalence of this norm with the $\dotHk[2]$-norm, $(u_n)$ is bounded in $\dotHk[2](\T^d)$.

    The embedding $\dotHk[2](\T^d) \hookrightarrow \dotHk[1](\T^d)$ is compact on the torus, and $\|\cdot\|_{\dotHk[1]}$ is equivalent to $\|\cdot\|_X$ by Lemma~\ref{lem:hilbert_structure}.
    Therefore, after passing to a subsequence, $u_n$ converges in $X$.
    Hence $(A+\gamma I)^{-1} : X \to X$ is compact and $A$ has compact resolvent.
\end{proof}

As a consequence of Theorem~\ref{thm:compact_resolvent_a}, the spectrum of $A$ consists of real eigenvalues of finite multiplicity, with no finite accumulation point.
We fix an $X$-orthonormal basis $(e_k)_{k \ge 1}$ of eigenfunctions of $A$, with corresponding eigenvalues
\[
A e_k = \lambda_k e_k, \qquad \lambda_1 \le \lambda_2 \le \cdots, \qquad \lambda_k \to+\infty.
\]
For $\delta > 0$, we denote by
\[
\Sigma_\delta \coloneqq \{k \ge 1 : \lambda_k \le \delta\}, \qquad X_\delta \coloneqq \operatorname{span}\{e_k : k \in \Sigma_\delta\},
\]
and by $P_\delta : X\to X_\delta$ the corresponding $X$-orthogonal projection. 
Since $A$ has compact resolvent, $X_\delta$ is finite-dimensional.
Throughout, we assume $\Sigma_\delta \neq \emptyset$.

\subsection{Identification with the linearized generator in \texorpdfstring{$X'$}{X'}}

Write the gradient flow~\eqref{eq:intro_gradient_flow} as $\partial_t \mu_t = \mathcal{G}(\mu_t)$, where 
\[
\mathcal{G}(\mu) \coloneqq \sigma\Delta\mu + \nabla\cdot(\mu\nabla\Theta_E(\mu)).
\] 
By stationarity, $\mathcal{G}(\bar\mu) = 0$.
Fix $p > d/2 + 2$ so that $H^p(\T^d) \hookrightarrow C^2(\T^d)$, set $\mathcal{U}_p \coloneqq \{\xi \in \dotHk[p](\T^d) : \|\nabla^2\xi\|_{L^\infty} < 1\}$, and define $\Xi(\xi) \coloneqq (T_\xi)_{\sharp}\bar\mu$.
In this setting, $T_\xi$ is a $C^1$-diffeomorphism of $\T^d$ for every $\xi \in \mathcal{U}_p$, and the pushforward derivatives used below are well defined.
By Lemma~\ref{lem:pushforward_map}, $\Xi(\xi) - \bar\mu \in X'$ and $D\Xi(0)\eta = \rieszmap[\bar\mu]\eta$.
Specifically, for perturbations of the form $\nu = \rieszmap[\bar\mu]\varphi$, our goal is to compute
\[
D\mathcal{G}(\bar\mu)\nu = \lim_{\varepsilon \to 0} \frac{\mathcal{G}(\Xi(\varepsilon\varphi))}{\varepsilon}
\]
in $\dot{H}^{-2}$, and then identify its closed realization on $X'$ with the operator induced by the Hessian form $a$.
To make this linearization in $\dot{H}^{-2}$ precise, we assume the following.

\begin{enumerate}[label=(H\arabic*),ref=(H\arabic*),start=3]
    \item\label{ass:H3} The map $\xi \mapsto \Theta_E(\Xi(\xi))$ is continuous from $\mathcal{U}_p$ to $X$, and there exists $\mathcal{K}_E \in \mathcal{L}(X', X)$ such that
    \[
    \Theta_E(\Xi(\xi)) = \Theta_E(\bar\mu) + \mathcal{K}_E(\Xi(\xi)-\bar\mu) + o(\|\Xi(\xi)-\bar\mu\|_{X'})
    \]
    in $X$ for $\xi \in \mathcal{U}_p$ sufficiently close to $0$.
\end{enumerate}

We first express $H_\E$ using the expansion of $\Theta_E(\Xi(\xi))$ at $\xi = 0$. 
The following lemma converts the differentiability of $\Theta_E(\Xi(\xi))$ as a map from $\dotHk[p](\T^d)$ to $X$ into a formula for the Wasserstein Hessian of $\E$ at $\bar\mu$. 
This formula will be used below to identify the weak form of the linearized PDE with the Wasserstein Hessian.
See Appendix~\ref{app:displacement-map} for the proof.

\begin{lemma}
    \label{lem:second_variation_pushforward}
    Assume~\ref{ass:H1}--\ref{ass:H3}.
    Then, for all $\varphi, \psi\in C^\infty(\T^d)$,
    \[
    H_\E(\varphi,\psi)=\left\langle D^2\Xi(0)[\varphi,\psi], \Phi \right\rangle_{X',X} + \left\langle \rieszmap[\bar\mu]\psi, \mathcal{K}_E \rieszmap[\bar\mu]\varphi \right \rangle_{X',X}.
    \]
\end{lemma}

Let $\varphi \in C^\infty(\T^d)$ and set $\nu \coloneqq \rieszmap[\bar\mu]\varphi$.
We claim that for $\psi \in C^{\infty}(\T^d)$,
\begin{equation}
    \label{eq:derivative_G}
    \left\langle D \mathcal{G}(\bar\mu) \nu, \psi\right\rangle_{\dot{H}^{-2}, \dotHk[2]} = \sigma\int_{\T^d}\nu \Delta \psi \, dx - \int_{\T^d} \nu \nabla\Phi \cdot \nabla\psi \, dx - \int_{\T^d} \bar\mu  \nabla(\mathcal{K}_E \nu) \cdot \nabla\psi \, dx.
\end{equation}
Indeed, set $\mu_\varepsilon \coloneqq \Xi(\varepsilon\varphi)$ and $q_\varepsilon \coloneqq (\mu_\varepsilon - \bar\mu)/\varepsilon$.
Since $\bar\mu \in W^{1, \infty}(\T^d)$, the change of variables formula for $T_{\varepsilon\varphi}$ gives $q_\varepsilon \to \nu$ in $L^2(\T^d)$, hence in $X'$.
Combining this with Assumption~\ref{ass:H3} implies
\[
\frac{\Theta_E(\mu_\varepsilon)-\Phi}{\varepsilon} = \mathcal{K}_E\nu + r_\varepsilon, \qquad \|r_\varepsilon\|_X \to 0.
\]
Thus, testing $(\mathcal{G}(\mu_\varepsilon)) / \varepsilon$ against $\psi \in C^{\infty}(\T^d)$ gives
\[
\sigma\int_{\T^d} q_\varepsilon \Delta\psi \, dx - \int_{\T^d}q_\varepsilon \nabla\Phi \cdot \nabla\psi \, dx - \int_{\T^d} \bar\mu \nabla(\mathcal{K}_E\nu + r_\varepsilon) \cdot \nabla\psi \, dx -\int_{\T^d}(\mu_\varepsilon-\bar\mu) \nabla(\mathcal{K}_E\nu + r_\varepsilon)\cdot \nabla\psi \, dx.
\]
The first two terms converge to
\[
\sigma \int_{\T^d} \nu \Delta\psi \, dx - \int_{\T^d} \nu \nabla\Phi \cdot \nabla\psi \, dx
\]
because $q_\varepsilon \to \nu$ in $L^2(\T^d)$. 
The term containing $r_\varepsilon$ in the third integral vanishes since
\[
\left|\int_{\T^d} \bar\mu\nabla r_\varepsilon \cdot \nabla\psi \,dx \right| \le \|r_\varepsilon\|_X \|\psi\|_X \to 0.
\]
Finally, the last integral vanishes because $\mu_\varepsilon \to \bar\mu$ in $L^\infty(\T^d)$, while $\mathcal{K}_E\nu + r_\varepsilon$ remains bounded in $X$.
Therefore,~\eqref{eq:derivative_G} holds.

\begin{proposition}[Identification of the linearized operator in $X'$]
    \label{prop:linearized_operator_identification}
    Assume~\ref{ass:H1}--\ref{ass:H3}. 
    Then,
    \[
    \left\langle -D\mathcal{G}(\bar\mu)\rieszmap[\bar\mu]\varphi,\psi\right\rangle_{\dot{H}^{-2}, \dotHk[2]} = a(\varphi, \psi)
    \]
    for all $\varphi, \psi\in C^\infty(\T^d)$.
    Moreover, this bilinear form extends by density to all $\varphi, \psi \in D(a)$.
    Set
    \[
    D(L) \coloneqq \left\{\nu \in \rieszmap[\bar\mu]D(a) : \exists f \in X'; \; \forall \psi \in D(a), \; \left\langle D\mathcal{G}(\bar\mu)\nu, \psi\right\rangle_{\dot{H}^{-2}, \dotHk[2]} = \left\langle f,\psi\right\rangle_{X',X} \right\}.
    \]
    For $\nu \in D(L)$, we define $L\nu \coloneqq f$.
    Then
    \[
    D(L) = \rieszmap[\bar\mu]D(A), \qquad L = -\rieszmap[\bar\mu]A\rieszmap[\bar\mu]^{-1}.
    \]
\end{proposition}

\begin{proof}
    Let $\varphi, \psi \in C^\infty(\T^d)$ and set $\nu = \rieszmap[\bar\mu]\varphi$.
    Integrating by parts twice and using $-\sigma\nabla \bar\mu = \bar\mu \nabla \Phi$ leads to
    \begin{align*}
        -\sigma \int_{\T^d} \nu \Delta \psi \, dx &= -\sigma \int_{\T^d} \bar\mu \nabla \varphi \cdot \nabla \Delta \psi \, dx = \sigma \sum_{i,j=1}^d \int_{\T^d} \partial_j (\bar\mu \partial_i \varphi) \partial_{ij} \psi \, dx \\
        &= \sigma \sum_{i,j=1}^d \int_{\T^d} \partial_j \bar\mu \partial_i \varphi \partial_{ij} \psi \, dx + \sigma \sum_{i,j=1}^d \int_{\T^d} \bar\mu \partial_{ij} \varphi \partial_{ij} \psi \, dx \\
        &= -\sum_{i,j=1}^d \int_{\T^d} \bar\mu \partial_j \Phi \partial_i \varphi \partial_{ij} \psi \, dx + \sigma q_0(\varphi, \psi).
    \end{align*}
    Hence, using again that $\nu = -\sum_{i=1}^d (\partial_i \bar\mu \partial_i \varphi + \bar\mu \partial_{ii} \varphi)$, we obtain
    \begin{align*}
        -\sigma \int_{\T^d} \nu \Delta \psi \, dx + \int_{\T^d} \nu \nabla \Phi \cdot \nabla \psi \, dx &= \sigma q_0(\varphi, \psi) - \sum_{i,j=1}^d \int_{\T^d} \partial_j \Phi \partial_i (\bar\mu \partial_i \varphi \partial_j \psi) \, dx \\
        &= \sigma q_0(\varphi, \psi) + \left\langle D^2 \Xi(0)[\varphi, \psi], \Phi \right\rangle_{X', X},
    \end{align*}
    where we apply Lemma~\ref{lem:pushforward_map} in the last equality.

    Moreover, $\int_{\T^d} \bar\mu \nabla(\mathcal{K}_E \nu) \cdot \nabla \psi \, dx = \left\langle \rieszmap[\bar\mu]\psi, \mathcal{K}_E \rieszmap[\bar\mu]\varphi \right\rangle_{X', X}$, and, by Lemma~\ref{lem:second_variation_pushforward},
    \begin{equation*}
        \left\langle -D \mathcal{G}(\bar\mu) \rieszmap[\bar\mu]\varphi, \psi \right\rangle_{\dot{H}^{-2}, \dotHk[2]} = \sigma q_0(\varphi, \psi) + H_\E(\varphi, \psi) = a(\varphi, \psi).
    \end{equation*}
    Since smooth functions are dense in $D(a) = \dotHk[2](\T^d)$, and the expression in~\eqref{eq:derivative_G} defines a continuous bilinear form on $\dotHk[2](\T^d) \times \dotHk[2](\T^d)$ under Assumption~\ref{ass:H3} and Lemma~\ref{lem:mu_bar_bounds}, the identity extends to all $\varphi, \psi \in D(a)$.

    For the second part, by the definition of $D(L)$, for $\varphi \in D(a)$, there exists $f \in X'$ such that $\left\langle D\mathcal{G}(\bar\mu) \rieszmap[\bar\mu]\varphi, \psi \right\rangle_{\dot{H}^{-2}, \dotHk[2]} = \left\langle f, \psi \right\rangle_{X', X}$ for all $\psi \in D(a)$. 
    Thus, for all $\psi \in D(a)$, 
    \[
    a(\varphi, \psi) = -\left\langle f, \psi \right\rangle_{X', X} = -\left\langle \rieszmap[\bar\mu]^{-1} f, \psi \right\rangle_X.
    \]
    Hence $\varphi \in D(A)$, $A \varphi = -\rieszmap[\bar\mu]^{-1} f$, and $L \rieszmap[\bar\mu]\varphi = -\rieszmap[\bar\mu] A \varphi $ because $L\rieszmap[\bar\mu]\varphi = f$.
    Therefore $D(L) \subset \rieszmap[\bar\mu]D(A) $.
    Conversely, if $\varphi \in D(A)$, for every $\psi \in D(a)$, one has
    \begin{equation*}
        \left\langle D\mathcal{G}(\bar\mu) \rieszmap[\bar\mu]\varphi, \psi \right\rangle_{\dot{H}^{-2}, \dotHk[2]} = -a(\varphi, \psi) = -\left\langle A \varphi, \psi \right\rangle_X = \left\langle -\rieszmap[\bar\mu] A \varphi, \psi \right\rangle_{X', X}.
    \end{equation*}
    Hence $\rieszmap[\bar\mu]\varphi \in D(L)$ and $L \rieszmap[\bar\mu]\varphi = -\rieszmap[\bar\mu] A \varphi$. 
    Therefore $D(L) = \rieszmap[\bar\mu]D(A)$ and $L = -\rieszmap[\bar\mu] A \rieszmap[\bar\mu]^{-1}$.
\end{proof}

\section{Feedback control and local convexification of the free energy}
\label{sec:feedback}

The previous section shows that the generator of the linearized dynamics at $\bar\mu$ is $L = -\rieszmap[\bar\mu] A \rieszmap[\bar\mu]^{-1}$, where $A$ is a self-adjoint realization of the Wasserstein Hessian.
When $A$ has negative eigenvalues, $\bar\mu$ is linearly unstable, and when $\lambda_1 > 0$ is small, the convergence is slow.
We therefore construct a feedback control whose range is $X_\delta$, the subspace associated with eigenvalues $\lambda_k \le \delta$, for a prescribed threshold $\delta > 0$.
The resulting perturbation shifts this part of the spectrum of $A$ and hence modifies $L$ accordingly.
At the nonlinear level, combined with a local lower semicontinuity assumption on the Hessian, this yields local displacement convexity of the controlled energy around $\bar\mu$.
We also prove local exponential stabilization of the closed-loop equation.

\subsection{Formulation and spectral decomposition}
\label{sec:formulation}

The controlled gradient flow reads as
\begin{equation}
    \label{eq:gradient_flow_controlled}\partial_t\mu = \nabla\cdot\left(\mu\nabla\frac{\delta\F}{\delta\mu}(\mu)\right) + \sum_{j=1}^m u_j(t)\,\nabla\cdot(\mu\nabla\alpha_j), \qquad m \coloneqq \operatorname{dim} X_\delta,
\end{equation}
where $\alpha_j \in D(A)$ are prescribed shape functions and $u_j(t) \in \R$ are scalar controls.
We set the initial condition $\mu^{\rm in} \in \operatorname{dom}(\F)$.
Linearization around $(\bar\mu,0)$ yields, in $X$,
\begin{equation}
    \label{eq:linearized_equation_potential}
    \partial_t \xi = -A\xi - Bu(t), \qquad Bu \coloneqq \sum_{j=1}^m u_j\alpha_j,
\end{equation}
with initial condition $\xi(0) = \xi_0 = \rieszmap[\bar\mu]^{-1}(\mu^{\rm in} - \bar\mu) \in X$.
Since $A$ is self-adjoint and bounded from below by Theorem~\ref{thm:hessian_form}, $-A$ generates the $C_0$-semigroup $(e^{-tA})_{t \ge 0}$ on $X$.
Thus, for $\xi_0 \in X$ and admissible controls $u$, the linearized equation~\eqref{eq:linearized_equation_potential} is understood in the semigroup sense.
The corresponding equation in $X'$ is $\partial_t\nu = L\nu - \rieszmap[\bar\mu] Bu(t)$.

We verify $\delta$-stabilizability for the system $\partial_t\xi = -A\xi - Bu$ using the infinite-dimensional Hautus condition, see, e.g., \cite[Sec.~5.2]{CurtainZwart1995}.
This requires
\[
\ker(\lambda I + A) \cap \ker(B^*) = \{0\} \qquad \text{for all } \lambda \in \C \text{ with } \Re\lambda \ge -\delta.
\]
Since $A$ is self-adjoint with compact resolvent, the condition reduces to requiring that for all $k \in \Sigma_\delta$, if $\varphi \in \ker(A-\lambda_k I)\setminus\{0\}$, then $B^*\varphi = (\langle\varphi, \alpha_j\rangle_X)_{j=1}^m \neq 0$.
Since $B^* \varphi \neq 0$ if and only if $\varphi \notin (\operatorname{span}{\{\alpha_j\}_{j=1}^m})^\perp$, it is enough to select $\alpha_j \in D(A)$ for $j \in \Sigma_\delta$ so that 
\[
X_\delta \subseteq \operatorname{span}\{\alpha_1, \dots, \alpha_m\}.
\]
For concreteness, we henceforth fix the choice $\alpha_j = e_j$ for $j \in \Sigma_\delta$.
Then $B B^* = P_\delta$ and $B^*\xi = (\langle\xi,e_k\rangle_X)_{k \in \Sigma_\delta}$.
This discussion gives the following Hautus verification.

\begin{proposition}[Hautus $\delta$-stabilizability]
    \label{prop:hautus}
    Let $\delta > 0$ and choose shape functions $\alpha_1, \dots, \alpha_m \in D(A)$ to be an $X$-orthonormal basis $\{e_1,\dots,e_m\}$ of $X_\delta$.
    Then the pair $(-A,-B)$ satisfies the infinite-dimensional Hautus condition at rate $\delta$, namely,
    \[
    \ker(\lambda I+A) \cap \ker (-B^*) = \{0\} \qquad \forall \lambda \in \mathbb{C} \text{ with }\Re\lambda\ge -\delta.
    \]
    Consequently, $(-A,-B)$ is $\delta$-stabilizable.
\end{proposition}

\begin{proof}
    For the choice $\alpha_j = e_j$, where $\{e_1,\dots,e_m\}$ is an $X$-orthonormal basis of $X_\delta$, one has $\|Bu\|_X^2 = \sum_{j=1}^m |u_j|^2 = \|u\|_U^2$, so $B \in \mathcal{L}(\R^m, X)$.
    We now verify the Hautus condition. 
    Let $\lambda \in \mathbb{C}$ satisfy $\Re\lambda\ge -\delta$, and let
    \[
    \varphi \in \ker(\lambda I+A) \cap \ker (-B^*).
    \]
    Since $A$ is self-adjoint, $\ker(\lambda I + A) \neq \{0\}$ only if $\lambda = -\lambda_k$ for some eigenvalue $\lambda_k$ of $A$. 
    The condition $\Re\lambda\ge -\delta$ then gives $\lambda_k \le \delta$, and hence $\varphi \in X_\delta$. 
    On the other hand, $B^*\varphi = 0$ means $\langle \varphi, e_j\rangle_X = 0$, for $j=1,\dots,m$.
    Since $\{e_1, \dots, e_m\}$ is an orthonormal basis of $X_\delta$, this implies $\varphi \perp X_\delta$. 
    But $\varphi \in X_\delta$, so $\varphi = 0$.
    Thus, using that $-A$ generates a $C_0$-semigroup, \cite[Thm.~5.2.11]{CurtainZwart1995} implies that the pair $(-A, -B)$ is $\delta$-stabilizable.
\end{proof}

\begin{remark}
    Since $D(A) \subseteq \dotHk[2](\T^d) \subset \dotHk[1](\T^d)$, each $\alpha_j \in \dotHk[1](\T^d)$. 
    Hence, if $\mu \in L^\infty(\T^d)$, the expression $\rieszmap[\mu]\alpha_j = -\nabla \cdot (\mu \nabla \alpha_j)$ is well-defined as an element of $X'$.
    Therefore, on any local neighborhood where the densities remain bounded, the nonlinear control operator $\rieszmap[\mu] B u$ is well-defined as an element of $X'$.
\end{remark}

\subsection{Linear-quadratic optimal control and construction of the feedback law}
\label{sec:feedback_law}

With the shape functions $\alpha_j$ fixed as above, it remains to specify the time-dependent control coefficients $u_j$ in feedback form.
We do this by posing the following time-weighted linear-quadratic regulator (LQR) problem for the linearized equation~\eqref{eq:linearized_equation_potential}.
The state space is $X$, and the control space is $U \coloneqq \R^m$, equipped with its Euclidean inner product.
For a prescribed decay rate $\delta > 0$ and a bounded, self-adjoint, nonnegative operator $Q : X \to X$, define the cost functional
\begin{equation}
    \label{eq:cost_functional}
    J(\xi_0, u) \coloneqq \frac{1}{2}\int_0^\infty e^{2\delta t} \left(\langle Q\xi_t, \xi_t\rangle_X + \|u(t)\|_U^2\right)\,dt.
\end{equation}
The set of admissible controls is 
\[
u \in L^2_\delta(0,\infty; U) \coloneqq \left\{ u :(0, \infty) \to U \;\Big|\; \int_0^\infty e^{2\delta t}\|u(t)\|_U^2 \, dt < +\infty\right\}.
\]
By Proposition~\ref{prop:hautus}, the pair $(-A, -B)$ is $\delta$-stabilizable.
Setting $z_t \coloneqq e^{\delta t} \xi_t$ and $v(t) \coloneqq e^{\delta t}u(t)$ transforms the weighted problem into the unweighted LQR dynamics
\[
\partial_t z_t = -(A - \delta I)z_t - Bv(t).
\]
We assume that $Q$ is diagonal in $\{e_k\}_{k \ge 1}$, namely $Q e_k = q_k e_k$ with $q_k \ge 0$ for all $k \ge 1$ and $q_k > 0$ for $k \in \Sigma_\delta$, so that $Q$ penalizes all directions in $X_\delta$.

The algebraic Riccati equation associated with this LQR problem is 
\begin{equation}
    \label{eq:ARE-abstract}
    - \langle (A-\delta I)\varphi, \Pi\psi\rangle_X - \langle \Pi\varphi,(A-\delta I)\psi\rangle_X - \langle B^*\Pi\varphi,B^*\Pi\psi\rangle_U + \langle Q\varphi, \psi\rangle_X = 0
\end{equation}
for all $\varphi, \psi \in D(A)$.
Here $\Pi : X \to X$ is the Riccati operator.
The diagonal assumption on $Q$ also implies the detectability of $(-A + \delta I, Q^{1/2})$: for every $\lambda \in \mathbb{C}$ with $\Re\lambda \ge 0$,
\[
\ker(\lambda I + A - \delta I) \subseteq X_\delta, \qquad \ker(Q^{1/2}) \cap X_\delta = \{0\}.
\]
The next proposition gives the resulting Riccati operator and the optimal feedback control.

\begin{proposition}[Optimal control from the Riccati equation]
    \label{prop:are}
    Suppose that the control profiles are chosen as in Proposition~\ref{prop:hautus} and that $Q$ satisfies the diagonal condition.
    Then, there exists a unique self-adjoint, nonnegative $\Pi \in \mathcal{L}(X)$ solving \eqref{eq:ARE-abstract}.
    For each initial condition $\xi_0 \in X$, the optimal control along the corresponding semigroup solution is given by
    \[
    u(t) = B^*\Pi\xi_t
    \]
    and it minimizes $J(\xi_0, u)$ over $L^2_\delta(0,\infty; U)$.
    The corresponding modified operator is
    \[
    A_\Pi \coloneqq A + B B^*\Pi = A + P_\delta\Pi,
    \]
    and $-A_\Pi$ generates an exponentially stable semigroup with decay rate at least $\delta$.
\end{proposition}

\begin{proof}
    Proposition~\ref{prop:hautus} establishes $\delta$-stabilizability of $(-A,-B)$, while the preceding paragraph establishes detectability of $(-A+\delta I, Q^{1/2})$.
    Applying~\cite[Thm.~6.2.7]{CurtainZwart1995} to the shifted system for $(z,v)$ gives a unique self-adjoint, nonnegative solution $\Pi \in \mathcal{L}(X)$ of \eqref{eq:ARE-abstract} and the optimal feedback $v(t) = B^*\Pi z_t$.
    Since $z_t = e^{\delta t}\xi_t$ and $v(t) = e^{\delta t}u(t)$, this is equivalent to $u(t) = B^*\Pi\xi_t$ in the original variables.
    The cited theorem also gives exponential stability of $-(A-\delta I + BB^*\Pi)$ for the shifted system, which is equivalent to decay at rate at least $\delta$ for the semigroup generated by $-A_\Pi$.
\end{proof}

Under this diagonal choice of $Q$, the Riccati equation reduces to scalar equations along the eigenfunctions of $A$.

\begin{proposition}
\label{prop:are_diagonal}
    Suppose that $Q$ satisfies the diagonal condition stated above.
    Then the unique nonnegative solution $\Pi \in \mathcal{L}(X)$ of~\eqref{eq:ARE-abstract} is the diagonal operator $\Pi e_k = \pi_k e_k$, with
    \begin{equation}
        \label{eq:pi_k_definition}
        \pi_k =
        \begin{cases}
            -(\lambda_k-\delta) + \sqrt{(\lambda_k-\delta)^2 + q_k} & \text{for } k\in\Sigma_\delta, \\[1ex]
            \displaystyle \frac{q_k}{2(\lambda_k-\delta)} & \text{for } k\notin\Sigma_\delta.
        \end{cases}
    \end{equation}
    In particular, the modified operator $A_\Pi$ is self-adjoint with eigenvalues $\lambda_k + \pi_k$ for $k \in \Sigma_\delta$ and $\lambda_k$ for $k \notin \Sigma_\delta$.
\end{proposition}

\begin{proof}
    By Proposition~\ref{prop:are}, there exists a unique nonnegative solution $\Pi \in \mathcal{L}(X)$ to the Riccati equation~\eqref{eq:ARE-abstract}.
    To identify it, define a diagonal operator $\widetilde\Pi$ by $\widetilde\Pi e_k = \pi_k e_k$ with $\pi_k \ge 0$ defined in~\eqref{eq:pi_k_definition}.
    Since $Q$ is bounded and $A$ has compact resolvent, the sequence $(q_k)_{k \ge 1}$ is bounded and $\inf_{k \notin \Sigma_\delta} \lambda_k > \delta$.
    Hence $(\pi_k)_{k \ge 1}$ is bounded, so $\widetilde{\Pi} \in \mathcal{L}(X)$ is self-adjoint and nonnegative.
    Testing~\eqref{eq:ARE-abstract} with $\varphi = \psi = e_k$ yields
    \begin{align*}
        0 &= -\langle (A-\delta I)e_k,\widetilde\Pi e_k\rangle_X - \langle \widetilde\Pi e_k, (A-\delta I)e_k \rangle_X - \langle B^*\widetilde\Pi e_k,B^*\widetilde\Pi e_k \rangle_U + \langle Qe_k, e_k\rangle_X \\
        &= -(\lambda_k-\delta)\pi_k\langle e_k,e_k \rangle_X - (\lambda_k-\delta)\pi_k\langle e_k,e_k\rangle_X - \pi_k^2\|B^* e_k\|_U^2 + q_k\langle e_k,e_k\rangle_X \\
        &= -2(\lambda_k-\delta)\pi_k -\mathbf{1}_{\{k \in \Sigma_\delta\}}\pi_k^2
        + q_k.
    \end{align*}
    Here we used $\|B^*e_k\|_U^2 = \langle P_\delta e_k, e_k\rangle_X =\mathbf{1}_{\{k\in\Sigma_\delta\}}$.
    If $k \notin \Sigma_\delta$, this gives $\pi_k = \frac{q_k}{2(\lambda_k-\delta)}$, since $\lambda_k > \delta$.
    If $k \in \Sigma_\delta$, it gives $\pi_k^2 + 2(\lambda_k - \delta)\pi_k - q_k = 0$, whose unique nonnegative root is $\pi_k = -(\lambda_k - \delta) + \sqrt{(\lambda_k - \delta)^2 + q_k}$.
    Testing~\eqref{eq:ARE-abstract} with $\varphi = e_k$ and $\psi = e_\ell$, $k \neq \ell$, gives zero since all the operators are diagonal in $\{e_k\}_{k \ge 1}$.
    Therefore, equation~\eqref{eq:ARE-abstract} holds for all $\varphi, \psi$ in the linear span of the eigenfunctions. 
    Since this span is dense in $D(A)$ with respect to $\|\cdot\|_A$ and all terms in~\eqref{eq:ARE-abstract} are continuous in that norm, the identity extends to every $\varphi, \psi \in D(A)$. 
    Thus, $\widetilde\Pi$ is a nonnegative solution of the algebraic Riccati equation.
    By the uniqueness guaranteed by Proposition~\ref{prop:are}, we conclude that $\Pi = \widetilde\Pi$.
    Finally, $A_\Pi = A + P_\delta\Pi$ has eigenfunctions $e_k$ with eigenvalues $\lambda_k + \pi_k$ for $k \in \Sigma_\delta$, and $\lambda_k$ for $k \notin \Sigma_\delta$.
\end{proof}

\begin{corollary}
\label{cor:spectral_gap}
    Under the hypotheses of Proposition~\ref{prop:are_diagonal}, the modified operator $A_\Pi = A + P_\delta \Pi$ is self-adjoint with eigenvalues
    \[
    \lambda_k^\Pi =
    \begin{cases}
        \lambda_k + \pi_k = \sqrt{(\lambda_k-\delta)^2 + q_k} + \delta, & k \in \Sigma_\delta,\\
        \lambda_k, & k \notin \Sigma_\delta.
    \end{cases}
    \]
    In particular, we have $A_\Pi \ge \lambda_\Pi I$ for
    \[
    \lambda_\Pi \coloneqq \min\left(\min_{k \in \Sigma_\delta}\sqrt{(\lambda_k-\delta)^2 + q_k} + \delta, \min_{k\notin\Sigma_\delta} \lambda_k\right) > \delta.
    \]
\end{corollary}

\begin{proof}
    This follows from Proposition~\ref{prop:are_diagonal}, the positivity condition $q_k > 0$ for $k \in \Sigma_\delta$, and the fact that a self-adjoint diagonal operator is bounded below by the infimum of its eigenvalues.
\end{proof}

\subsection{Local convexification of the modified energy}
\label{sec:convexification}

We now show that $A_\Pi \ge \lambda_\Pi I$ yields a lower bound for the Wasserstein Hessian of a modified energy at $\bar\mu$.
For every $\mu \in \operatorname{dom}(\F)$ such that $\mu - \bar\mu \in X'$, define
\[
\F_{\rm cl}(\mu) \coloneqq \F(\mu) + \mathcal{Q}(\mu), \quad \mathcal{Q}(\mu) \coloneqq \frac{1}{2}\left\langle P_\delta \Pi \rieszmap[\bar\mu]^{-1}(\mu - \bar\mu), \rieszmap[\bar\mu]^{-1}(\mu - \bar\mu)\right\rangle_X.
\]

\begin{proposition}[Modified energy and the controlled dynamics]
\label{prop:modified_energy_closed_loop}
    For every $\mu \in \operatorname{dom}(\F)$ such that $\mu - \bar\mu \in X'$ and every $h \in X'$,
    \[
    D\mathcal{Q}(\mu)[h] = \left\langle h, P_\delta\Pi\rieszmap[\bar\mu]^{-1}(\mu-\bar\mu)\right\rangle_{X',X}.
    \]
    Consequently, with the feedback law
    \[
    u(t) = B^* \Pi \rieszmap[\bar\mu]^{-1}(\mu_t-\bar\mu),
    \]
    the controlled equation~\eqref{eq:gradient_flow_controlled} is the Wasserstein gradient flow of $\F_{\rm cl}$
    \[
    \partial_t\mu = \nabla\cdot\left( \mu \nabla \frac{\delta\F_{\rm cl}}{\delta\mu}(\mu)\right).
    \]
\end{proposition}

\begin{proof}
    By Proposition~\ref{prop:are_diagonal}, both $\Pi$ and $P_\delta$ are diagonal in the basis $\{e_k\}_{k \ge 1}$, so they commute, and $P_\delta \Pi$ is bounded, self-adjoint, and nonnegative on $X$.
    Hence, for $h \in X'$,
    \[
    D\mathcal{Q}(\mu)[h] = \left\langle \rieszmap[\bar\mu]^{-1}h, P_\delta \Pi \rieszmap[\bar\mu]^{-1}(\mu-\bar\mu) \right\rangle_X = \left\langle h,
    P_\delta \Pi \rieszmap[\bar\mu]^{-1}(\mu-\bar\mu)\right\rangle_{X',X},
    \]
    implying that $\delta\mathcal{Q}/\delta\mu = P_\delta \Pi \rieszmap[\bar\mu]^{-1}(\mu-\bar\mu)$.
    Since the controlled part of~\eqref{eq:gradient_flow_controlled} is $\nabla \cdot (\mu \nabla (Bu))$, we use the feedback form of $u$ from Proposition~\ref{prop:are} to compute
    \[
    \nabla \cdot (\mu \nabla (Bu(t))) = \nabla \cdot \left(\mu \nabla \left(P_\delta \Pi \rieszmap[\bar\mu]^{-1}(\mu_t-\bar\mu)\right)\right) = \nabla \cdot \left(\mu \nabla \frac{\delta\mathcal{Q}}{\delta\mu}(\mu)\right).
    \]
    It follows that the controlled equation is the Wasserstein gradient flow of $\F_{\rm cl}$.
\end{proof}

The previous proposition identifies the first variation of $\mathcal{Q}$ with the feedback potential.
We now compute the second variation at $\bar\mu$ and recover the modified Hessian form.

\begin{proposition}[Hessian of the modified energy]
\label{prop:closed_loop_hessian}
    Assume the hypotheses of Theorem~\ref{thm:hessian_form} and Proposition~\ref{prop:are_diagonal}.
    Define, for $\varphi, \psi \in D(a)$,
    \[
    a_\Pi(\varphi, \psi) \coloneqq a(\varphi, \psi) + \langle P_\delta \Pi \varphi, \psi\rangle_X.
    \]
    Then, for all $\varphi, \psi \in C^{\infty}(\T^d)$,
    \[
    \mathrm{Hess}_{\Wtwo} \F_{\rm cl}(\bar\mu)(\nabla \varphi, \nabla \psi) = a_\Pi(\varphi, \psi).
    \]
    Moreover, the bilinear form defined by the left-hand side admits the unique continuous extension $a_\Pi$ to $D(a)\times D(a)$.
\end{proposition}

\begin{proof}
    By Theorem~\ref{thm:hessian_form}, for $\varphi, \psi \in C^\infty(\T^d)$,
    \[
    \mathrm{Hess}_{\Wtwo}\F(\bar\mu)(\nabla\varphi, \nabla\psi) = a(\varphi,\psi).
    \]
    Let $\varphi \in C^\infty(\T^d)$.
    By Lemma~\ref{lem:local_displacement_interpolation}, $s \mapsto \Xi(s \varphi)$ is a local $\Wtwo$-geodesic through $\bar\mu$ with initial velocity $\nabla \varphi$.
    Since $\Xi(s\varphi) - \bar\mu = s\rieszmap[\bar\mu]\varphi + o(s)$ in $X'$, we have $\rieszmap[\bar\mu]^{-1}(\Xi(s\varphi) - \bar\mu) = s\varphi + o(s)$ in $X$.
    Therefore
    \[
    \frac{d^2}{ds^2} \mathcal{Q}(\Xi(s\varphi)) \Big|_{s=0} = \langle P_\delta \Pi \varphi, \varphi\rangle_X.
    \]
    Polarization gives $\mathrm{Hess}_{\Wtwo} \mathcal{Q}(\bar\mu)(\nabla \varphi, \nabla \psi) = \langle P_\delta \Pi \varphi, \psi\rangle_X$ for $\varphi, \psi \in C^{\infty}(\T^d)$.
    The linearity of the Hessian gives the formula for $\F_{\rm cl}$.
    Finally, since $P_\delta \Pi$ is bounded on $X$, the perturbation $(\varphi,\psi) \mapsto \langle P_\delta \Pi\varphi, \psi\rangle_X$ is continuous with respect to the form norm on $D(a)$.
    Hence the identity extends uniquely to $D(a) \times D(a)$.
\end{proof}

Since $P_\delta \Pi$ is bounded and self-adjoint on $X$, the form $a_\Pi$ is a bounded symmetric perturbation of the closed form $a$.
It is therefore represented by the self-adjoint operator $A_\Pi = A + P_\delta\Pi$ in $X$.
Thus, for $\varphi \in D(A)$ and $\psi \in D(a)$,
\[
a_\Pi(\varphi, \psi) = \langle A_\Pi \varphi, \psi\rangle_X = -\langle L_\Pi \rieszmap[\bar\mu] \varphi, \rieszmap[\bar\mu] \psi \rangle_{X'}, \qquad L_\Pi \coloneqq -\rieszmap[\bar\mu] A_\Pi \rieszmap[\bar\mu]^{-1}.
\]
Under Assumption~\ref{ass:H3}, Proposition~\ref{prop:linearized_operator_identification} and the preceding identity show that the linearization at $\bar\mu$ of the gradient flow of $\F_{\rm cl}$ is generated by $L_\Pi$.
By Corollary~\ref{cor:spectral_gap}, $A_\Pi \ge \lambda_\Pi I$ with $\lambda_\Pi > \delta$.
Consequently, for $\varphi \in D(a)$,
\[
a_\Pi(\varphi, \varphi) \ge \lambda_\Pi\|\varphi\|_X^2.
\]

We now pass from the strict Hessian lower bound at $\bar\mu$ to a local displacement convexity statement.
We call a set $\mathcal{V} \subseteq \operatorname{dom}(\F)$ an \emph{admissible class at $\bar\mu$} if $\mathrm{Hess}_{\Wtwo}\E(\mu)$ is well-defined on $\mathscr{D}$ for every $\mu \in \mathcal{V}$, and if there exists $r_0 \in (0,1)$ such that every probability density $\mu$ with $\|\mu/\bar\mu - 1\|_{L^\infty} \le r_0$ belongs to $\mathcal{V}$.
From now on we fix such a class $\mathcal{V}$, together with a corresponding $r_0 \in (0,1)$, and, for $r \in (0,1)$, we set
\[
\mathcal{V}_r \coloneqq \left\{\mu \in \mathcal{V} : \left\|\frac{\mu}{\bar\mu} - 1\right\|_{L^\infty} \le r \right\}.
\]
In these neighborhoods, we impose the following additional assumptions, which are used only for the local convexification step.

\begin{enumerate}[label=(H\arabic*),ref=(H\arabic*), start=4]
    \item\label{ass:H4} For every $\varepsilon > 0$, there exists $r_\varepsilon \in (0, r_0]$ such that, for every $\mu \in \mathcal{V}_{r_\varepsilon}$ and every $\nabla \phi \in \mathscr{D}$,
    \[
    \mathrm{Hess}_{\Wtwo} \E(\mu)(\nabla\phi, \nabla\phi) \ge \mathrm{Hess}_{\Wtwo} \E(\bar\mu)(\nabla\phi, \nabla\phi) - \varepsilon\|\phi\|_X^2.
    \]
    \item\label{ass:H5} For every $k \in \Sigma_\delta$, $e_k \in W^{2,\infty}(\T^d)$.
\end{enumerate}

Using Assumption~\ref{ass:H5}, we are able to provide a local lower semicontinuity result for the Wasserstein Hessian of $\mathcal{Q}$.

\begin{lemma}
\label{lem:finite_rank_hessian_continuity}
    Assume~\ref{ass:H5}.
    Then there exists $C_Q > 0$ such that, for every $r \in (0,1)$, every $\mu \in \mathcal{V}_r$, and every $\phi \in W^{2,\infty}(\T^d)$,
    \[
    \mathrm{Hess}_{\Wtwo} \mathcal{Q}(\mu)(\nabla\phi, \nabla\phi) \ge \langle P_\delta\Pi\phi,\phi\rangle_X - C_Q r \|\phi\|_X^2.
    \]
\end{lemma}

\begin{proof}
    We have $P_\delta\Pi \phi = \sum_{\ell=1}^m \pi_\ell \langle \phi,e_\ell\rangle_X e_\ell$.
    Thus
    \[
    \mathcal{Q}(\mu) = \frac{1}{2} \sum_{\ell=1}^m \pi_\ell\left(\int_{\T^d} e_\ell \, d(\mu - \bar\mu)\right)^2.
    \]
    Fix $r \in (0,1)$, $\mu \in \mathcal{V}_r$, and $\phi \in W^{2,\infty}(\T^d)$.
    Write $\mu = (1+\rho)\bar\mu$, with $\|\rho\|_{L^\infty}\le r$.
    By Lemma~\ref{lem:local_displacement_interpolation}, $s \mapsto (T_{s\phi})_\sharp\mu$ is a local $\Wtwo$-geodesic through $\mu$ with initial velocity $\nabla\phi$.
    Set
    \[
    m_\ell(s) \coloneqq \int_{\T^d} e_\ell \, d((T_{s\phi})_\sharp\mu - \bar\mu).
    \]
    The chain rule gives $m_\ell'(0) = \int_{\T^d} \nabla e_\ell \cdot \nabla\phi \, d\mu$ and $m_\ell''(0) = \int_{\T^d} (\nabla\phi)^\top \nabla^2 e_\ell \nabla\phi \, d\mu$.
    Therefore,
    \[
    \mathrm{Hess}_{\Wtwo} \mathcal{Q}(\mu)(\nabla\phi,\nabla\phi) = \sum_{\ell=1}^m \pi_\ell ((m_\ell'(0))^2 + m_\ell(0) m_\ell''(0)),
    \]
    $(m_\ell'(0))^2 \ge \langle e_\ell,\phi\rangle_X^2 - C_Q r\|\phi\|_X^2$, and $m_\ell(0)m_\ell''(0)\ge - C_Q r\|\phi\|_X^2$ for some $C_Q > 0$ because
    \[
    |m_\ell(0)| \le r\int_{\T^d} |e_\ell| \, d\bar\mu, \;\; |m_\ell'(0) - \langle e_\ell,\phi\rangle_X| \le r\|\phi\|_X, \;\; |m_\ell''(0)| \le (1+r)\|\nabla^2 e_\ell\|_{L^\infty}\|\phi\|_X^2.
    \]
    At $\mu = \bar\mu$, one has $m_\ell(0) = 0$ and $m_\ell'(0)= \langle e_\ell, \phi\rangle_X$, so $\mathrm{Hess}_{\Wtwo} \mathcal{Q}(\bar\mu)(\nabla\phi,\nabla\phi) = \langle P_\delta\Pi\phi,\phi\rangle_X$ and
    \[
    \mathrm{Hess}_{\Wtwo} \mathcal{Q}(\mu)(\nabla\phi,\nabla\phi) \ge \langle P_\delta\Pi\phi,\phi\rangle_X - C_Q r\|\phi\|_X^2. \qedhere
    \]
\end{proof}

Combining this lemma with Assumption~\ref{ass:H4} gives our desired result.

\begin{proposition}
\label{prop:persistence_hessian_lower_bound}
    Assume the hypotheses of Proposition~\ref{prop:are_diagonal} and that Assumptions~\ref{ass:H1}--\ref{ass:H2} and~\ref{ass:H4}--\ref{ass:H5} hold.
    Then, for every $\hat{\lambda} \in (0, \lambda_\Pi)$, there exists $r_{\hat{\lambda}} \in (0, r_0]$ such that for every $\mu \in \mathcal{V}_{r_{\hat{\lambda}}}$ and every $\phi \in W^{2,\infty}(\T^d)$,
    \[
    \mathrm{Hess}_{\Wtwo} \F_{\rm cl}(\mu)(\nabla\phi,\nabla\phi) \ge \hat{\lambda} \int_{\T^d} |\nabla\phi|^2\,d\mu.
    \]
\end{proposition}

\begin{proof}
    Set $q_\mu(\phi,\phi) \coloneqq \int_{\T^d}|\nabla^2\phi|^2 \, d\mu$.
    By Proposition~\ref{prop:entropy_hessian},
    \[
    \mathrm{Hess}_{\Wtwo} \Ent(\mu)(\nabla\phi, \nabla\phi) = q_\mu(\phi,\phi).
    \]
    Fix $\hat{\lambda} \in (0,\lambda_\Pi)$. 
    Choose $\eta > 0$ such that $\lambda_\Pi - \eta > \hat{\lambda}$, and let $r_\eta \in (0,r_0]$ be given by Assumption~\ref{ass:H4}.
    If $0 < r < r_\eta$ and $\mu \in \mathcal{V}_r$, then $\mu = (1+\rho)\bar\mu$ with $\|\rho\|_{L^\infty} \le r$, and hence $q_\mu \ge (1-r) q_{\bar\mu} = (1-r)q_0$.
    Using Assumption~\ref{ass:H4} and Lemma~\ref{lem:finite_rank_hessian_continuity}, we obtain
    \[
    \mathrm{Hess}_{\Wtwo} \F_{\rm cl}(\mu)(\nabla\phi,\nabla\phi) \ge a_\Pi(\phi,\phi) - (\eta + C_Q r)\|\phi\|_X^2 - \sigma r q_0(\phi,\phi).
    \]
    Let $M \coloneqq C_E + \|P_\delta\Pi\|_{\mathcal{L}(X)}$.
    Since $a_\Pi(\phi,\phi) = H_\E(\phi,\phi) + \sigma q_0(\phi,\phi) + \langle P_\delta\Pi \phi, \phi\rangle_X$ and $|H_\E(\phi,\phi) + \langle P_\delta\Pi\phi,\phi\rangle_X| \le M\|\phi\|_X^2$, we have
    \[
    \sigma q_0(\phi,\phi) \le a_\Pi(\phi,\phi) + M\|\phi\|_X^2.
    \]
    Together with $a_\Pi \ge \lambda_\Pi I$, this gives
    \[
    \mathrm{Hess}_{\Wtwo}\F_{\rm cl}(\mu)(\nabla\phi,\nabla\phi) \ge \left(1 - \frac{\eta}{\lambda_\Pi} - r\left(1 + \frac{M + C_Q}{\lambda_\Pi}\right)\right) a_\Pi(\phi,\phi).
    \]
    Choose $r_{\hat{\lambda}} \in (0,r_\eta)$ sufficiently small so that
    \[
    \frac{\lambda_\Pi}{1 + r_{\hat{\lambda}}} \left(1 - \frac{\eta}{\lambda_\Pi} - r_{\hat{\lambda}} \left(1 + \frac{M + C_Q}{\lambda_\Pi}\right)\right) \ge \hat{\lambda}.
    \]
    This is possible because, for the chosen $\eta$, the left-hand side converges to $\lambda_\Pi - \eta > \hat{\lambda}$ as $r_{\hat{\lambda}} \to 0$.
    Using again $a_\Pi\ge\lambda_\Pi I$ and $\mu \le (1 + r_{\hat{\lambda}})\bar\mu$, we obtain that for every $\mu \in \mathcal{V}_{r_{\hat{\lambda}}}$,
    \[
    \mathrm{Hess}_{\Wtwo}\F_{\rm cl}(\mu)(\nabla\phi,\nabla\phi) \ge \hat{\lambda} \int_{\T^d} |\nabla\phi|^2\,d\mu. \qedhere
    \]
\end{proof}

Given $\hat{\lambda} \in (0, \lambda_\Pi)$, we write $r_\star \coloneqq r_{\hat{\lambda}}$ for the radius provided by Proposition~\ref{prop:persistence_hessian_lower_bound}.
We now pass from the pointwise Hessian estimate on $\mathcal{V}_{r_\star}$ to a local displacement convexity statement, using the regularity of the optimal map between densities that are close to $\bar\mu$ in $C^{0,\alpha}(\T^d)$ (Lemma~\ref{lem:brenier_stability}).
For $\alpha \in (0,1)$ and $\varepsilon > 0$, set
\[
\Nbhd^\alpha_\varepsilon \coloneqq \left\{\mu \in \Pcal_2(\T^d) : \|\mu - \bar\mu\|_{C^{0,\alpha}} \le \varepsilon \right\}.
\]

\begin{theorem}[Local displacement convexity]
    \label{thm:local_strong_displacement_convexity}
    Assume the hypotheses of Proposition~\ref{prop:persistence_hessian_lower_bound} and fix $\alpha \in (0,1)$ and $\hat{\lambda} \in (0,\lambda_\Pi)$. 
    Then there exists $\varepsilon > 0$ such that $\Nbhd^\alpha_\varepsilon \subseteq \mathcal{V}_{r_\star}$ and $\F_{\rm cl}$ is $\hat{\lambda}$-displacement convex on $\Nbhd^\alpha_\varepsilon$, meaning that for every $\mu_0, \mu_1 \in \Nbhd^\alpha_\varepsilon$, the constant speed $\Wtwo$-geodesic $(\mu_s)_{s \in [0,1]}$ connecting them satisfies
    \[
    \F_{\rm cl}(\mu_s) \le (1-s)\F_{\rm cl}(\mu_0) + s \F_{\rm cl}(\mu_1) - \hat{\lambda}\frac{s(1-s)}{2} \Wtwo^2(\mu_0,\mu_1)
    \]
    for every $s \in [0,1]$.
\end{theorem}

\begin{proof}
    Fix a bounded open interval $J \supset [0,1]$, let $\eta \in (0, 1/2)$ be fixed below, let $\varepsilon = \varepsilon(\eta) > 0$ be given by Lemma~\ref{lem:brenier_stability}, and let $\mu_0, \mu_1 \in \Nbhd^\alpha_{\varepsilon}$.
    By that lemma, the optimal transport map from $\mu_0$ to $\mu_1$ is $T_\psi$, with $\psi \in C^{2,\alpha}(\T^d) \subset W^{2,\infty}(\T^d)$ and $\|\nabla\psi\|_{C^1} \le \eta$.
    Since $J$ is bounded, for $\eta$ small enough one has $|s| \eta < 1/2$ for every $s \in J$, hence $|s| \|\nabla^2\psi\|_{L^\infty} < 1$ and $|s| \|\nabla\psi\|_{L^\infty} < 1/2$.
    By Lemma~\ref{lem:local_displacement_interpolation}, $T_{s\psi}$ is therefore an optimal transport map from $\mu_0$ to $\mu_s \coloneqq (T_{s\psi})_\sharp\mu_0$ for every $s \in J$, and the restriction of $s \mapsto \mu_s$ to $[0,1]$ is the displacement interpolation between $\mu_0$ and $\mu_1$, hence the constant speed $\Wtwo$-geodesic connecting them.
    By the change of variables formula, $\mu_s(T_{s\psi}(x)) \det(I + s\nabla^2\psi(x)) = \mu_0(x)$ for a.e.\ $x$, and therefore
    \[
    \frac{\mu_s(T_{s\psi}(x))}{\bar\mu(T_{s\psi}(x))} = \frac{\mu_0(x)}{\bar\mu(x)} \cdot \frac{\bar\mu(x)}{\bar\mu(T_{s\psi}(x))} \cdot \frac{1}{\det(I + s\nabla^2\psi(x))}.
    \]
    Hence, for a.e.\ $x \in \T^d$,
    \[
    \begin{split}
        \left|\frac{\mu_s(T_{s\psi}(x))}{\bar\mu(T_{s\psi}(x))} - 1\right| &\le \left|\frac{\mu_0(x)}{\bar\mu(x)} - 1\right| \left|\frac{\bar\mu(x)}{\bar\mu(T_{s\psi}(x))}\right| \left|\frac{1}{\det(I + s\nabla^2\psi(x))}\right| \\
        &\quad + \left|\frac{\bar\mu(x)}{\bar\mu(T_{s\psi}(x))} - 1\right| \left|\frac{1}{\det(I + s\nabla^2\psi(x))}\right| + \left|\frac{1}{\det(I + s\nabla^2\psi(x))} - 1\right| \\
        &\le \frac{1}{(1 - |s|\eta)^d} \left(\frac{C_0}{c_0^2}\,\varepsilon + \frac{|s|\,\|\nabla \bar\mu\|_{L^\infty} \eta}{c_0} + 2C_d\,|s|\,\eta \right),
    \end{split}
    \]
    where $C_d > 0$ depends only on $d$, and where we used Lemma~\ref{lem:mu_bar_bounds}(ii)--(iii), the bounds $\|\nabla\psi\|_{C^1} \le \eta$ and $|s|\eta < 1/2$ for every $s \in J$, and $\|\mu_0 - \bar\mu\|_{L^\infty} \le \|\mu_0 - \bar\mu\|_{C^{0,\alpha}} \le \varepsilon$.
    The right-hand side is bounded by $C(\varepsilon + \eta)$ with $C > 0$ independent of $\varepsilon$ and $\eta$.
    Choosing first $\eta$ and then $\varepsilon$ small enough, we obtain $|\mu_s(T_{s\psi}(x)) / \bar\mu(T_{s\psi}(x)) - 1| \le r_\star$ for a.e.\ $x \in \T^d$ and every $s \in J$.
    Since $T_{s\psi}$ is a Lipschitz bijection of $\T^d$, it maps $\mathcal{L}^d$-null sets to $\mathcal{L}^d$-null sets, and hence $\|\mu_s/\bar\mu - 1\|_{L^\infty} \le r_\star$.
    Since $r_\star \le r_0$ and $\mathcal{V}$ is an admissible class at $\bar\mu$, this gives $\mu_s \in \mathcal{V}_{r_\star}$ for every $s \in J$.

    Fix $s \in J$. 
    Let $\varphi_s$ be the velocity potential of Lemma~\ref{lem:velocity_potential}.
    By that lemma, $\nabla \varphi_s \in \mathscr{D}$, the curve $h \mapsto \mu_{s+h}$ is a local $\Wtwo$-geodesic through $\mu_s$ with initial velocity $\nabla\varphi_s$, and
    \[
    \int_{\T^d} |\nabla\varphi_s|^2 \, d\mu_s = \int_{\T^d} |\nabla\psi|^2 \, d\mu_0 = \Wtwo^2(\mu_0,\mu_1),
    \]
    where the last equality follows from the optimality of $T_\psi$ together with the identity $|T_\psi(x) - x| = d_{\T^d}(x,T_\psi(x))$ of Lemma~\ref{lem:brenier_stability}.
    Define $g : J \to \R$ by $g(s) \coloneqq \F_{\rm cl}(\mu_s)$.
    Since $\mu_s \in \mathcal{V}$, the Wasserstein Hessian of $\E$ at $\mu_s$ is well defined on $\mathscr{D}$.
    Together with Proposition~\ref{prop:entropy_hessian} for the entropy and Lemma~\ref{lem:finite_rank_hessian_continuity} for $\mathcal{Q}$, and since $\nabla\varphi_s \in \mathscr{D}$, this gives that $g$ is twice differentiable at every $s \in J$, with $g''(s) = \mathrm{Hess}_{\Wtwo} \F_{\rm cl}(\mu_s)(\nabla\varphi_s, \nabla\varphi_s)$.
    Therefore, Proposition~\ref{prop:persistence_hessian_lower_bound} gives
    \[
    g''(s) \ge \hat{\lambda} \int_{\T^d} |\nabla\varphi_s|^2 \, d\mu_s = \hat{\lambda} \Wtwo^2(\mu_0,\mu_1)
    \]
    for every $s \in J$.
    Hence $s \mapsto g(s) - \frac{\hat{\lambda}}{2}\Wtwo^2(\mu_0,\mu_1)s^2$ has non-decreasing derivative on $J$, and is therefore convex on $J \supset [0,1]$.
    Evaluating at $s \in [0,1]$ gives the desired estimate.
\end{proof}

\begin{remark}
    For $p > d$, one has $W^{2,p}(\T^d) \hookrightarrow C^1(\T^d) \subset C^{0,\alpha}(\T^d)$ for every $\alpha \in (0,1)$.
    Hence, the set $\Nbhd^\alpha_\varepsilon$ contains every probability density of the form $\bar\mu + \nu$ with $\nu \in W^{2,p}(\T^d)$ and $\|\nu\|_{W^{2,p}}$ sufficiently small.
    If one assumes in addition that $\bar\mu \in W^{2,\infty}(\T^d)$, which holds in the model class of Section~\ref{sec:mckean-vlasov-models}, then \cite[Lemmas A.19--A.20]{seo2026local} provide a quantitative counterpart of Lemma~\ref{lem:brenier_stability} on $W^{2,p}$-balls, with
    \[
    \|\nabla\psi\|_{W^{2,p}} \le C\|\mu_1 - \mu_0\|_{W^{1,p}}.
    \]
\end{remark}

\subsection{Local exponential stabilization}
\label{sec:local-stab}

By Proposition~\ref{prop:modified_energy_closed_loop}, the closed-loop equation is the Wasserstein gradient flow of $\F_{\rm cl} = \F + \mathcal{Q}$.
Moreover, Assumptions~\ref{ass:H1} and~\ref{ass:H5} ensure that the feedback control term is well defined and $X'$-valued for every $\nu \in L^2(\T^d)$, hence on the solution space $\mathcal{Y}_T$ introduced below.
The stabilization proof uses the spectral gap of $A_\Pi$ through a maximal-regularity estimate for $L_\Pi$, and then absorbs the nonlinear remainder as in~\cite[Section~3.4]{KaliseMoschenPavliotis2025}.
For a solution $\mu_t$, set $\nu_t \coloneqq \mu_t-\bar\mu$ and $\xi_t \coloneqq \rieszmap[\bar\mu]^{-1}\nu_t$.
The equation becomes
\begin{equation}
    \label{eq:nonlinear_equation_x_prime}
    \partial_t \nu_t = L_\Pi\nu_t + \mathfrak{R}(\nu_t) + \mathfrak{B}_{\rm fb}(\nu_t),
\end{equation}
where
\[
\mathfrak{R}(\nu) \coloneqq \nabla\cdot\left((\bar\mu + \nu)\nabla\frac{\delta\F}{\delta\mu}(\bar\mu + \nu)\right) - L\nu, \qquad \mathfrak{B}_{\rm fb}(\nu) \coloneqq \nabla \cdot \left(\nu\nabla\left(P_\delta \Pi \rieszmap[\bar\mu]^{-1}\nu\right)\right).
\]

For \(T \in (0,\infty]\), set
\[
\mathcal{Y}_T \coloneqq L^\infty(0,T; L^2(\T^d)) \cap L^2(0,T; D(L)) \cap H^1(0,T; X'),
\]
with norm
\[
\|\nu\|_{\mathcal{Y}_T} \coloneqq \|\nu\|_{L^\infty(0, T; L^2(\T^d))} + \|\nu\|_{L^2(0, T; D(L))} + \|\partial_t\nu\|_{L^2(0, T; X')}.
\]
Here $D(L)$ is endowed with the norm $\|\cdot\|_{L} \coloneqq \|\cdot\|_{X'} + \|L\cdot\|_{X'}$.
By a solution of~\eqref{eq:nonlinear_equation_x_prime} on $[0,T]$ we mean $\nu \in \mathcal{Y}_T$ satisfying~\eqref{eq:nonlinear_equation_x_prime} in $L^2(0,T;X')$.

We claim that $\mathcal{Y}_T$ embeds continuously into
$C([0, T]; L^2(\T^d))$. 
Indeed, if $\nu \in \mathcal{Y}_T$ and $z = \rieszmap[\bar\mu]^{-1} \nu$, then $z \in L^2(0, T; D(A))$ and $\partial_t z \in L^2(0, T; X)$, which implies $\partial_t z \in L^2(0, T; D(A)')$, where $D(A)'$ denotes the dual of $D(A)$ with respect to the pivot space $D(a)$, since $X$ is continuously embedded in $D(A)'$.
Given that $D(A) \hookrightarrow D((A + \beta I)^{1/2}) = D(a) \hookrightarrow D(A)'$ for $\beta > -\inf \sigma(A)$, the Lions--Magenes lemma implies $z \in C([0, T]; D(a)) = C([0, T]; \dotHk[2](\T^d))$ (see~\cite[Ch. 3, Lem. 1.2]{temam1977navier}). 
It follows that $\nu \in C([0, T]; L^2(\T^d))$.

We impose one additional assumption for the local stabilization argument.

\begin{enumerate}[label=(H\arabic*),ref=(H\arabic*),start=6]
    \item\label{ass:H6} There exist $r > 0$ and $C_{\rm nl} > 0$ such that for every probability density $\mu_0 = \bar\mu + \nu_0$ with $\nu_0 \in L^2(\T^d)$ and $\|\nu_0\|_{L^2} < r$, there exist $T > 0$ and a solution $\nu \in \mathcal{Y}_T$ of~\eqref{eq:nonlinear_equation_x_prime}, with $\nu(0) = \nu_0$, such that $\bar\mu + \nu_t$ is a probability density for every $t \in [0,T]$.
    In addition, for all $\nu, \zeta \in D(L)$ such that
    $\bar\mu + \nu$ and $\bar\mu + \zeta$ are probability densities, with $\|\nu\|_{L^2} < r$, $\|\zeta\|_{L^2} < r$, the uncontrolled nonlinear remainder satisfies
    \[
    \|\mathfrak{R}(\nu) - \mathfrak{R}(\zeta)\|_{X'} \le C_{\rm nl} \left(\|\nu\|_{L^2} + \|\zeta\|_{L^2} \right) \|\nu - \zeta\|_{L^2}.
    \]
\end{enumerate}

\begin{theorem}[Local exponential stabilization]
\label{thm:local_stab}
    Assume~\ref{ass:H1}--\ref{ass:H3}, the hypotheses of Proposition~\ref{prop:are_diagonal}, and Assumptions~\ref{ass:H5}--\ref{ass:H6}.
    Then, for every $\gamma \in (0,\lambda_\Pi)$, there exist $\varepsilon > 0$ and $C_\gamma > 0$ such that, for every probability density $\mu_0 = \bar\mu + \nu_0$ satisfying $\nu_0 \in L^2(\T^d)$ and $\|\nu_0\|_{L^2} \le \varepsilon$, the closed-loop equation admits a unique global solution, which satisfies
    \[
    e^{\gamma(\cdot)}\nu \in \mathcal{Y}_\infty, \qquad \|e^{\gamma(\cdot)}\nu\|_{\mathcal{Y}_\infty} \le C_\gamma \|\nu_0\|_{L^2}.
    \]
    Moreover, for every $t \ge 0$,
    \[
    \|\nu_t\|_{X'} \le C_\gamma e^{-\gamma t}\|\nu_0\|_{L^2}.
    \]
    Equivalently, if $\xi_t = \rieszmap[\bar\mu]^{-1}\nu_t$, then for every $t \ge 0$, $\|\xi_t\|_X \le C_\gamma e^{-\gamma t}\|\nu_0\|_{L^2}$.
\end{theorem}

\begin{proof}
    Fix $\gamma \in (0,\lambda_\Pi)$ and define $\mathfrak{R}_\Pi(\nu) \coloneqq \mathfrak{R}(\nu) + \mathfrak{B}_{\rm fb}(\nu)$ and $\mathfrak{R}_{\Pi,\gamma}(w)(t) \coloneqq e^{\gamma t}\mathfrak{R}_\Pi(e^{-\gamma t}w_t)$.
    Let $C_{\rm lin,\gamma}$ and $C_{\rm rem,\gamma}$ denote the constants in Proposition~\ref{prop:weighted_linear_density} and Lemma~\ref{lem:weighted_remainder}, respectively.
    Choose
    \[
    \rho \coloneqq \min\left\{\frac{r}{2},\frac{1}{4C_{\rm lin,\gamma}C_{\rm rem,\gamma}}\right\}
    \]
    and then choose $\varepsilon > 0$ so small that $\varepsilon < \min\{r,\rho\}$ and $C_{\rm lin,\gamma} \varepsilon \le \rho / 4$.

    Let $\mu_0 = \bar\mu + \nu_0$ satisfy the hypotheses of the theorem.
    By Assumption~\ref{ass:H6} and Lemma~\ref{lem:maximal_continuation}, equation~\eqref{eq:nonlinear_equation_x_prime} admits a unique maximal solution $\nu$ on $[0,T_{\max})$, with $\nu \in \mathcal{Y}_T$ for every $T < T_{\max}$.
    Set $w_t\coloneqq e^{\gamma t}\nu_t$.
    Then, on every interval $[0,T]$ with $T < T_{\max}$, $w$ solves
    \[
    \partial_t w_t - (L_\Pi + \gamma I)w_t = \mathfrak{R}_{\Pi,\gamma}(w)(t),
    \]
    with $w_0 = \nu_0$.
    By Proposition~\ref{prop:weighted_linear_density},
    \[
    \|w\|_{\mathcal{Y}_T} \le C_{\rm lin,\gamma}\left(\|\nu_0\|_{L^2} + \|\mathfrak{R}_{\Pi,\gamma}(w)\|_{L^2(0,T;X')}\right)
    \]
    and, whenever $\|w\|_{L^\infty(0,T;L^2)} < r$, Lemma~\ref{lem:weighted_remainder} gives $\|\mathfrak{R}_{\Pi,\gamma}(w)\|_{L^2(0,T;X')} \le C_{\rm rem,\gamma}\|w\|_{\mathcal{Y}_T}^2$.
    We claim that $\|w\|_{\mathcal{Y}_T} \le \rho$ for every $T < T_{\max}$.
    Indeed, on any interval on which $\|w\|_{\mathcal{Y}_T} \le \rho$, the choice of $\rho$ gives $\|w\|_{L^\infty(0,T;L^2)} \le \rho < r$, and hence
    \[
    \|w\|_{\mathcal{Y}_T} \le C_{\rm lin,\gamma}\|\nu_0\|_{L^2} + C_{\rm lin,\gamma}C_{\rm rem,\gamma}\|w\|_{\mathcal{Y}_T}^2 \le \frac{\rho}{4} + \frac{\rho}{4} = \frac{\rho}{2}.
    \]
    Since $\mathcal{Y}_T$ embeds continuously into $C([0,T]; L^2(\T^d))$, the map $T \mapsto \|w\|_{\mathcal{Y}_T}$ is continuous and non-decreasing on $(0, T_{\max})$.
    Let $T_* \coloneqq \sup \left\{T < T_{\max} : \|w\|_{\mathcal{Y}_T} \le \rho \right\}$.
    Since $\|w\|_{\mathcal{Y}_T} \to \|\nu_0\|_{L^2} \le \varepsilon < \rho$ as $T \downarrow 0$, one has $T_* > 0$.
    If $T_* < T_{\max}$, continuity gives $\|w\|_{\mathcal{Y}_{T_*}} \le \rho$, hence $\|w\|_{\mathcal{Y}_{T_*}} \le \rho/2$ by the estimate above, and continuity again yields $T \in (T_*, T_{\max})$ with $\|w\|_{\mathcal{Y}_T} \le \rho$, contradicting the definition of $T_*$.
    Therefore $T_* = T_{\max}$, and the claim follows.

    It follows that
    \[
    \sup_{t < T_{\max}} \|\nu_t\|_{L^2} = \sup_{t < T_{\max}} e^{-\gamma t} \|w_t\|_{L^2} \le \rho < r.
    \]
    If $T_{\max} < +\infty$, then $\sup_{T < T_{\max}} \|\nu\|_{\mathcal{Y}_T} < +\infty$, since multiplication by $e^{-\gamma t}$ is bounded on $\mathcal{Y}_T$ uniformly for $T < T_{\max}$. 
    Together with $\sup_{t < T_{\max}} \|\nu_t\|_{L^2} \le \rho < r$, Lemma~\ref{lem:maximal_continuation} extends the solution beyond $T_{\max}$, contradicting maximality. 
    Hence $T_{\max} = \infty$.
    Taking the supremum over finite time intervals in the definition of the $\mathcal{Y}_\infty$-norm gives $e^{\gamma(\cdot)}\nu = w \in \mathcal{Y}_\infty$ and $\|w\|_{\mathcal{Y}_\infty} \le \rho$.
    Letting $T \to \infty$ in the weighted linear estimate and using $C_{\rm lin,\gamma} C_{\rm rem,\gamma} \rho \le 1/4$, we obtain
    \[
    \|w\|_{\mathcal{Y}_\infty} \le C_{\rm lin,\gamma}\|\nu_0\|_{L^2} + C_{\rm lin,\gamma}C_{\rm rem,\gamma}\rho\|w\|_{\mathcal{Y}_\infty} \le C_{\rm lin,\gamma}\|\nu_0\|_{L^2} + \frac{1}{4}\|w\|_{\mathcal{Y}_\infty}.
    \]
    Hence $\|w\|_{\mathcal{Y}_\infty} \le \frac{4}{3} C_{\rm lin,\gamma}\|\nu_0\|_{L^2}$.
    Since the subspace of functions in $L^2(\T^d)$ with zero total mass embeds continuously into $X'$, with constant $C_{\rm emb}$, setting $C_\gamma \coloneqq \frac{4}{3} C_{\rm lin,\gamma}\max\{1, C_{\rm emb}\}$ gives
    \[
    \|\nu_t\|_{X'} \le C_{\rm emb} e^{-\gamma t}\|w_t\|_{L^2} \le C_\gamma e^{-\gamma t}\|\nu_0\|_{L^2}
    \]
    for every $t \ge 0$.
    Finally, $\|\xi_t\|_X = \|\nu_t\|_{X'}$ proves the claimed estimate for $\xi_t$.
\end{proof}

\begin{corollary}
\label{cor:local_w2_stabilization}
    Under the assumptions of Theorem~\ref{thm:local_stab}, for every $t \ge 0$,
    \[
    \Wtwo(\mu_t, \bar\mu) \le C_\gamma e^{-\gamma t} \|\nu_0\|_{L^2}.
    \]
\end{corollary}

\begin{proof}
    By~\cite[Thm.~1.1]{peyre2018comparison}, $\Wtwo(\mu_t, \bar\mu) \le 2\|\mu_t - \bar\mu\|_{X'}$.
    The estimate follows from Theorem~\ref{thm:local_stab}.
\end{proof}

\section{Model classes and extensions}
\label{sec:applications-numerics}

After establishing the geometric and dynamical properties induced by the designed feedback control, we verify Assumptions~\ref{ass:H1}--\ref{ass:H6} for McKean--Vlasov dynamics, for which stabilization was already studied in~\cite{KaliseMoschenPavliotis2025}, and discuss two extensions of the construction.
The first incorporates fixed moment constraints into Fokker--Planck dynamics, while the second considers the same feedback mechanism on closed Riemannian manifolds, with the sphere as a numerical example.
The numerical implementation and additional experiments are collected in Appendix~\ref{app:numerics}.

\subsection{McKean--Vlasov models}
\label{sec:mckean-vlasov-models}

Consider the McKean--Vlasov energy defined in~\eqref{eq:mckean_vlasov_equation} with $W$ even.
If $V, W \in W^{2,\infty}(\T^d)$, then Example~\ref{ex:mckean-vlasov-energy} gives Assumptions~\ref{ass:H1}--\ref{ass:H2} with $\bar\mu \in W^{2,\infty}(\T^d)$.
Moreover,
\[
\Theta_E(\mu) = V + W * \mu, \qquad \mathcal{K}_E \nu = W * \nu,
\]
so $\|\mathcal{K}_E \nu\|_X = \|W * \nu\|_X \le C_0^{1/2} \|W * \nu\|_{\dotHk[1]} \le C_d C_0^{1/2} \|W\|_{W^{2,\infty}} \|\nu\|_{\dotHminusone} \le C \|\nu\|_{X'}$ using Lemma~\ref{lem:mu_bar_bounds} and Lemma~\ref{lem:hilbert_structure}, where $C$ depends on $W$, $\bar\mu$, and $d$.
Hence $\mathcal{K}_E \in \mathcal{L}(X', X)$ and Assumption~\ref{ass:H3} holds with zero remainder.

The relation between the linearized operator on $X'$ and the Hessian operator is also explicit.
Writing $L$ for the closed realization of the linearization defined above, one has, for $\nu \in D(L)$,
\[
L\nu = \sigma\Delta\nu + \nabla\cdot(\nu\nabla\Phi) + \nabla\cdot\left(\bar\mu \nabla(W\ast\nu)\right) = \sigma \nabla \cdot \left(\bar\mu \nabla\left(\frac{\nu}{\bar\mu}\right)\right) + \nabla\cdot\left(\bar\mu\nabla(W\ast\nu)\right),
\]
where the second equality uses $\sigma \nabla \bar\mu = -\bar\mu \nabla\Phi$.
For $\nu=\rieszmap[\bar\mu]\varphi$ with $\varphi,\psi\in C^\infty(\T^d)$, the integrations by parts underlying the Hessian formula in Example~\ref{ex:mckean-vlasov-energy} give
\[
\left\langle -L\rieszmap[\bar\mu]\varphi, \psi\right\rangle_{X', X} = \sigma \int_{\T^d} \bar\mu \,\nabla^2\varphi : \nabla^2\psi \, dx + H_\E(\varphi,\psi) = a(\varphi,\psi).
\]

We next check that $\mathrm{Hess}_{\Wtwo}\E(\mu)$ is well defined on $\mathscr{D}$ for every $\mu \in \Pcal_2(\T^d) \cap L^\infty(\T^d)$.
Fix $\mu \in \Pcal_2(\T^d) \cap L^\infty(\T^d)$ and $\phi \in W^{2,\infty}(\T^d)$, and let $\mu_s = (T_{s\phi})_\sharp\mu$, which by Lemma~\ref{lem:local_displacement_interpolation} is a
local $\Wtwo$-geodesic for $|s|$ small.
Then
\[
\E(\mu_s) = \int_{\T^d} V(T_{s\phi}(x)) \, d\mu(x) + \frac{1}{2} \iint_{\T^d \times \T^d} W (T_{s\phi}(x) - T_{s\phi}(y)) \, d\mu(x) \, d\mu(y).
\]
Since $\nabla V$ is Lipschitz, it is differentiable outside an $\mathcal{L}^d$-null set $N_V$, and since $T_{s\phi}$ is bi-Lipschitz, $T_{s\phi}^{-1}(N_V)$ is $\mathcal{L}^d$-null and hence $\mu$-null.
The difference quotients of $s \mapsto \nabla V(T_{s\phi}) \cdot \nabla\phi$ are dominated by $\|\nabla^2 V\|_{L^\infty} |\nabla\phi|^2 \in L^1(\mu)$, so dominated convergence gives
\[
\frac{d^2}{ds^2} \int_{\T^d} V(T_{s\phi}) \, d\mu = \int_{\T^d} \nabla^2 V(T_{s\phi}) \, \nabla\phi \cdot \nabla\phi \, d\mu.
\]
The same argument applies to the interaction term, using that $x \mapsto T_{s\phi}(x) - T_{s\phi}(y)$ is bi-Lipschitz for each fixed $y$, together with Fubini's theorem.
Evaluating at $s = 0$ recovers the expression for $H_\E$ in Example~\ref{ex:mckean-vlasov-energy} with $\bar\mu$ replaced by $\mu$, which is a quadratic form in $\nabla \phi$.
Hence, $\mathcal{V} = \Pcal_2(\T^d) \cap L^\infty(\T^d)$ is admissible at $\bar\mu$ with any $r_0 \in (0,1)$.

For this class, Assumption~\ref{ass:H4} is also a direct consequence of the explicit Hessian formula.
Indeed, write $\mu = (1 + \rho)\bar\mu$ with $\|\rho\|_{L^\infty} \le r$.
Using $|(1+\rho(x))(1+\rho(y))-1| \le 2r + r^2$, we have
\[
\left|\int_{\T^d}(\mu-\bar\mu) \, \nabla^2V \, \nabla\phi \cdot \nabla\phi \, dx\right| \le r\|\nabla^2V\|_{L^\infty}\|\phi\|_X^2,
\]
and
\begin{align*}
    &\left|\iint_{\T^d\times\T^d}\nabla^2W(x-y)(\nabla\phi(x)-\nabla\phi(y))\cdot(\nabla\phi(x)-\nabla\phi(y)) \, d(\mu \otimes \mu -\bar\mu \otimes \bar\mu)\right| \\
    &\hspace{10cm}\le 4(2r+r^2)\|\nabla^2W\|_{L^\infty}\|\phi\|_X^2.
\end{align*}
Consequently, $\left|\mathrm{Hess}_{\Wtwo}\E(\mu)(\nabla\phi,\nabla\phi) - \mathrm{Hess}_{\Wtwo}\E(\bar\mu)(\nabla\phi,\nabla\phi)\right| \le Cr\|\phi\|_X^2$, for $C > 0$ depending on $\|V\|_{W^{2,\infty}}$ and $\|W\|_{W^{2,\infty}}$.
This verifies Assumption~\ref{ass:H4}.

Regarding Assumption~\ref{ass:H5}, set $\nu_k \coloneqq \rieszmap[\bar\mu]e_k$.
By the ground-state transform (see~\cite{KaliseMoschenPavliotis2025}), $h_k \coloneqq \nu_k / \sqrt{\bar\mu}$ satisfies $H h_k = \lambda_k h_k$ with $H \coloneqq -\sigma \Delta + \Psi + \mathcal{K}$, where $\Psi \in L^\infty(\T^d)$, and $\mathcal{K}$ is an integral operator with kernel $\kappa \in L^\infty(\T^d \times \T^d)$ since $V, W \in W^{2,\infty}(\T^d)$.
By~\cite[Proposition~2.10]{KaliseMoschenPavliotis2025}, $h_k \in H^2(\T^d)$.
Applying the Calderón--Zygmund estimate (see, e.g.,~\cite[Thm. 9.11]{gilbarg2015elliptic}) to
\begin{equation}
    \label{eq:eigenfunction-elliptic-regularity}
    -\sigma \Delta h_k = (\lambda_k - \Psi)h_k - \mathcal{K}h_k,
\end{equation}
we obtain $h_k \in W^{2,p}(\T^d)$ whenever $h_k \in L^p(\T^d)$, for every $1 < p < \infty$.
Starting from $h_k \in H^2(\T^d)$, we iterate the preceding implication using the Sobolev embeddings $W^{2,p}(\T^d)\hookrightarrow L^{dp/(d-2p)}(\T^d)$ if $2p < d$, $W^{2,p}(\T^d) \hookrightarrow L^q(\T^d)$ for every $q < \infty$ if $2p = d$, and $W^{2,p}(\T^d) \hookrightarrow L^\infty(\T^d)$ if $2p > d$.
Hence $h_k \in L^\infty(\T^d)$ after finitely many steps and, consequently, $h_k \in W^{2,p}(\T^d)$ for every $1 < p < \infty$.
It follows that $\nu_k = \sqrt{\bar\mu} h_k \in W^{2,p}(\T^d)$ for every $1 < p <\infty$.
Finally, $e_k$ satisfies
\[
-\Delta e_k = \bar\mu^{-1}\nu_k + \nabla(\log\bar\mu) \cdot \nabla e_k.
\]
Since $\bar\mu^{-1} \nu_k \in L^{p}(\T^d)$ and $\nabla \log\bar\mu \in L^{\infty}(\T^d)$, the Calderón--Zygmund estimate and the same bootstrap strategy as above, now using the embeddings $W^{2,p}(\T^d) \hookrightarrow W^{1,q}(\T^d)$, yield $e_k \in W^{2,p}(\T^d)$ for every $1 < p < \infty$. 
The right-hand side then belongs to $W^{1,p}(\T^d)$, and the corresponding higher-order Calderón--Zygmund estimate yields $e_k \in W^{3,p}(\T^d)$ for every $1 < p < \infty$.
Choosing $p > d$, the Sobolev--Morrey embedding $W^{3,p}(\T^d) \hookrightarrow W^{2,\infty}(\T^d)$ verifies Assumption~\ref{ass:H5}.
It remains to verify Assumption~\ref{ass:H6}.
The uncontrolled nonlinear remainder is $\mathfrak{R}(\nu) = \nabla \cdot (\nu \nabla (W * \nu))$ and hence
\[
\mathfrak{R}(\nu) - \mathfrak{R}(\zeta) = \nabla \cdot ((\nu-\zeta) \nabla (W * \nu)) + \nabla \cdot (\zeta \nabla (W*(\nu-\zeta))).
\]
The bound $\|\nabla W * f\|_{L^\infty} \le C\|f\|_{L^2}$ therefore yields the Lipschitz estimate required in Assumption~\ref{ass:H6}.
Local existence in $\mathcal{Y}_T$ follows from the fixed-point argument of~\cite[Section~3.4]{KaliseMoschenPavliotis2025}, using Proposition~\ref{prop:weighted_linear_density} for the linear estimate and the bound above, together with Lemma~\ref{lem:feedback_remainder}, for the nonlinear terms.
The associated drift $\nabla\Phi + \nabla(W * \nu) + \nabla(P_\delta \Pi \rieszmap[\bar\mu]^{-1} \nu)$ is bounded on finite time intervals by~\ref{ass:H1}, the bound above, and~\ref{ass:H5}.
Hence the mass-conservation and nonnegativity arguments of~\cite[Appendix~B]{KaliseMoschenPavliotis2025} apply as well.
Figure~\ref{fig:mv-convexification} illustrates the finite-rank feedback convexification for the double-well and Kuramoto examples.

\begin{figure}[ht]
    \centering
    \begin{minipage}[t]{0.48\textwidth}
        \centering
        \includegraphics[width=\linewidth]{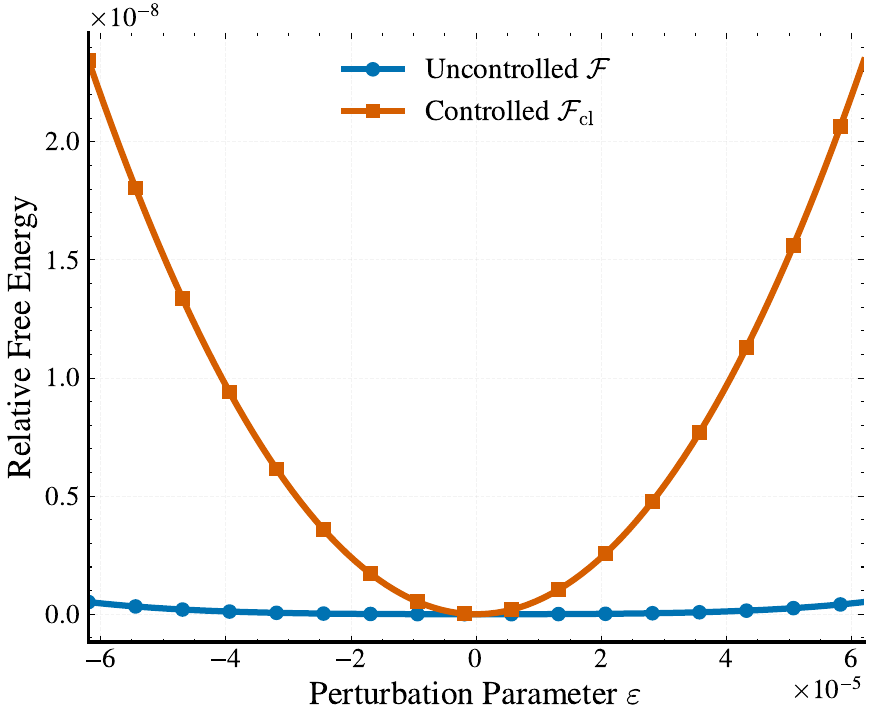}
    \end{minipage}
    \hfill
    \begin{minipage}[t]{0.48\textwidth}
        \centering
        \includegraphics[width=\linewidth]{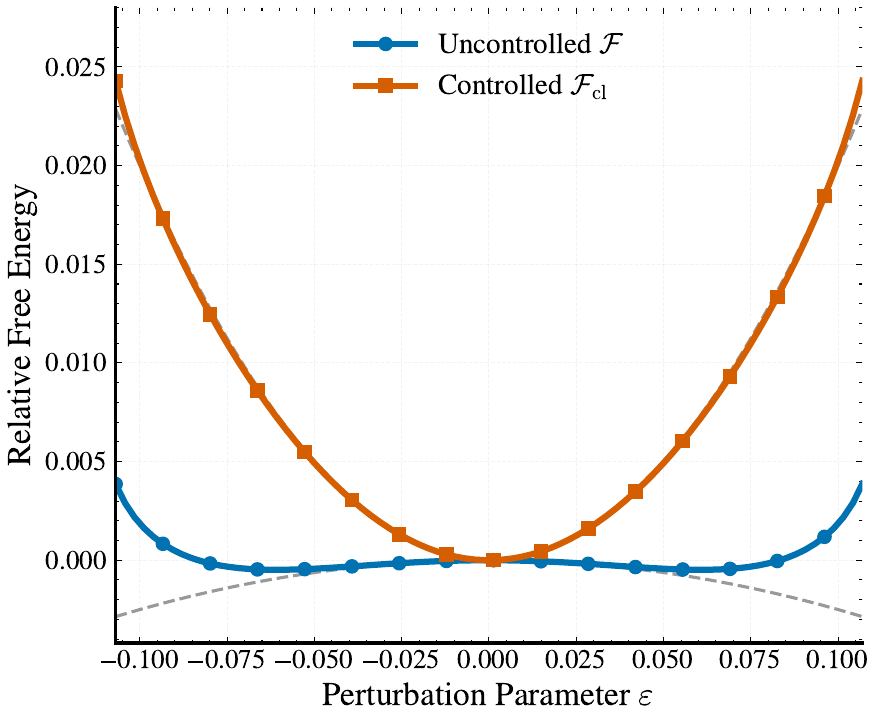}
    \end{minipage}
    \caption{\textbf{Convexification for McKean--Vlasov examples}.
    In both panels, we plot the functionals $\F(\mu_\varepsilon)$ and $\F_{\rm cl}(\mu_\varepsilon)$ along the family $\mu_\varepsilon = \Xi(\varepsilon\phi)$, with $\phi$ fixed.
    Left: double-well energy with $V(x) = -2\cos(4\pi x)$, $W \equiv 0$, and $\sigma = 0.5$.
    Right: Kuramoto interaction on $2\pi \T$ with $V \equiv 0$, $W(x) = -3\cos(2\pi x)$, and $\sigma = 1$. 
    The uniform distribution is linearly unstable, and the feedback restores local convexity.}
    \label{fig:mv-convexification}
\end{figure}

\subsection{Constrained dynamics}
\label{sec:constrained-dynamics}

We next consider dynamics with fixed linear moments on $\T^d$, so that the admissible set is
\[
\mathcal{M} \coloneqq \left\{\mu \in \Pcal_2(\T^d) : \ \int_{\T^d} g_\ell \, d\mu = c_\ell, \ell = 1, \ldots, k\right\},
\]
with $g_\ell \in W^{2,\infty}(\T^d)$. 
The constrained evolution can be written as
\[
\partial_t \mu = \sigma \Delta\mu + \nabla \cdot \left(\mu\nabla\Theta_E(\mu)\right) + \sum_{\ell=1}^k m_\ell(t) \nabla \cdot (\mu\nabla g_\ell),
\]
where $m_\ell(t) \in \R$ are scalar Lagrange multipliers chosen so that the dynamics stay in $\mathcal{M}$ at all times.
For $\ell = 1, \ldots, k$, set $C_\ell(\mu) \coloneqq \int_{\T^d} g_\ell \, d\mu$.
Differentiating $C_\ell(\mu_t) = c_\ell$ and integrating by parts on $\T^d$ gives
\[
\sum_{r=1}^k G_{\ell r}(\mu_t) m_r(t) = -b_\ell(\mu_t), \quad b_\ell(\mu) \coloneqq \langle g_\ell, \Theta_E(\mu) \rangle_{\potential[\mu]} - \sigma\int_{\T^d}\mu\,\Delta g_\ell\,dx,
\]
where $G_{\ell r}(\mu) \coloneqq \langle g_\ell, g_r \rangle_{\potential[\mu]}$.
If $G(\mu_t) = (G_{\ell r}(\mu_t))_{\ell r}$ is invertible, the multipliers are uniquely determined by $m(t) = -G(\mu_t)^{-1} b(\mu_t)$; see, e.g.,~\cite{EberleNiethammerSchlichting2017}.

Moment constraints of this type also arise in mean-field models for lithium-ion electrodes.
A charging protocol may prescribe, for instance, the first moment of the particle state~\cite{DreyerEtAl2010,EberleNiethammerSchlichting2017},
\[
\int_\R x \rho_t(x) \, dx = c(t),
\]
which corresponds to a moving constraint level rather than a fixed one when compared to the proposition below.
For fixed $c_\ell$, the tangent space and the projected Hessian at a constrained equilibrium are time-independent.
If the constraint levels vary in time, the reference state follows a prescribed path and the corresponding linearization is generally nonautonomous, leading to a time-dependent Riccati equation.

\begin{proposition}[Constrained equilibrium and projected linearization]
    \label{prop:constrained_projected_hessian}
    Let $\bar\mu \in \mathcal{M}$ be stationary for $\F$ restricted to $\mathcal{M}$, and assume that Assumptions~\ref{ass:H1}--\ref{ass:H2} hold at $\bar\mu$.
    Define
    \[
    g_\ell^\circ \coloneqq g_\ell - \int_{\T^d} g_\ell \,d\bar\mu, \qquad X_{\mathcal{M}} \coloneqq \left\{\phi \in \potential[\bar\mu] : \langle \phi,g_\ell^\circ\rangle_{\potential[\bar\mu]} = 0, \ell = 1, \ldots, k\right\},
    \]
    and assume that $G^X_{\ell r} \coloneqq \langle g_\ell^\circ, g_r^\circ\rangle_{X}$ is invertible.
    Then the following holds:
    \begin{itemize}
        \item[(i)] there exist unique multipliers $\bar\lambda_1, \ldots, \bar\lambda_k \in \R$ and a constant $c_{\bar\mu} \in \R$ such that
        \begin{equation}
            \label{eq:euler_lagrange_constrained}
            \sigma(\log\bar\mu + 1) + \Theta_E(\bar\mu) + \sum_{\ell=1}^k \bar\lambda_\ell g_\ell = c_{\bar\mu} \quad \text{a.e.\ on } \T^d,
        \end{equation}
        so $\bar\mu \propto \exp \left(-(\Theta_E(\bar\mu) + \sum_{\ell=1}^k \bar\lambda_\ell g_\ell) / \sigma \right)$, and $\bar\mu \in W^{1,\infty}(\T^d)$ satisfies $0 < c_0 \le \bar\mu \le C_0$.
        \item[(ii)] for $\ell = 1, \dots, k$, define 
        \[
        h_\ell(\phi,\psi) \coloneqq \mathrm{Hess}_{\Wtwo} C_\ell(\bar\mu)(\nabla\phi, \nabla\psi) = \int_{\T^d} \nabla^2 g_\ell \nabla \phi \cdot \nabla \psi \, d\bar\mu.
        \]
        The constrained Hessian form
        \[
        a_{\mathcal{M}}(\phi,\psi) \coloneqq a(\phi,\psi) + \sum_{\ell=1}^k \bar\lambda_\ell h_\ell(\phi,\psi), \qquad D(a_{\mathcal{M}}) = D(a) \cap X_{\mathcal{M}},
        \]
        is densely defined, closed, and bounded from below on $X_{\mathcal{M}}$.
        It therefore defines a self-adjoint operator $A_{\mathcal{M}}$ on $X_{\mathcal{M}}$ with compact resolvent.
        \item[(iii)] If, in addition, Assumption~\ref{ass:H3} holds at $\bar\mu$, then the linearized constrained dynamics in $X_{\mathcal{M}}$ admit the weak formulation
        \[
        \langle \partial_t\xi_t, \psi\rangle_X + a_{\mathcal{M}}(\xi_t, \psi) = 0, \qquad \psi \in D(a_{\mathcal{M}}).
        \]
    \end{itemize}
\end{proposition}

\begin{proof}
    For $\phi \in X$, $DC_\ell(\bar\mu)[\rieszmap[\bar\mu]\phi] = \langle\rieszmap[\bar\mu]\phi,g_\ell^\circ\rangle_{X',X} = \langle\phi,g_\ell^\circ\rangle_X$.
    Since $G^X$ is invertible, the differentials $DC_1(\bar\mu), \ldots, DC_k(\bar\mu)$ are linearly independent on $X'$.
    The Lagrange multiplier condition for the restriction of $\F$ to $\mathcal{M}$ therefore gives the existence of $\bar\lambda_1, \ldots, \bar\lambda_k, c_{\bar\mu} \in \R$ such that~\eqref{eq:euler_lagrange_constrained} holds.
    If $(\widetilde\lambda_1, \ldots, \widetilde\lambda_k)$ is another family of multipliers, subtracting both equations and subtracting the average with respect to $\bar\mu$ yields $\sum_{\ell=1}^k(\bar\lambda_\ell - \widetilde\lambda_\ell) g_\ell^\circ = 0$ in $X$.
    Taking the $X$-inner product with $g_r^\circ$ for $r=1,\ldots,k$ gives $G^X(\bar\lambda - \widetilde\lambda) = 0$, and hence $\bar\lambda = \widetilde\lambda$.
    Equation~\eqref{eq:euler_lagrange_constrained} then yields
    \[
    \bar\mu \propto \exp\left(-\frac{1}{\sigma} \left( \Theta_E(\bar\mu)+\sum_{\ell=1}^k\bar\lambda_\ell g_\ell \right) \right).
    \]
    Since $\Theta_E(\bar\mu) \in W^{1,\infty}(\T^d)$ and $g_\ell \in W^{2,\infty}(\T^d)$, (i) follows as in Lemma~\ref{lem:mu_bar_bounds}.

    For $\phi \in C^{\infty}(\T^d)$ and $\mu_s = \Xi(s\phi)$,
    \[
    \frac{d^2}{ds^2} C_\ell(\mu_s)\Big|_{s=0} = \int_{\T^d}\nabla^2g_\ell\,\nabla\phi \cdot \nabla\phi \, d\bar\mu.
    \]
    Polarization gives the expression for $h_\ell$. Moreover, $|h_\ell(\phi,\psi)| \le \|\nabla^2g_\ell\|_{L^\infty} \|\phi\|_X \|\psi\|_X$, so each $h_\ell$ is a bounded symmetric bilinear form on $X$.
    Thus, $a_{\mathcal{M}}$ is a bounded symmetric form perturbation of $a$ restricted to $D(a) \cap X_{\mathcal{M}}$.
    It remains to verify that $D(a_{\mathcal{M}})$ is dense in $X_{\mathcal{M}}$.
    Indeed, the $X$-orthogonal projection is
    \[
    P_{\mathcal{M}}\phi = \phi - \sum_{\ell,r=1}^k ((G^X)^{-1})_{\ell r} \langle \phi,g_r^\circ \rangle_X g_\ell^\circ.
    \]
    Since $g_\ell^\circ \in W^{2,\infty}(\T^d) \subset D(a)$, one has $P_{\mathcal{M}}D(a) \subset D(a) \cap X_{\mathcal{M}}$. 
    The density of $D(a)$ in $X$ therefore implies that $D(a_{\mathcal{M}})$ is dense in $X_{\mathcal{M}}$.
    Since $X_{\mathcal{M}}$ is closed in $X$, the restriction of $a$ to $D(a) \cap X_{\mathcal{M}}$ is closed and bounded from below on $X_{\mathcal{M}}$.
    The same argument as in Theorem~\ref{thm:hessian_form} then shows that $a_{\mathcal{M}}$ is closed and bounded from below.
    The representation theorem therefore defines a self-adjoint operator $A_{\mathcal{M}}$ on $X_{\mathcal{M}}$.
    The embedding $D(a_{\mathcal{M}}) \hookrightarrow X_{\mathcal{M}}$ is compact given that $\dotHk[2](\T^d) \hookrightarrow \dotHk[1](\T^d) \simeq X$ is compact. 
    The proof of Theorem~\ref{thm:compact_resolvent_a} then shows that $A_{\mathcal{M}}$ has compact resolvent.

    Let $\mathcal{G}_{\mathcal{M}}(\mu) \coloneqq \mathcal{G}(\mu) + \sum_{\ell=1}^k m_\ell(\mu)\nabla\cdot(\mu\nabla g_\ell)$, where $m(\mu) = -G(\mu)^{-1}b(\mu)$.
    Assume now~\ref{ass:H3} and define $\E_{\bar\lambda} \coloneqq \E + \sum_{\ell=1}^k \bar\lambda_\ell C_\ell$.
    By parts~(i)--(ii), Assumptions~\ref{ass:H1}--\ref{ass:H2} hold for $\E_{\bar\lambda}$.
    For $\zeta \in \mathcal{U}_p$ ($p > d/2 + 2$), the added constraint terms cancel in first-variation differences, and hence
    \[
    \Theta_{E_{\bar\lambda}}(\Xi(\zeta)) - \Theta_{E_{\bar\lambda}}(\bar\mu) = \Theta_E(\Xi(\zeta)) - \Theta_E(\bar\mu).
    \]
    Thus Assumption~\ref{ass:H3} holds for $\E_{\bar\lambda}$ with the same operator $\mathcal{K}_E$.
    By the Euler--Lagrange equation and integration by parts,
    \[
    b_\ell(\bar\mu) = \int_{\T^d} \bar\mu \nabla g_\ell \cdot \nabla\Theta_E(\bar\mu) \, dx - \sigma\int_{\T^d} \bar\mu \Delta g_\ell \, dx = -\sum_{r=1}^k G^X_{\ell r} \bar\lambda_r.
    \]
    Since $G(\bar\mu) = G^X$, one has $m(\bar\mu) = \bar\lambda$.

    Let $\mu_\varepsilon = \Xi(\varepsilon\phi)$ and $q_\varepsilon = (\mu_\varepsilon - \bar\mu)/\varepsilon$. 
    Assumption~\ref{ass:H3} and Lemma~\ref{lem:pushforward_map} imply $G(\mu_\varepsilon) \to G(\bar\mu)$ and $b(\mu_\varepsilon) \to b(\bar\mu)$ since $\mu_{\varepsilon} \to \bar\mu$ in $L^1$-norm and $(\mu_\varepsilon)_{\varepsilon}$ is uniformly bounded in $L^{\infty}(\T^d)$.
    Since $G(\bar\mu) = G^X$ is invertible, $G(\mu_\varepsilon)$ is invertible for all sufficiently small $\varepsilon$.
    Hence $m(\mu_\varepsilon) \to \bar\lambda$.
    For $\psi \in D(a_{\mathcal{M}})$,
    \[
    \int_{\T^d} \bar\mu \nabla g_\ell \cdot \nabla\psi \, dx = \langle g_\ell^\circ,\psi\rangle_X = 0.
    \]
    It follows that
    \[
    \frac{1}{\varepsilon} \sum_{\ell=1}^k (m_\ell(\mu_\varepsilon) - \bar\lambda_\ell) \left\langle \nabla \cdot (\mu_\varepsilon \nabla g_\ell), \psi \right\rangle_{X',X} = -\sum_{\ell=1}^k (m_\ell(\mu_\varepsilon) - \bar\lambda_\ell) \int_{\T^d}q_\varepsilon \nabla g_\ell \cdot \nabla\psi\,dx \to 0.
    \]
    Therefore, when tested against functions in $D(a_{\mathcal{M}})$, the linearization of $\mathcal{G}_{\mathcal{M}}$ agrees in weak form with that of the gradient flow associated with $\F + \sum_{\ell=1}^k \bar\lambda_\ell C_\ell$.
    Applying Proposition~\ref{prop:linearized_operator_identification} to $\E_{\bar\lambda}$ gives
    \[
    -\left\langle D\mathcal{G}_{\mathcal{M}}(\bar\mu) \rieszmap[\bar\mu]\phi, \psi \right\rangle_{\dot{H}^{-2}, \dotHk[2]} = a_{\mathcal{M}}(\phi, \psi), \qquad \phi \in C^\infty(\T^d) \cap X_{\mathcal{M}}, \psi \in D(a_{\mathcal{M}}).
    \]
    The identity then extends to all $\phi, \psi \in D(a_{\mathcal{M}})$ by density and continuity of the form.
    Therefore, we can write $\langle \partial_t\xi, \psi\rangle_X + a_{\mathcal{M}}(\xi, \psi) = 0$ for $\psi \in D(a_{\mathcal{M}})$, which proves (iii).
\end{proof}

The feedback construction of Section~\ref{sec:feedback} can then be carried out on the constrained space $X_{\mathcal{M}}$, using the eigenfunctions of $A_{\mathcal{M}}$ associated with eigenvalues below the prescribed threshold $\delta > 0$. 
With the same diagonal choice of the cost operator $Q$, the resulting Riccati operator shifts only these constrained modes. 
Since the corresponding control directions $\alpha_j$ lie in $X_{\mathcal{M}}$, the linearized moment constraints are preserved. 
The same feedback can be incorporated into the nonlinear constrained dynamics by viewing the closed-loop equation as the constrained Wasserstein gradient flow of $\F + \mathcal{Q}_{\mathcal{M}}$, where $\mathcal{Q}_{\mathcal{M}}$ is the analogue of $\mathcal{Q}$ built from $A_{\mathcal{M}}$.
The closed-loop Lagrange multipliers are determined by
\[
G(\mu)m^{\rm cl}(\mu) = - b(\mu) - \varrho(\mu), \qquad \varrho_\ell(\mu) \coloneqq \left\langle g_\ell, \frac{\delta \mathcal{Q}_{\mathcal{M}}}{\delta \mu}\right\rangle_{\potential}.
\]
Since $\delta \mathcal{Q}_{\mathcal{M}} /\delta\mu \in X_{\mathcal{M}}$ and it vanishes at $\bar\mu$, one has $\varrho(\bar\mu) = 0$ and $D\varrho(\bar\mu) = 0$.
Hence, the correction to the multipliers induced by the controllers does not alter the constrained linearization, which remains the $\delta$-stabilized system on $X_{\mathcal{M}}$.
This construction is illustrated in Appendix~\ref{app:sec:standard-models} for a double-well model in one dimension.

\subsection{Closed Riemannian manifolds and stabilization on the sphere}
\label{sec:sphere-example}

We now extend the Wasserstein Hessian realization developed on the flat torus to connected closed Riemannian manifolds.
With this formulation, the same strategy can be used to construct a feedback control. 
The main change is that the Ricci tensor enters the entropy Hessian, while the Riccati construction is otherwise unchanged.
We record the extension needed for the linear Fokker--Planck equation considered below.

Let $(\Omega, g)$ be a smooth, closed, and connected Riemannian manifold, and let $d\vol_g$ denote the Riemannian volume measure on $\Omega$.
Assume $V \in W^{2,\infty}(\Omega)$ and consider the free energy 
\[
\F_g(\mu) \coloneqq \sigma\Ent_g(\mu) + \int_\Omega V \, d\mu, \quad \Ent_g(\mu) \coloneqq \int_\Omega \mu \log\mu \, d\vol_g, \quad \sigma > 0.
\]
Its gradient flow in $\Pcal_2(\Omega)$ is~\cite[Ex. 15.5]{villani2008optimal}
\begin{equation}
    \label{eq:manifold-FP}
    \partial_t\mu = \sigma\Delta_g\mu + \operatorname{div}_g(\mu\nabla_gV),
\end{equation}
and its stationary density is $\bar\mu \propto e^{-V/\sigma} \in W^{2,\infty}(\Omega)$, where $\nabla_g$, $\operatorname{div}_g$, and $\Delta_g \coloneqq \operatorname{div}_g \nabla_g$ denote the Riemannian gradient, divergence, and Laplace--Beltrami operator, respectively.
Since $\Omega$ is compact, $\bar\mu$ is bounded above and below away from zero.

We consider the space $X_g \coloneqq \{\varphi : \Omega \to \R : \|\varphi\|_{X_g} < \infty\} / \{ \|\cdot\|_{X_g} = 0 \}$, where
\[
\|\varphi\|_{X_g} \coloneqq \sqrt{\langle \varphi, \varphi \rangle_{X_g}}, \qquad \langle \varphi, \psi \rangle_{X_g} \coloneqq \int_\Omega \langle \nabla_g\varphi ,\nabla_g\psi \rangle_g \, d\bar\mu.
\]
Since $\Omega$ is connected, and $\bar\mu$ is bounded above and below away from zero, $X_g$ is linearly isomorphic to $\dotHk[1](\Omega)$ with equivalent norms.
We denote its Riesz map by $\rieszmap[\bar\mu]^g \varphi \coloneqq -\operatorname{div}_g (\bar\mu \nabla_g\varphi)$, where $\dotHk[k](\Omega) \coloneqq H^k(\Omega)/\R$ is as previously stated.
Using the Wasserstein Hessian formula in~\cite[Formula~15.7]{villani2008optimal} gives, for $\varphi, \psi \in C^\infty(\Omega)$,
\[
\mathrm{Hess}_{\Wtwo} \F_g(\bar\mu)(\nabla_g\varphi, \nabla_g\psi) = a_g(\varphi,\psi),
\]
where
\[
a_g(\varphi,\psi) \coloneqq \int_\Omega \left[ \sigma \nabla_g^2 \varphi : \nabla_g^2 \psi + \left(\sigma\mathrm{Ric}_g + \nabla_g^2V \right) (\nabla_g\varphi, \nabla_g\psi) \right] d\bar\mu.
\]

\begin{proposition}[Closed-manifold extension]
    \label{prop:closed_manifold_extension}
    The form $a_g$, with domain $D(a_g) = \dotHk[2](\Omega)$, is densely defined, symmetric, closed, and bounded from below on $X_g$. 
    Its associated self-adjoint operator $A_g$ has compact resolvent.
    Moreover, the linearization of~\eqref{eq:manifold-FP} at $\bar\mu$ admits the weak formulation
    \[
    \langle \partial_t\xi_t, \psi\rangle_{X_g} + a_g(\xi_t, \psi) = 0, \qquad \psi \in D(a_g),
    \]
    where $\xi = (\rieszmap[\bar\mu]^g)^{-1} (\mu - \bar\mu)$.
\end{proposition}

\begin{proof}
    The symmetry of $a_g$ follows from the symmetry of $\mathrm{Ric}_g$ and $\nabla_g^2 V$.
    We have
    \[
    \left|\int_\Omega \left[\sigma\mathrm{Ric}_g(\nabla_g \varphi, \nabla_g \psi) + \nabla_g^2 V(\nabla_g \varphi, \nabla_g \psi) \right] d\bar\mu \right| \leq C_g \|\varphi\|_{X_g} \|\psi\|_{X_g},
    \]
    where $C_g \coloneqq \|\nabla_g^2 V\|_{L^\infty(\Omega)} + \sigma\|\mathrm{Ric}_g\|_{L^{\infty}(\Omega)} < \infty$ because $\Omega$ is compact, $g$ is smooth, and $V \in W^{2,\infty}(\Omega)$.
    Moreover, given the bounds on $\bar\mu$, integration with respect to $d\bar\mu$ and $d\vol_g$ leads to equivalent $L^2$-norms.
    The equivalence of Sobolev norms on the closed manifold $\Omega$, together with the Poincaré inequality on the connected manifold $\Omega$, therefore yields
    \[
    \|\varphi\|_{\dotHk[2](\Omega)}^2 \simeq \int_{\Omega} \|\nabla_g^2 \varphi\|_g^2 \, d\bar\mu + \|\varphi\|_{X_g}^2.
    \]
    Hence, for $\gamma > C_g$, the norm
    \[
    \|\varphi\|_{a_g, \gamma}^2 \coloneqq a_g(\varphi,\varphi) + \gamma\|\varphi\|_{X_g}^2
    \]
    is equivalent to the $\dotHk[2](\Omega)$-norm. 
    Since $\dotHk[2](\Omega)$ is complete and dense in $X_g$, the form $a_g$ is densely defined, closed, and bounded from below. 
    The representation theorem then gives an associated self-adjoint operator $A_g$ on $X_g$.
    Additionally, the embedding $\dotHk[2](\Omega) \hookrightarrow \dotHk[1](\Omega)$ is compact, and the $\dotHk[1](\Omega)$-norm is equivalent to the $X_g$-norm. 
    The argument of Theorem~\ref{thm:compact_resolvent_a} therefore shows that $A_g$ has compact resolvent.

    It remains to identify the linearized equation.
    Set
    \[
    \Delta_{\bar\mu} \varphi \coloneqq \Delta_g \varphi - \frac{1}{\sigma} \langle\nabla_g V, \nabla_g \varphi\rangle_g, \qquad \mathcal{G}_g \nu \coloneqq \sigma \Delta_g \nu + \operatorname{div}_g(\nu \nabla_g V).
    \]
    Since $\nabla_g \bar\mu = -\sigma^{-1}\bar\mu \nabla_g V$, one has $\rieszmap[\bar\mu]^g\varphi = -\bar\mu \Delta_{\bar\mu} \varphi$ and $\mathcal{G}_g \rieszmap[\bar\mu]^g\varphi = \sigma \rieszmap[\bar\mu]^g \Delta_{\bar\mu}\varphi$.
    Thus, for $\varphi, \psi \in C^\infty(\Omega)$, integration by parts with respect to $d\bar\mu$ gives
    \[
    -\left\langle \mathcal{G}_g \rieszmap[\bar\mu]^g \varphi, \psi \right\rangle_{\dot{H}^{-2}(\Omega), \dotHk[2](\Omega)} = \sigma\int_\Omega(\Delta_{\bar\mu} \varphi)(\Delta_{\bar\mu} \psi) \, d\bar\mu = a_g(\varphi,\psi),
    \]
    where the last identity follows by polarizing and integrating the weighted Bochner formula
    \[
    \frac{1}{2} \Delta_{\bar\mu}|\nabla_g f|_g^2 = |\nabla_g^2f|_g^2 + \langle\nabla_g f,\nabla_g \Delta_{\bar\mu} f\rangle_g + \left( \mathrm{Ric}_g + \frac{1}{\sigma} \nabla_g^2 V \right)(\nabla_g f,\nabla_g f),
    \]
    and using the symmetry of $\Delta_{\bar\mu}$ with respect to $d\bar\mu$.
    By density, the identity extends to $\varphi, \psi \in D(a_g)$.
    If $\nu = \rieszmap[\bar\mu]^g\xi$ satisfies $\partial_t\nu = \mathcal{G}_g\nu$, then, for $\psi \in D(a_g)$,
    \[
    \langle \partial_t \xi_t, \psi\rangle_{X_g} = \langle \rieszmap[\bar\mu]^g \partial_t \xi_t, \psi \rangle_{X_g', X_g} = \langle \mathcal{G}_g \rieszmap[\bar\mu]^g\xi, \psi \rangle_{X_g', X_g} = -a_g(\xi, \psi),
    \]
    which gives the claimed weak formulation with $A_g = -\sigma \Delta_{\bar\mu}$.
\end{proof}

Since the Ricci contribution is only a bounded perturbation on $X_g$, the compact resolvent structure underlying the construction of Section~\ref{sec:feedback} is preserved.
The same feedback procedure based on the eigenfunctions of $A_g$ can therefore be carried out, yielding linearized $\delta$-stabilization on $\Omega$. 
A corresponding nonlinear result would require analogues of Assumptions~\ref{ass:H4}--\ref{ass:H6}.

We now specialize to the unit sphere $\mathbb{S}^2$ with its round metric.
Let $\theta \in [0, \pi]$ denote the polar angle and consider the double-well potential $V(\theta) = -\kappa \cos^2\theta$, with $\kappa > 0$.
Equation~\eqref{eq:manifold-FP} then has stationary density
\[
\bar\mu \propto \exp\left(\frac{\kappa}{\sigma} \cos^2\theta\right).
\]
The controlled and uncontrolled convergence to $\bar\mu$ is illustrated in Figure~\ref{fig:sphere}.
For $0 < \theta < \pi$, let $e_\theta \coloneqq \partial_\theta$ and $e_\varphi \coloneqq (\sin\theta)^{-1} \partial_\varphi$, so that $(e_\theta, e_\varphi)$ forms an orthonormal basis of the tangent space at each point away from the poles.
Since $\mathrm{Ric}_g = g$ on the unit sphere, at the equator one has
\[
\left(\sigma\mathrm{Ric}_g + \nabla_g^2 V\right)(e_\theta, e_\theta) = \sigma - 2\kappa.
\]
Hence, if $\kappa > \sigma/2$, this quantity is negative in a neighborhood of the equator.
Thus, the positive Ricci curvature of the sphere does not ensure local displacement convexity of the free energy by itself.

\begin{figure}[!htbp]
    \centering
    \includegraphics[width=0.5\textwidth]{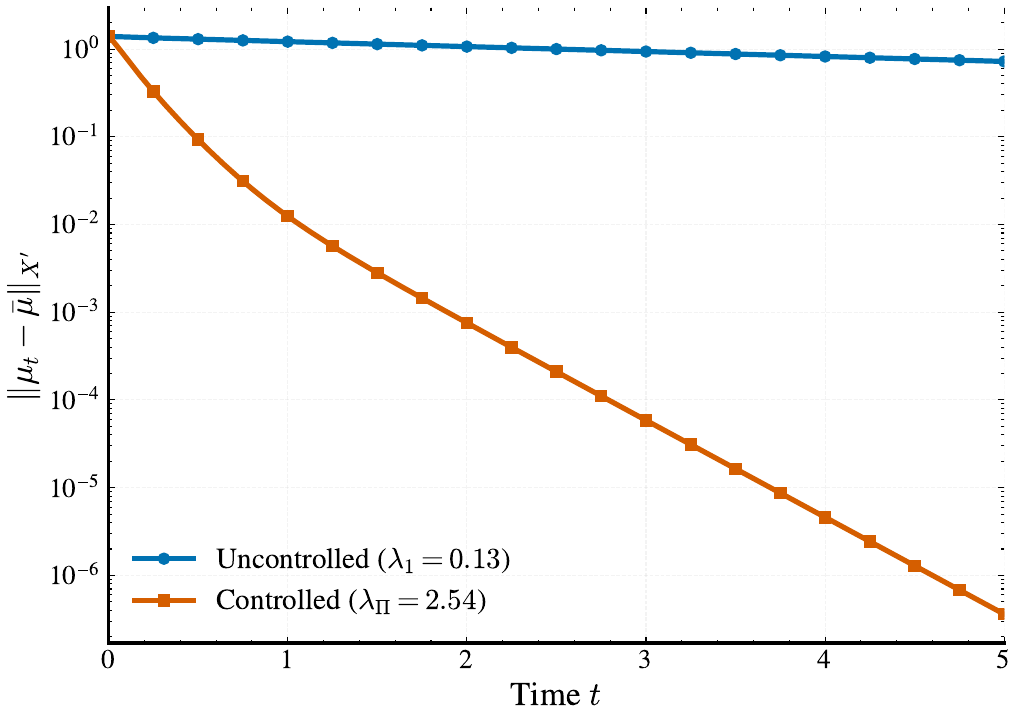}
    \caption{\textbf{Feedback acceleration on $\mathbb{S}^2$}. 
    For $V(\theta) = -\kappa\cos^2\theta$, with $\kappa = 2$ and $\sigma = 0.5$, the uncontrolled operator has spectral gap $\lambda_1 \approx 0.13$. 
    For the target rate $\delta = 2$, the rank-one Riccati operator raises the first eigenvalue above $\lambda_2$. 
    Consequently, the spectral gap of the closed-loop operator is $\lambda_\Pi = \lambda_2 \approx 2.54$, accelerating convergence to $\bar\mu$.}
    \label{fig:sphere}
\end{figure}

\section{Conclusion}

This work identifies the Wasserstein Hessian at a stationary density with a self-adjoint operator whose negative is conjugate to the generator of the linearized dynamics.
This identification gives the feedback control built via an algebraic Riccati equation a geometric interpretation: it produces a finite-rank perturbation of the Hessian that reshapes the free energy locally while also stabilizing the linearized gradient flow. 
Under the local Hessian lower-semicontinuity assumptions, the modified energy is locally $\hat\lambda$-displacement convex, with $\hat\lambda > \delta$, on a H\"older neighborhood $\Nbhd^\alpha_\varepsilon$ of $\bar\mu$.
With the additional local existence and remainder estimates, $\bar\mu$ is a locally exponentially stable equilibrium, while the controlled PDE retains its Wasserstein gradient-flow structure.

The construction presupposes knowledge of the target $\bar\mu$ and of the eigenfunctions of the Wasserstein Hessian realization.
Its online implementation, however, does not require the reconstruction of the density since, for $k \in \Sigma_\delta$,
\[
u_k(t) = \pi_k \int_{\T^d} e_k \, d(\mu_t - \bar\mu) = \pi_k \int_{\T^d} e_k \, d\mu_t,
\]
where the second equality uses the normalization~\eqref{eq:mean_zero}.
Extensions to settings in which the target is not known explicitly are left for future work.

\section*{Acknowledgements}

D. Kalise is supported by the Air Force Office of Scientific Research under award number FA8655-26-1-B011.
L. M. Moschen gratefully acknowledges financial support from the Department of Mathematics at Imperial through the Roth Scholarship, as well as additional support from the ICL-CNRS Lab.
G. A. Pavliotis is partially supported by an ERC-EPSRC Frontier Research Guarantee through grant no. EP/X038645, ERC through Advanced grant no. 247031.

\section*{Code availability}

Python implementations of the Wasserstein Hessian discretization, the feedback control construction, and the numerical experiments presented in this paper are available at \url{https://github.com/lucasmoschen/feedback-control-gradient-flows}.

\bibliographystyle{abbrv}
\bibliography{refs.getmref}

\appendix

\section{Proof of Proposition~\ref{prop:entropy_hessian}}
\label{app:proof_prop_entropy}

\begin{proof}
    Fix $\varphi \in W^{2,\infty}(\T^d)$.
    Let $\mu_s = (T_{s\varphi})_{\sharp}\mu$ for $|s|$ sufficiently small.
    By Lemma~\ref{lem:local_displacement_interpolation}, $(\mu_s)$ is the corresponding local $\Wtwo$-geodesic.
    Let $J_s(x) \coloneqq \det(I + s \nabla^2\varphi(x))$.
    The change of variables formula gives $\mu_s(T_{s\varphi}(x)) J_s(x) = \mu(x)$ for a.e.\ $x$ since $T_{s\varphi}$ is a bi-Lipschitz homeomorphism.
    Hence
    \[
    \Ent(\mu_s) = \int_{\T^d} \mu(x)\log\mu(x) \, dx - \int_{\T^d} \mu(x) \log J_s(x) \, dx.
    \]
    The first term is finite because $\mu \in L^\infty(\T^d)$ and $\T^d$ has finite volume.
    For a.e.\ $x$, $J_s(x)$ is a polynomial in $s$ whose coefficients are bounded in $L^\infty$, and $J_s \ge c > 0$ uniformly in $x$ for $|s|$ small, so differentiation under the integral is justified.
    Using
    \[
    \frac{d}{ds} \log\det(I+sS) = \operatorname{tr}((I+sS)^{-1}S),
    \]
    we obtain
    \[
    \frac{d^2}{ds^2}\Ent(\mu_s) \Big|_{s=0} = \int_{\T^d}\mu(x)\operatorname{tr}((\nabla^2\varphi(x))^2)\,dx = \int_{\T^d} \mu(x) |\nabla^2\varphi(x)|^2 \, dx.
    \]
    Polarization gives~\eqref{eq:Hess_ent}.
\end{proof}

\section{Displacement-map calculations}
\label{app:displacement-map}

We collect the pushforward and displacement facts used in Proposition~\ref{prop:linearized_operator_identification} and Theorem~\ref{thm:local_strong_displacement_convexity}.

\begin{lemma}
    \label{lem:pushforward_map}
    Assume~\ref{ass:H1}.
    For $\xi \in \mathcal{U}_p$, set $\Xi(\xi) = (T_\xi)_{\sharp}\bar\mu$.
    Then $T_\xi$ is a $C^1$-diffeomorphism of $\T^d$.
    Moreover, the translated map $\xi \mapsto \Xi(\xi)-\bar\mu$ maps $\mathcal{U}_p$ into $X'$ and is differentiable in a neighborhood of $0$ and $C^2$ at $0$ as a map into $X'$.
    At $\xi = 0$,
    \[
    D\Xi(0)\eta = -\nabla \cdot (\bar\mu \nabla\eta) = \rieszmap[\bar\mu]\eta,
    \]
    and, for $\eta_1, \eta_2 \in \dotHk[p](\T^d)$ and $\psi \in C^{\infty}(\T^d)$,
    \[
    \left\langle D^2\Xi(0)[\eta_1,\eta_2],\psi\right\rangle_{X',X} = \int_{\T^d}\bar\mu\,\nabla^2\psi\,\nabla\eta_1\cdot\nabla\eta_2\,dx.
    \]
    Moreover, for $\eta_1, \eta_2 \in C^\infty(\T^d)$ and $\zeta \in W^{1,\infty}(\T^d)$,
    \[
    \left\langle D^2\Xi(0)[\eta_1,\eta_2],\zeta\right\rangle_{X',X} = -\sum_{i,j=1}^d\int_{\T^d}\partial_j\zeta\,\partial_i(\bar\mu \,\partial_i\eta_1\,\partial_j\eta_2) \, dx.
    \]
\end{lemma}

\begin{proof}
    Since $p > d/2 + 2$, $H^p(\T^d) \hookrightarrow C^2(\T^d)$. 
    If $\xi \in \mathcal{U}_p$, then $I + \nabla^2\xi$ is everywhere invertible with a positive determinant; hence, $T_\xi$ is a local $C^1$-diffeomorphism. 
    Since $T_\xi$ is homotopic to the identity through the map $\mathrm{Id} + r\nabla\xi$, for $0 \le r \le 1$, it has degree one and therefore is a global diffeomorphism.
    For $\psi \in C^{\infty}(\T^d)$,
    \begin{equation}
        \label{eq:Xi_representation}
        \int_{\T^d} \psi \, d\Xi(\xi) = \int_{\T^d} \psi(x + \nabla\xi(x)) \bar\mu(x) \, dx.
    \end{equation}
    Since $T_\xi$ is a $C^1$-diffeomorphism and $\bar\mu \in L^\infty(\T^d)$, the measure $\Xi(\xi)$ has a bounded density.
    Thus, $\Xi(\xi) - \bar\mu \in X'$, and
    \[
    \left\langle D\Xi(0)\eta, \psi\right\rangle_{X',X} = \int_{\T^d}\bar\mu\,\nabla\psi\cdot\nabla\eta\,dx = \left\langle -\nabla\cdot(\bar\mu\nabla\eta),\psi\right\rangle_{X',X}.
    \]
    Hence, $D\Xi(0)\eta = \rieszmap[\bar\mu]\eta$.
    Differentiating once more gives
    \[
    \left\langle D^2\Xi(0)[\eta_1,\eta_2], \psi \right\rangle_{X',X} = \int_{\T^d} \bar\mu \,\nabla^2\psi \, \nabla\eta_1 \cdot \nabla\eta_2 \, dx.
    \]

    For $\eta_1, \eta_2 \in C^\infty(\T^d)$, integration by parts gives, for $\zeta \in C^{\infty}(\T^d)$,
    \[
    \left\langle D^2\Xi(0)[\eta_1,\eta_2], \zeta \right\rangle_{X',X} = -\sum_{i,j=1}^d \int_{\T^d}\partial_j\zeta \, \partial_i(\bar\mu \,\partial_i \eta_1 \, \partial_j\eta_2)\,dx.
    \]
    Since $\bar\mu \in W^{1,\infty}(\T^d)$ and $\eta_1, \eta_2 \in C^\infty(\T^d)$, the functions $g_j \coloneqq \sum_{i=1}^d \partial_i (\bar\mu\,\partial_i\eta_1\,\partial_j\eta_2)$ belong to $L^\infty(\T^d)$ and satisfy
    \[
    \sum_{j=1}^d\|g_j\|_{L^2} \le C\|\eta_1\|_{\dotHk[p]}\|\eta_2\|_{\dotHk[p]}.
    \]
    Therefore, $D^2\Xi(0)$ is a bounded bilinear map from $\dotHk[p](\T^d) \times \dotHk[p](\T^d)$ to $H^{-1}(\T^d)$, hence to $X'$.
    If $\zeta \in W^{1,\infty}(\T^d)$, choose periodic mollifications $\zeta_n \in C^\infty(\T^d)$ with $\zeta_n \to \zeta$ in $W^{1,q}(\T^d)$ for every finite $q$. 
    The right-hand side of the identity above converges since $g_j \in L^\infty(\T^d)$, and the left-hand side converges since $D^2\Xi(0)[\eta_1,\eta_2] \in X'$ and $\zeta_n \to \zeta$ in $X$.
    The formula therefore holds for every $\zeta \in W^{1,\infty}(\T^d)$.
    The representation~\eqref{eq:Xi_representation} gives differentiability of $\Xi$ for $\xi$ near $0$, and the second-order Taylor remainder, controlled by $H^p(\T^d) \hookrightarrow C^2(\T^d)$, gives that $\Xi$ is $C^2$ at $0$ as a map into $X'$.
\end{proof}

\begin{lemma}[Local displacement interpolation]
    \label{lem:local_displacement_interpolation}
    Let $\mu \in \Pcal_2(\T^d)$ be absolutely continuous with respect to Lebesgue measure and let $\varphi \in W^{2,\infty}(\T^d)$.
    Then, for every $s$ satisfying $|s| \|\nabla^2\varphi\|_{L^\infty} < 1$ and $|s| \|\nabla\varphi\|_{L^\infty} < \frac{1}{2}$, the map $T_{s\varphi}$ is a bi-Lipschitz homeomorphism of $\T^d$ and an optimal transport map from $\mu$ to $(T_{s\varphi})_\sharp\mu$ for the quadratic cost, and $\tau \mapsto \mu_{\tau s}$, $\tau \in [0,1]$, is the displacement interpolation between $\mu$ and $\mu_s$.
    Moreover, for every $a > 0$ with $3a \|\nabla^2\varphi\|_{L^\infty} < 1$ and $4a \|\nabla\varphi\|_{L^\infty} < 1$, the curve $s \mapsto \mu_s = (T_{s\varphi})_\sharp\mu$ is a constant speed $\Wtwo$-geodesic on $[-a,a]$, with velocity $\nabla\varphi$ at $s = 0$.
\end{lemma}

\begin{proof}
    Set $L_s \coloneqq |s| \|\nabla^2\varphi\|_{L^\infty}$ and $\Psi_s(x) \coloneqq \frac{1}{2}|x|^2 + s\varphi(x)$ as a function on $\R^d$ through the periodic lift of $\varphi$.
    Then $\Psi_s$ is a convex function of class $C^{1,1}$ since $\nabla^2 \Psi_s \ge (1-L_s)I$ a.e.
    Also, $\Psi_s - \frac{1}{2}|\cdot|^2 = s\varphi$ is periodic and $T_{s\varphi} = \nabla\Psi_s$ satisfies $T_{s\varphi}(x + k) = T_{s\varphi}(x) + k$ for every $k \in \mathbb{Z}^d$.
    Since $s\nabla\varphi$ is Lipschitz with constant $L_s < 1$, one has $(1-L_s) |x - y| \le |T_{s\varphi}(x) - T_{s\varphi}(y)| \le (1+L_s) |x - y|$, so minimizing over translations in $\mathbb{Z}^d$ gives the corresponding bi-Lipschitz estimates on $\T^d$.
    Hence, $T_{s\varphi}$ is injective, and the degree argument of Lemma~\ref{lem:pushforward_map} gives surjectivity. 
    Since $|s| \|\nabla\varphi\|_{L^\infty} < 1/2$, the displacement $|T_{s\varphi}(x) - x| = |s \nabla\varphi(x)|$ is smaller than $1/2$, so it does not wrap around $\T^d$ and $|T_{s\varphi}(x) - x| = d_{\T^d}(x, T_{s\varphi}(x))$ for every $x \in \T^d$.
    Thus, the segment from $x$ to $T_{s\varphi}(x)$ projects to a minimizing geodesic of $\T^d$.

    Set $\Lambda(v) \coloneqq \inf_{z \in \T^d} \{ s\varphi(z) + \frac{1}{2} d^2_{\T^d}(z,v)\}$, fix $x \in \R^d$, and set $y = T_{s\varphi}(x)$.
    The function $z \mapsto s\varphi(z) + \frac{1}{2}|z-y|^2$ is convex, and its gradient $\nabla \Psi_s(z) - y$ vanishes at $z = x$, so it attains its global minimum there.
    By periodicity of $\varphi$ this minimum equals $\Lambda(y)$, and it equals $s\varphi(x) + \frac{1}{2}d^2_{\T^d}(x,y)$.
    Hence, $(\mathrm{Id}, T_{s\varphi})_\sharp\mu$ is optimal for the quadratic cost by the equivalence (a)$\,\Leftrightarrow\,$(d) in~\cite[Thm.~5.10(ii)]{villani2008optimal}, and since $|T_{s\varphi}(x) - x| = d_{\T^d}(x, T_{s\varphi}(x))$, the displacement interpolation between $\mu$ and $\mu_s$ is $\tau \mapsto (T_{\tau s\varphi})_\sharp\mu = \mu_{\tau s}$.

    To obtain a geodesic on an interval containing $0$, let $a > 0$ satisfy $3 a \|\nabla^2\varphi\|_{L^\infty} < 1$ and $4a \|\nabla\varphi\|_{L^\infty} < 1$.
    By the preceding paragraphs $T_{-a\varphi}$ is a bi-Lipschitz homeomorphism of $\T^d$ and $\mu_{-a} \ll \mathcal{L}^d$.
    Set $\theta_a \coloneqq (\varphi - \frac{a}{2}|\nabla\varphi|^2)\circ T_{-a\varphi}^{-1}$.
    The chain rule gives $\theta_a \in W^{2,\infty}(\T^d)$ with
    \[
    \nabla\theta_a = \nabla\varphi \circ T_{-a\varphi}^{-1}, \qquad \|\nabla^2\theta_a\|_{L^\infty} \le \frac{\|\nabla^2\varphi\|_{L^\infty}}{1 - a\|\nabla^2\varphi\|_{L^\infty}},
    \]
    and $T_{h\theta_a} \circ T_{-a\varphi} = T_{(h-a)\varphi}$, so that $(T_{h\theta_a})_\sharp\mu_{-a} = \mu_{h-a}$.
    By the choice of $a$,
    \[
    2a \|\nabla^2\theta_a\|_{L^\infty} \le \frac{2a\|\nabla^2\varphi\|_{L^\infty}}{1 - a\|\nabla^2\varphi\|_{L^\infty}} < 1, \qquad 2a\|\nabla\theta_a\|_{L^\infty} = 2a\|\nabla\varphi\|_{L^\infty} < \frac{1}{2},
    \]
    so the preceding argument applies to $\mu_{-a}$ and $\theta_a$ with parameter $h = 2a$.
    Therefore, the curve $h \mapsto \mu_{h-a}$, $h \in [0,2a]$, is the displacement interpolation between $\mu_{-a}$ and $\mu_a$, hence a constant speed $\Wtwo$-geodesic, and reparametrizing by $s = h-a$ gives the claim on $[-a,a]$.
    Finally, differentiating $\int_{\T^d} f \, d\mu_s = \int_{\T^d} f \circ T_{s\varphi} \, d\mu$ for $f \in C^\infty(\T^d)$ gives $\partial_s \mu_s + \nabla \cdot (\mu_s(\nabla \varphi \circ T_{s\varphi}^{-1})) = 0$ in the distributional sense, so the velocity at $s = 0$ is $\nabla\varphi$.
\end{proof}

\begin{proof}[Proof of Lemma~\ref{lem:second_variation_pushforward}]
    Let $\widehat{\E} \coloneqq \E \circ \Xi$.
    By the definition of the first variation applied to the curve $s \mapsto \Xi(\xi + s\psi)$,
    \[
    D\widehat\E(\xi)[\psi] = \left\langle D\Xi(\xi)\psi, \Theta_E(\Xi(\xi))\right\rangle_{X',X}
    \]
    for $\psi \in C^{\infty}(\T^d)$.
    By Lemma~\ref{lem:pushforward_map}, $\Xi(\xi) - \bar\mu = D\Xi(0) \xi + o(\|\xi\|_{\dotHk[p]})$ in $X'$. 
    Hence Assumption~\ref{ass:H3} gives $\Theta_E(\Xi(\xi)) = \Theta_E(\bar\mu) + \mathcal{K}_E D\Xi(0)\xi + o(\|\xi\|_{\dotHk[p]})$ in $X$.
    Therefore
    \[
    D(\Theta_E\circ\Xi)(0)[\varphi] = \mathcal{K}_E D\Xi(0)\varphi = \mathcal{K}_E\rieszmap[\bar\mu]\varphi.
    \]
    Notice that $\frac{D\widehat\E(t\varphi)[\psi] - D\widehat\E(0)[\psi]}{t}$ equals
    \[
    \left\langle \frac{D\Xi(t\varphi)\psi - D\Xi(0)\psi}{t}, \Theta_E(\Xi(t\varphi)) \right\rangle_{X',X} + \left\langle D\Xi(0)\psi, \frac{\Theta_E(\Xi(t\varphi))-\Theta_E(\bar\mu)}{t} \right\rangle_{X',X}.
    \]

    By Lemma~\ref{lem:pushforward_map}, in $X'$,
    \[
    \frac{D\Xi(t\varphi)\psi - D\Xi(0)\psi}{t} \to D^2\Xi(0)[\varphi,\psi],
    \]
    while $\Theta_E(\Xi(t\varphi)) \to \Theta_E(\bar\mu)= \Phi$ in $X$. 
    Moreover, by the preceding computation, in $X$,
    \[
    \frac{\Theta_E(\Xi(t\varphi)) - \Theta_E(\bar\mu)}{t} \to \mathcal{K}_E \rieszmap[\bar\mu] \varphi.
    \]
    Therefore $D^2\widehat\E(0)[\varphi,\psi] = \left\langle D^2\Xi(0)[\varphi,\psi],\Phi\right\rangle_{X',X} + \left\langle \rieszmap[\bar\mu]\psi, \mathcal{K}_E \rieszmap[\bar\mu]\varphi \right\rangle_{X',X}$.
    Finally, by Lemma~\ref{lem:local_displacement_interpolation}, for $s$ sufficiently small, the curve $s \mapsto \Xi(s\varphi)$ is a local $\Wtwo$-geodesic through $\bar\mu$ at $s=0$ with initial velocity $\nabla\varphi$. 
    Hence, by the definition of the Wasserstein Hessian,
    \[
    D^2\widehat\E(0)[\varphi,\varphi] = H_\E(\varphi,\varphi).
    \]
    By symmetry of the bilinear form $D^2\widehat\E(0)$, polarization gives the stated bilinear identity.
\end{proof}

\begin{lemma}
    \label{lem:velocity_potential}
    Let $\mu_0 \in \Pcal_2(\T^d)$ be absolutely continuous and $\psi \in W^{2,\infty}(\T^d)$.
    For $s \in \R$ with $|s| \|\nabla^2\psi\|_{L^\infty} < 1$ and $|s| \|\nabla\psi\|_{L^\infty} < 1/2$, set $\mu_s \coloneqq (T_{s\psi})_\sharp\mu_0$ and define, for $y \in \T^d$ and $x = T_{s\psi}^{-1}(y)$, $\varphi_s(y) \coloneqq \psi(x) + \frac{s}{2}|\nabla\psi(x)|^2$.
    Then, for every such $s$:
    \begin{enumerate}[label=(\roman*)]
        \item $\nabla\varphi_s(y) = \nabla\psi\left(T_{s\psi}^{-1}(y)\right)$ and $\nabla^2\varphi_s(y) = \nabla^2 \psi(x) \left(I + s\nabla^2\psi(x)\right)^{-1}$ for a.e.\ $y$.
        In particular, $\nabla\varphi_s \in \mathscr{D}$ with $\|\nabla^2 \varphi_s\|_{L^\infty} \le \|\nabla^2\psi\|_{L^\infty}/(1-|s|\|\nabla^2\psi\|_{L^\infty})$.
        \item Whenever $s+h$ satisfies the same inequalities, $\mu_{s+h} = (T_{h\varphi_s})_\sharp\mu_s$.
        In particular, for $|h|$ small, $h \mapsto \mu_{s+h}$ is a local $\Wtwo$-geodesic through $\mu_s$ with initial velocity $\nabla\varphi_s$.
        \item $\displaystyle \int_{\T^d} |\nabla\varphi_s|^2 \, d\mu_s = \int_{\T^d} |\nabla\psi|^2 \, d\mu_0$.
    \end{enumerate}
\end{lemma}

\begin{proof}
    (i) By Lemma~\ref{lem:local_displacement_interpolation}, $T_{s\psi}$ is a bi-Lipschitz homeomorphism of $\T^d$ and since $\psi \in W^{2,\infty}(\T^d)$, it is differentiable a.e.\ with $DT_{s\psi} = I + s \nabla^2\psi$.
    Since $T_{s\psi}^{-1}$ is Lipschitz, it maps $\mathcal{L}^d$-null sets to $\mathcal{L}^d$-null sets and is differentiable a.e., with $DT_{s\psi}^{-1}(y) = \left(I + s\nabla^2\psi(x)\right)^{-1}$.
    Writing $g \coloneqq \psi + \frac{s}{2}|\nabla\psi|^2 \in W^{1,\infty}(\T^d)$, one has $\nabla g = (I + s\nabla^2\psi)\nabla\psi$ a.e., and $\varphi_s = g \circ T_{s\psi}^{-1}$, so the chain rule gives, for a.e.\ $y$,
    \[
    \nabla\varphi_s(y) = \left(I + s\nabla^2\psi(x)\right)^{-1} \nabla g(x) = \nabla\psi(x),
    \]
    where we used that $I + s\nabla^2\psi$ is symmetric.
    Differentiating once more yields the stated formula for $\nabla^2\varphi_s$.
    The bound follows since the eigenvalues of $I + s\nabla^2\psi$ lie in $[1 - |s| \|\nabla^2\psi\|_{L^\infty}, 1 + |s| \|\nabla^2\psi\|_{L^\infty}]$.
    In particular $\varphi_s \in W^{2,\infty}(\T^d)$ and $\|\nabla\varphi_s\|_{L^\infty} = \|\nabla\psi\|_{L^\infty}$.

    (ii) By (i), for a.e.\ $x$, $T_{h\varphi_s} (T_{s\psi}(x)) = T_{s\psi}(x) + h \nabla\varphi_s(T_{s\psi}(x)) = T_{s\psi}(x) + h \nabla\psi(x) = T_{(s+h)\psi}(x)$,
    that is, $T_{h\varphi_s}\circ T_{s\psi} = T_{(s+h)\psi}$ $\mu_0$-a.e.
    Pushing $\mu_0$ forward gives 
    \[
    (T_{h\varphi_s})_\sharp\mu_s = (T_{h\varphi_s}\circ T_{s\psi})_\sharp\mu_0 = \mu_{s+h}.
    \]
    Moreover, $\mu_s \ll \mathcal{L}^d$ since $T_{s\psi}$ is bi-Lipschitz and $\mu_0 \ll \mathcal{L}^d$, and $\varphi_s \in W^{2,\infty}(\T^d)$ by (i).
    Hence Lemma~\ref{lem:local_displacement_interpolation}, applied to $\varphi_s$ and $\mu_s$, shows that for $|h|$ small $h \mapsto \mu_{s+h}$ is a local $\Wtwo$-geodesic through $\mu_s$ with initial velocity $\nabla\varphi_s$.

    (iii) By the definition of the pushforward and (i),
    \[
    \int_{\T^d}|\nabla\varphi_s|^2\,d\mu_s = \int_{\T^d}|\nabla\varphi_s \circ T_{s\psi}|^2\,d\mu_0 = \int_{\T^d}|\nabla\psi|^2\,d\mu_0. \qedhere
    \]
\end{proof}

\begin{lemma}
    \label{lem:brenier_stability}
    Assume~\ref{ass:H1} and fix $\alpha \in (0,1)$.
    For every $\eta > 0$, there exists $\varepsilon > 0$ such that if $\mu_0, \mu_1$ are probability densities on $\T^d$ with $\|\mu_i - \bar\mu\|_{C^{0,\alpha}} \le \varepsilon$ for $i = 0,1$, then the optimal transport map from $\mu_0$ to $\mu_1$ for the quadratic cost is of the form $T_\psi = \mathrm{Id} + \nabla\psi$ with $\psi \in C^{2,\alpha}(\T^d)$,
    $\|\nabla\psi\|_{C^1} \le \eta$, and $|T_\psi(x) - x| = d_{\T^d}(x, T_\psi(x))$ for $\mu_0$-a.e.\ $x$.
    In particular $\nabla\psi \in \mathscr{D}$.
\end{lemma}

\begin{proof}
    By Lemma~\ref{lem:mu_bar_bounds}, $\bar\mu \in W^{1,\infty}(\T^d) \subset C^{0,\alpha}(\T^d)$ and $c_0 \le \bar\mu \le C_0$.
    Let $\varepsilon \le c_0/2$, so that $c_0/2 \le \mu_i \le C_0 + 1$ and $\|\mu_i\|_{C^{0,\alpha}} \le \|\bar\mu\|_{C^{0,\alpha}} + 1$ for $i = 0,1$.
    We know that the optimal transport map from $\mu_0$ to $\mu_1$ is $T_\psi$, with $\frac{1}{2}|\cdot|^2 + \psi$ convex on $\R^d$ and $\psi$ periodic.
    Normalize $\psi$ by $\int_{\T^d}\psi\,dx = 0$.
    The construction in~\cite{cordero1999transport} produces this map as a limit of couplings whose supports consist of pairs whose displacement is $d_{\T^d}$.
    In particular, the proof of~\cite[Thm.~1]{cordero1999transport} gives
    \begin{equation}
        \label{eq:displacement_realizes_distance}
        |T_\psi(x) - x| = d_{\T^d}\left(x, T_\psi(x)\right) \le \operatorname{diam}(\T^d) = \sqrt{d}/2
    \end{equation}
    for $\mu_0$-a.e.\ $x$.
    Hence $\|\nabla\psi\|_{L^\infty} \le \sqrt{d}/2$, and since $\psi$ is periodic with zero mean, $\|\psi\|_{L^\infty} \le d/4$.

    The function $\psi$ solves, in the Alexandrov sense (see~\cite[\S4]{cordero1999transport}), the Monge--Amp\`ere equation
    \[
    \det\left(I + \nabla^2\psi(x)\right) = f\left(x, \nabla\psi(x)\right), \qquad f(x,p) \coloneqq \frac{\mu_0(x)}{\mu_1(x+p)},
    \]
    and $I + \nabla^2\psi \ge 0$, since $\frac{1}{2}|\cdot|^2 + \psi$ is convex.
    Moreover, $f$ is positive and $\alpha$-H\"older in $(x,p)$, with $\|f\|_{C^{0,\alpha}}$ bounded in terms of $c_0$ and $\|\mu_i\|_{C^{0,\alpha}}$.
    Since $\T^d$ with the flat metric is a compact Hessian manifold, with $\frac{1}{2}|\cdot|^2$ as a local potential, \cite[Thm.~5]{CaffarelliViaclovsky2001} gives $\psi \in C^{2,\alpha}(\T^d)$ with $\|\psi\|_{C^{2,\alpha}} \le C_1$, where $C_1$ depends only on $d$, $\alpha$, $c_0$, $C_0$ and $\|\bar\mu\|_{C^{0,\alpha}}$.

    By~\eqref{eq:displacement_realizes_distance} and the optimality of $T_\psi$, $\int_{\T^d} |\nabla\psi|^2\,d\mu_0 = \Wtwo^2(\mu_0,\mu_1)$.
    Since $\|\mu_i - \bar\mu\|_{L^2} \le \|\mu_i - \bar\mu\|_{C^{0,\alpha}} \le \varepsilon$, \cite[Thm.~1.1]{peyre2018comparison} and the continuous embedding of zero-mean $L^2(\T^d)$ functions into $X'$ give
    \[
    \Wtwo(\mu_0,\mu_1) \le \Wtwo(\mu_0,\bar\mu) + \Wtwo(\bar\mu,\mu_1) \le 2\|\mu_0-\bar\mu\|_{X'} + 2\|\mu_1-\bar\mu\|_{X'} \le C\varepsilon, 
    \]
    and therefore, using $\mu_0 \ge c_0/2$, $\|\nabla\psi\|_{L^2} \le \left(\tfrac{2}{c_0}\right)^{1/2}\Wtwo(\mu_0,\mu_1) \le C_2\,\varepsilon$.
    Since $\|\psi\|_{C^{2,\alpha}} \le C_1$, the map $\nabla\psi$ is Lipschitz with constant $C_1$.
    On the ball centered where $\|\nabla\psi\|_{L^\infty}$ is attained with radius $\min\{\|\nabla\psi\|_{L^\infty}/(2C_1), 1/2\}$, $|\nabla\psi| \ge \|\nabla\psi\|_{L^\infty}/2$, $\|\nabla\psi\|_{L^2}^2 \ge c_d\,C_1^{-d} \|\nabla\psi\|_{L^\infty}^{2+d}$ for $\|\nabla\psi\|_{L^\infty} \le C_1$, and hence
    \[
    \|\nabla\psi\|_{L^\infty} \le C\,\varepsilon^{2/(2+d)}.
    \]
    Interpolating between $C^1(\T^d)$ and $C^{2,\alpha}(\T^d)$ (see, e.g.,~\cite[\S6.8]{gilbarg2015elliptic}) gives
    \[ 
    \|\nabla^2\psi\|_{L^\infty} \le C \|\nabla\psi\|_{L^\infty}^{\alpha/(1+\alpha)} \|\psi\|_{C^{2,\alpha}}^{1/(1+\alpha)}.
    \]
    Choosing $\varepsilon \le c_0/2$ small enough that $\|\nabla\psi\|_{C^1} = \|\nabla\psi\|_{L^\infty} + \|\nabla^2\psi\|_{L^\infty} \le \eta$ gives the claim.
\end{proof}

\section{Estimates for the local stabilization theorem}
\label{app:local-stabilization-estimates}

This appendix collects the estimates used in Subsection~\ref{sec:local-stab}.

\begin{lemma}
\label{lem:l2_density_to_form_domain}
    Assume~\ref{ass:H1}. 
    If $f \in L^2(\T^d)$ with zero total mass, then $z = \rieszmap[\bar\mu]^{-1} f \in \dotHk[2](\T^d)$ and
    \[
    \|z\|_{\dotHk[2]} \le C\|f\|_{L^2}.
    \]
    Conversely, if $z \in \dotHk[2](\T^d)$, then $f = \rieszmap[\bar\mu]z \in L^2(\T^d)$ and $f$ has zero total mass, and
    \[
    \|f\|_{L^2} \le C\|z\|_{\dotHk[2]}.
    \]
\end{lemma}

\begin{proof} 
    Since $f$ has zero total mass, the equation $-\nabla \cdot (\bar\mu \nabla z) = f$ has a unique weak solution $z \in X$.
    Testing the equation with $z$, using $\bar\mu \ge c_0 > 0$, and the Poincaré--Wirtinger inequality, we obtain, for some constant $C_P > 0$,
    \[
    \|\nabla z\|_{L^2}^2 \le \frac{1}{c_0}\int_{\T^d}\bar\mu|\nabla z|^2\,dx = \frac{1}{c_0}\int_{\T^d} f z \,dx \le \frac{1}{c_0}\|f\|_{L^2}\|z\|_{L^2} \le \frac{C_P}{c_0}\|f\|_{L^2}\|\nabla z\|_{L^2}.
    \]
    Therefore, $\|\nabla z\|_{L^2} \le C_P\|f\|_{L^2}/c_0$.
    Using $\nabla\bar\mu = -\sigma^{-1} \bar\mu \nabla\Phi$, we rewrite the equation as 
    \[ 
    -\Delta z + \sigma^{-1} \nabla\Phi \cdot \nabla z = \bar\mu^{-1}f. 
    \] 
    Applying~\cite[Section~6.3.1, Theorem~1]{evans2022partial} to the periodic extension of this equation and covering the torus by finitely many balls yields 
    \[ 
    \|z\|_{H^2} \le C(\|\bar\mu^{-1}f\|_{L^2} + \|z\|_{L^2}) \le \frac{C}{c_0}(1 + C_P^2)\|f\|_{L^2}.
    \]
    Since $\|z\|_{\dotHk[2]} \le \|z\|_{H^2}$, the first estimate follows. 
    Conversely, if $z \in \dotHk[2](\T^d)$, then
    \[ 
    \|\rieszmap[\bar\mu]z\|_{L^2} \le \|\bar\mu\|_{L^\infty} \|\Delta z\|_{L^2} + \|\nabla\bar\mu\|_{L^\infty}\|\nabla z\|_{L^2} \le C_2\|z\|_{\dotHk[2]}
    \]
    for some constant $C_2 > 0$, which concludes the second estimate.
\end{proof}

\begin{lemma}
\label{lem:closed_loop_domain_equivalence}
    Under the hypotheses of Propositions~\ref{prop:linearized_operator_identification} and~\ref{prop:are_diagonal}, $D(L_{\Pi}) = D(L) = \rieszmap[\bar\mu]D(A)$, and the graph norms $\|\cdot\|_{L}$ and $\|\cdot\|_{L_\Pi} \coloneqq \|\cdot\|_{X'} + \|L_{\Pi}\cdot\|_{X'}$ are equivalent in $D(L)$.
    Moreover, $D(L) \hookrightarrow L^2(\T^d)$ continuously.
\end{lemma}

\begin{proof}
    Since $P_\delta\Pi \in \mathcal{L}(X)$, one has $D(A_\Pi) = D(A)$.
    Hence, by Proposition~\ref{prop:linearized_operator_identification}, $D(L_\Pi) = \rieszmap[\bar\mu]D(A_\Pi) = \rieszmap[\bar\mu]D(A) = D(L)$.
    Moreover, $L_\Pi - L = -\rieszmap[\bar\mu]P_\delta\Pi\rieszmap[\bar\mu]^{-1}$ is bounded from $X'$ to $X'$, implying that for every $\nu \in D(L)$, $\|\nu\|_{L} \le (1 + \|P_\delta \Pi\|_{\mathcal{L}(X)}) \|\nu\|_{L_{\Pi}}$ and $\|\nu\|_{L_{\Pi}} \le (1 + \|P_\delta \Pi\|_{\mathcal{L}(X)}) \|\nu\|_{L}$.
    This proves the equivalence of graph norms.
    If $\nu \in D(L)$, then $\nu = \rieszmap[\bar\mu]z$ with $z \in D(A)$. 
    As in the proof of Theorem~\ref{thm:hessian_form}, for $\gamma > C_E$, the norm $\|\cdot\|_{a,\gamma}$ is equivalent to $\|\cdot\|_{\dotHk[2]}$ on $\dotHk[2](\T^d)$. 
    In particular, there exists some $C' > 0$ such that $\|z\|_{\dotHk[2]} \le C'\left(\|z\|_X + \|Az\|_X\right)$.
    Lemma~\ref{lem:l2_density_to_form_domain} gives $\|\nu\|_{L^2} \le C\|z\|_{\dotHk[2]} \le CC'\left(\|\nu\|_{X'} + \|L\nu\|_{X'}\right)$, proving the continuous embedding.
\end{proof}

\begin{proposition}
\label{prop:weighted_linear_density}
    Assume the hypotheses of Proposition~\ref{prop:are_diagonal}.
    Fix $\gamma \in [0, \lambda_\Pi)$. 
    Then, for every $T \in (0,\infty]$, every $f \in L^2(0,T; X')$, and every $\nu_0 \in L^2(\T^d)$ with zero total mass, the problem 
    \begin{equation}
        \label{eq:maximal_regularity}
        \partial_t \nu_t - (L_\Pi + \gamma I)\nu_t = f,
    \end{equation}
    with initial condition $\nu_0$ has a unique solution $\nu \in \mathcal{Y}_T$, and
    \[
    \|\nu\|_{\mathcal{Y}_T} \le C_{\rm lin,\gamma} (\|\nu_0\|_{L^2} + \|f\|_{L^2(0,T;X')}),
    \]
    where $C_{\rm lin, \gamma}$ is independent of $T$.
\end{proposition}

\begin{proof}
    Let $z = \rieszmap[\bar\mu]^{-1}\nu$, $z_0 = \rieszmap[\bar\mu]^{-1}\nu_0$, and $g = \rieszmap[\bar\mu]^{-1}f$. 
    The problem~\eqref{eq:maximal_regularity} is equivalent to $\partial_t z_t + (A_\Pi - \gamma I) z_t = g$, with initial condition $z_0$.
    The operator $S_\gamma \coloneqq A_\Pi - \gamma I$ is positive self-adjoint and satisfies $S_\gamma \ge (\lambda_\Pi - \gamma)I$.
    Moreover, $D(S_\gamma) = D(A)$ and $D(S_\gamma^{1/2}) = \dotHk[2](\T^d)$, with graph norm $\|\cdot\|_{S_\gamma^{1/2}} \coloneqq \|\cdot\|_X + \|S_\gamma^{1/2} \cdot\|_X$ being equivalent to $\|\cdot\|_{\dotHk[2]}$.
    By Lemma~\ref{lem:l2_density_to_form_domain}, $z_0 \in \dotHk[2](\T^d)$ and $\|z_0\|_{\dotHk[2]} \le C\|\nu_0\|_{L^2}$ for some constant $C > 0$.
    Let $P_N$ be the orthogonal projection onto $\operatorname{span}\{e_1,\dots,e_N\}$ and consider the finite-dimensional problem
    \[
    \partial_t z^N_t + S_\gamma z_t^N = P_N g
    \]
    with initial condition $P_N z_0$.
    Then, taking the $X$-inner product with $S_\gamma z_t^N$ yields
    \[
    \frac{1}{2}\frac{d}{dt} \|S_\gamma^{1/2}z_t^N\|_X^2 + \|S_\gamma z_t^N\|_X^2 = \langle P_N g, S_\gamma z_t^N\rangle_X \le \frac{1}{2} \|g\|_X^2 + \frac{1}{2}\|S_\gamma z_t^N\|_X^2.
    \]
    Integrating in time, using $\partial_t z_t^N = -S_\gamma z_t^N + P_N g$, passing to the weak limit $N \to \infty$, and applying lower semicontinuity of the relevant norms yields 
    \[
    \|z\|_{L^\infty(0,T; \dotHk[2])} + \|z\|_{L^2(0,T; D(A))} + \|\partial_t z\|_{L^2(0,T; X)} \le C_{\rm lin, \gamma}\left(\|z_0\|_{\dotHk[2]} + \|g\|_{L^2(0,T; X)} \right).
    \]
    Applying $\rieszmap[\bar\mu]$ and using $D(L)=\rieszmap[\bar\mu]D(A)$, Lemma~\ref{lem:l2_density_to_form_domain}, Lemma~\ref{lem:closed_loop_domain_equivalence}, and the definition of the $X'$-norm gives the estimate in $\mathcal{Y}_T$. 
    Uniqueness follows by applying the same estimate to the difference of two solutions.
\end{proof}

\begin{lemma}
\label{lem:feedback_remainder}
    Assume~\ref{ass:H5}. 
    Then there exists $C_{\rm fb} > 0$ such that, for all $\nu, \eta\in L^2(\T^d)$ with zero total mass,
    \[
    \|\mathfrak{B}_{\rm fb}(\nu) - \mathfrak{B}_{\rm fb}(\eta)\|_{X'} \le C_{\rm fb}\left(\|\nu\|_{L^2} + \|\eta\|_{L^2}\right) \|\nu-\eta\|_{L^2}.
    \]
\end{lemma}

\begin{proof}
    Since $P_\delta\Pi$ has range in $X_\delta$, Assumption~\ref{ass:H5} and the embedding of the subspace of zero total mass functions in $L^2(\T^d)$ into $X'$ gives $\|\nabla P_\delta\Pi\rieszmap[\bar\mu]^{-1} \nu\|_{L^\infty} \le C\|\nu\|_{X'} \le C\|\nu\|_{L^2}$ for some constant $C > 0$.
    Also, for $\rho \in L^2(\T^d)$ and $\nabla\chi \in L^\infty(\T^d)$, for some constant $C' > 0$, $\|\nabla\cdot(\rho\nabla\chi)\|_{X'} \le C'\|\rho\|_{L^2}\|\nabla\chi\|_{L^\infty}$.
    The identity
    \[
    \mathfrak{B}_{\rm fb}(\nu)-\mathfrak{B}_{\rm fb}(\eta) = \nabla\cdot((\nu-\eta)\nabla P_\delta\Pi\rieszmap[\bar\mu]^{-1}\nu) + \nabla \cdot(\eta \nabla P_\delta\Pi\rieszmap[\bar\mu]^{-1}(\nu-\eta))
    \]
    then gives the result.
\end{proof}

\begin{lemma}
\label{lem:full_remainder}
    Assume~\ref{ass:H5}--\ref{ass:H6}. 
    There exists $C_\Pi > 0$ such that, for all $\nu, \eta \in D(L) \subset L^2(\T^d)$ such that $\bar\mu + \nu$ and $\bar\mu + \eta$ are probability densities, with $\|\nu\|_{L^2} < r$ and $\|\eta\|_{L^2} < r$, one has
    \[
    \|\mathfrak{R}_\Pi(\nu)-\mathfrak{R}_\Pi(\eta)\|_{X'} \le C_\Pi\left(\|\nu\|_{L^2}+\|\eta\|_{L^2}\right)\|\nu-\eta\|_{L^2}.
    \]
    In particular, if $\bar\mu + \nu$ is a probability density and $\|\nu\|_{L^2} < r$, then $\|\mathfrak{R}_\Pi(\nu)\|_{X'} \le C_\Pi\|\nu\|_{L^2}^2$.
\end{lemma}

\begin{proof}
    The estimate for $\mathfrak{R}$ is exactly Assumption~\ref{ass:H6}. 
    The estimate for $\mathfrak{B}_{\rm fb}$ is Lemma~\ref{lem:feedback_remainder}. 
    Adding the two estimates gives the claimed bound. 
\end{proof}

\begin{lemma}
\label{lem:weighted_remainder}
    Assume~\ref{ass:H5}--\ref{ass:H6}, and fix $\gamma > 0$.
    Then, there exists $C_{\rm rem,\gamma} > 0$ such that, for every $T > 0$,
    \[
    \|\mathfrak{R}_{\Pi, \gamma}(w)\|_{L^2(0,T;X')} \le C_{\rm rem, \gamma}\|w\|_{\mathcal{Y}_T}^2
    \]
    whenever $w \in \mathcal{Y}_T$, $\|w\|_{L^\infty(0,T; L^2)} < r$, and $\bar\mu+e^{-\gamma t} w_t$ is a probability density a.e.\ in $(0,T)$.
\end{lemma}

\begin{proof}
    Set $\nu_t = e^{-\gamma t}w(t)$. 
    Since $\|\nu_t\|_{L^2} \le \|w_t\|_{L^2} < r$, Lemma~\ref{lem:full_remainder} gives
    \[
    \|\mathfrak{R}_{\Pi,\gamma}(w)(t)\|_{X'} = e^{\gamma t}\|\mathfrak{R}_\Pi(e^{-\gamma t}w_t)\|_{X'} \le C_\Pi e^{-\gamma t}\|w_t\|_{L^2}^2.
    \]
    Hence $\|\mathfrak{R}_{\Pi,\gamma}(w)\|_{L^2(0,T; X')} \le C_\Pi\|w\|_{L^\infty(0,T;L^2)}\|w\|_{L^2(0,T;L^2)}$.
    By Lemma~\ref{lem:closed_loop_domain_equivalence}, $D(L)$ is continuously embedded in $L^2(\T^d)$. 
    Therefore,
    \[
    \|w\|_{L^2(0,T;L^2)} \le C\|w\|_{L^2(0,T;D(L))}.
    \]
    Combining the last two estimates gives the result.
\end{proof}

\begin{lemma}[Uniqueness and maximal continuation]
    \label{lem:maximal_continuation}
    Assume~\ref{ass:H1}--\ref{ass:H3}, the hypotheses of Proposition~\ref{prop:are_diagonal}, and Assumptions~\ref{ass:H5}--\ref{ass:H6}.
    Fix a probability density $\mu_0 = \bar\mu + \nu_0$ with $\nu_0 \in L^2(\T^d)$ and $\|\nu_0\|_{L^2} < r$, and let $\nu$ be a local solution from Assumption~\ref{ass:H6}.
    Then $\nu$ admits a unique maximal continuation $\nu \colon [0,T_{\max} )\longrightarrow L^2(\T^d)$ for some $ T_{\max} \in (0,\infty]$, such that, for every $T < T_{\max}$, $\nu \in \mathcal{Y}_T$, equation~\eqref{eq:nonlinear_equation_x_prime} holds in $L^2(0,T;X')$, and, for every $t \in [0,T_{\max})$, $\bar\mu + \nu_t$ is a probability density and $\|\nu_t\|_{L^2} < r$.
    Moreover, if $T_{\max} < \infty$, then $\limsup_{t\uparrow T_{\max}} \|\nu_t\|_{L^2} \ge r$ or $\sup_{T < T_{\max}} \|\nu\|_{\mathcal{Y}_T} = \infty$.
\end{lemma}

\begin{proof}
    Fix $\gamma \in (0, \lambda_\Pi)$ and let $\nu^1, \nu^2 \in \mathcal{Y}_T$ be two solutions of~\eqref{eq:nonlinear_equation_x_prime} with the same initial condition such that $\bar\mu + \nu^{i}_t$ is a probability density and $\|\nu^i_t\|_{L^2} < r$ for each $t \in [0,T]$ and $i = 1, 2$.
    Setting $z = \nu^1 - \nu^2$, we have $z(0) = 0$ and
    \[
    \partial_t z - L_\Pi z = \mathfrak{R}_\Pi(\nu^1) - \mathfrak{R}_\Pi(\nu^2).
    \]
    Applying Proposition~\ref{prop:weighted_linear_density} to $e^{\gamma t} z_t$, using Lemma~\ref{lem:full_remainder}, and observing that multiplication by the exponential term is bounded on $\mathcal{Y}_\tau$, uniformly for $\tau \le 1$, we obtain, for every $\tau \le \min\{T, 1\}$, 
    \[
    \|z\|_{\mathcal{Y}_\tau} \le C \left(\|\nu^1\|_{L^2(0, \tau; L^2)} + \|\nu^2\|_{L^2(0, \tau; L^2)} \right) \|z\|_{\mathcal{Y}_\tau}
    \]
    for a constant $C > 0$ independent of $\tau$.
    Since $\nu^1, \nu^2 \in L^2(0, T; L^2)$, when $\tau \downarrow 0$, $\|\nu^i\|_{L^2(0, \tau; L^2)} \to 0$ for $i=1,2$.
    Thus $z = 0$ on $[0, \tau_0]$ for some $\tau_0 > 0$.
    Let
    \[
    \tau_* \coloneqq \sup \left\{\tau \in [0,T]: z_t = 0 \text{ for every } t \in [0,\tau]\right\}.
    \]
    The preceding argument shows that $\tau_* > 0$.
    Moreover, since $\nu^1, \nu^2 \in C([0,T]; L^2)$, the two solutions agree on $[0,\tau_*]$. 
    Suppose that $\tau_* < T$, set $\widetilde{\nu}^i_t \coloneqq \nu^i_{\tau_* + t}$ for $i=1,2$, and $\widetilde{z} = \widetilde{\nu}^1 - \widetilde{\nu}^2$.
    Then $\widetilde{z}_0 = 0$. 
    Applying the preceding estimate on $[0,h]$ gives
    \[
    \|\widetilde{z}\|_{\mathcal{Y}_h} \le C \left(\|\widetilde{\nu}^1\|_{L^2(0, h; L^2)} + \|\widetilde{\nu}^2\|_{L^2(0, h; L^2)} \right) \|\widetilde{z}\|_{\mathcal{Y}_h},
    \]
    which using the same argument as above and $h > 0$ sufficiently small, implies $\widetilde{z} = 0$ on $[0,h]$, contradicting the definition of $\tau_*$. 
    Therefore $\tau_* = T$, and uniqueness follows.

    By Assumption~\ref{ass:H6}, there exists a local solution starting from $\nu_0$.
    Since $\mathcal{Y}_T \hookrightarrow C([0,T]; L^2)$ and $\|\nu_0\|_{L^2} < r$, we reduce the interval $[0,T]$ so that $\|\nu_t\|_{L^2} < r$ throughout that interval.
    Let $T_{\max}$ be the supremum of all $T > 0$ for which there exists a solution $\widehat{\nu} \in \mathcal{Y}_T$ satisfying $\widehat{\nu}_0 = \nu_0$, $\|\widehat{\nu}_t\|_{L^2} < r$ and $\bar\mu + \widehat{\nu}_t$ being a probability density for every $t \in [0,T]$.
    The local uniqueness proved above shows that any two such solutions agree on their common interval of existence and hence define a unique solution on $[0, T_{\max})$.
    Indeed, Assumption~\ref{ass:H6} may be applied at the endpoint of any such interval, and the resulting local solution can be concatenated with the preceding one because their $L^2$ traces agree.
    This gives the unique maximal continuation with the stated properties.

    Suppose now that $T_{\max} < \infty$, $\limsup_{t\uparrow T_{\max}} \|\nu_t\|_{L^2} < r$, and $\sup_{T < T_{\max}} \|\nu\|_{\mathcal{Y}_T} < \infty$.
    The latter estimate and monotone convergence in the defining norms of $\mathcal{Y}_T$ imply that $\nu \in \mathcal{Y}_{T_{\max}}$.
    Since $\mathcal{Y}_{T_{\max}} \hookrightarrow C([0,T_{\max}]; L^2)$, there exists $\nu_{T_{\max}} \coloneqq \lim_{t \uparrow T_{\max}} \nu_t$ in $L^2(\T^d)$ satisfying $\|\nu_{T_{\max}}\|_{L^2} < r$.
    Moreover, $\bar\mu + \nu_{T_{\max}}$ is a probability density, since the set of probability densities is closed under $L^2$ convergence on $\T^d$.
    Assumption~\ref{ass:H6}, applied with initial condition $\nu_{T_{\max}}$, provides a local solution starting at $T_{\max}$.
    After reducing its existence interval if necessary, this solution satisfies $\|\nu_t\|_{L^2} < r$. 
    Concatenating it with $\nu$ at $T_{\max}$ gives a continuation beyond $T_{\max}$, contradicting maximality.
    The asserted continuation alternative follows.
\end{proof}

\section{Numerical methodology and supplementary experiments}
\label{app:numerics}

The numerical experiments use a common spectral Galerkin discretization for the Hessian eigenvalue problem.
After describing this discretization, we give the supplementary experiments, including a nonlinear nonlocal test, an empirical attraction scan, and the constrained double-well example from Section~\ref{sec:constrained-dynamics}.
The sphere computation in Section~\ref{sec:sphere-example} uses the same construction, with spherical harmonics in place of a Fourier discretization.

\subsection{Numerical methodology}
\label{app:sec:numerical-methodology}

On the one-dimensional flat torus, we discretize $\T = [0,b]_{\rm per}$ with $N_{\rm fine} = 512$ equispaced grid points.
When a potential is defined on $[-b/2, b/2]$, it is shifted to $[0,b]$ by $x \mapsto x + b/2$.
For the sphere experiment from Section~\ref{sec:sphere-example}, Fourier modes are replaced by spherical harmonics.
When the stationary density is not explicit, it is computed by a fixed-point iteration of the relation
\[
\bar\mu = Z^{-1}\exp\left(-\frac{\Theta_E(\bar\mu)}{\sigma}\right), \qquad Z = \int_{\T} \exp\left(-\frac{\Theta_E(\bar\mu)}{\sigma}\right)\,dx.
\]

Let $(\chi_j)_{j=1}^M$ denote the nonconstant Fourier (or spherical harmonic) basis used for the discretization scaled to improve the conditioning of the $X$-norm.
We assemble the Hessian and metric matrices
\begin{equation}
    \label{eq:discrete-hessian-matrices}
    H_{ij} = a(\chi_j, \chi_i), \qquad M_{ij}= \langle\chi_j, \chi_i\rangle_X,
\end{equation}
and solve the generalized eigenvalue problem $Hv = \lambda Mv$.
The first $N$ eigenfunctions $(\phi_k)_{k=1}^N$ are normalized in $X$, so that $\langle \phi_i, \phi_j \rangle_X = \delta_{ij}$.
We set $\psi_k \coloneqq \rieszmap[\bar\mu]\phi_k \in X'$ and
\[
\xi_N(t) = \sum_{k=1}^N c_k(t) \phi_k, \qquad \nu_N(t) = \rieszmap[\bar\mu]\xi_N(t) = \sum_{k=1}^N c_k(t) \psi_k.
\]
Since the eigenfunctions are $X$-orthonormal, $c_k(t) = \langle\nu_N(t), \phi_k\rangle_{X',X}$ and $\|\nu_N(t)\|_{X'} = \|c(t)\|_{\ell^2}$.
Writing $D_N \coloneqq \operatorname{diag}(\lambda_1^\Pi, \ldots, \lambda_N^\Pi)$, the Galerkin approximation of equation~\eqref{eq:nonlinear_equation_x_prime} is
\begin{equation}
    \label{eq:closed-loop-galerkin-system}
    \dot{c}(t) = -D_N c(t) + R_N(c(t)), \qquad R_{N,k}(c) = \left\langle \mathfrak{R}_\Pi\left(\sum_{j=1}^N c_j \psi_j\right), \phi_k \right\rangle_{X',X}.
\end{equation}

For the one-dimensional experiments, we integrate the coefficient
system adaptively using an explicit Runge--Kutta method of order $8$. 
The sphere experiment uses first-order exponential
time differencing, exploiting the diagonal linear part.

\subsection{Additional numerical examples}
\label{app:sec:standard-models}

For the periodic double-well potential
\[
V(x) = -2\cos(4\pi x), \qquad x \in [0,1]_{\rm per},
\]
the first eigenvalue of $A$ is $\lambda_1 \approx 0.063$, corresponding to slow convergence.
Figure~\ref{fig:spectral_shift} confirms the discrete spectral shift predicted by Corollary~\ref{cor:spectral_gap}.
For the closed-loop operator, the observed spectral gap is well approximated by $\lambda_\Pi = \delta + \sqrt{(\lambda_1-\delta)^2 + \delta^2}$.
The time to 99\% decay decreases with $\delta$, while the integrated control cost increases.

\begin{figure}[!htbp]
    \centering
    \includegraphics[width=0.85\textwidth]{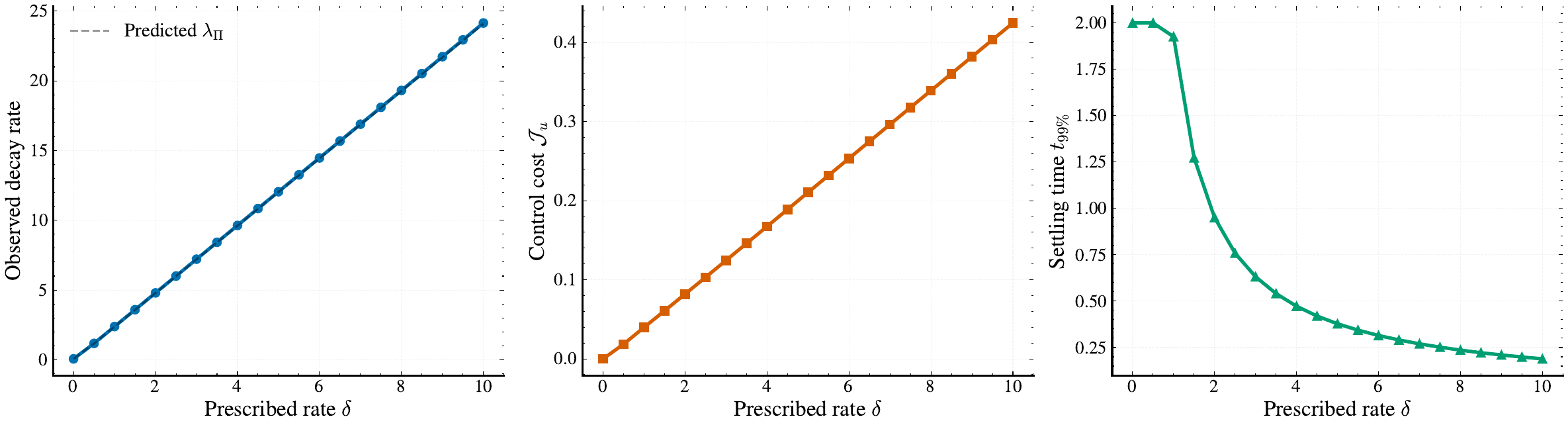}
    \caption{\textbf{Spectral shift and cost analysis}. 
    Double-well potential $V(x) = -2\cos(4\pi x)$ on $\T$ with $\sigma = 0.5$.
    We consider $N_{\text{fine}} = 512$ and $N = 64$.
    Left: observed decay rate and theoretical prediction $\lambda_\Pi$.
    Center: control cost $\mathcal{J}_u = \frac{1}{2} \int_0^T \|u\|_U^2\,dt$.
    Right: time to 99\% decay, defined as the first time $t$ such that $\|\mu_t-\bar\mu\|_{X'} \le 0.01\|\mu_0-\bar\mu\|_{X'}$ or $t = 2$.}
    \label{fig:spectral_shift}
\end{figure}

As a nonlinear nonlocal example, take
\[
\E(\mu) = \int_{\T} V(x) \, d\mu(x) + \int_{\T} \cos(K\ast\mu)(x) \, dx, \qquad K(x) = 5\cos(4\pi x),
\]
with $V(x) = -2\cos(4\pi x)$ and $\sigma = 0.5$.
Its first variation is $V - K * \sin(K * \mu)$.
For this model, Assumption~\ref{ass:H3} holds with
\[
\mathcal{K}_E \nu = -K * \left(\cos(K * \bar\mu)(K * \nu)\right).
\]
Indeed, since $K \in W^{2,\infty}(\T)$, $\|K * \nu\|_{W^{1,\infty}} \le C\|\nu\|_{X'}$.
Taylor expansion of the sine therefore yields
\[
\left\|\Theta_E(\bar\mu + \nu) - \Theta_E(\bar\mu) - \mathcal{K}_E \nu \right\|_X \le C\|\nu\|_{X'}^2,
\]
which verifies Assumption~\ref{ass:H3}.
The nonlinear term is the particular case $U(z) = \cos z$ of Example~\ref{ex:nonlinear-nonlocal-interaction}.
The following numerical experiment is illustrative, so Assumptions~\ref{ass:H4}--\ref{ass:H6} are not verified here.
In Figure~\ref{fig:nl_nonlocal}, the closed-loop dynamics decay for all tested amplitudes, while the uncontrolled dynamics exhibit large nonlinear transients.

\begin{figure}[!htbp]
    \centering
    \includegraphics[width=0.85\textwidth]{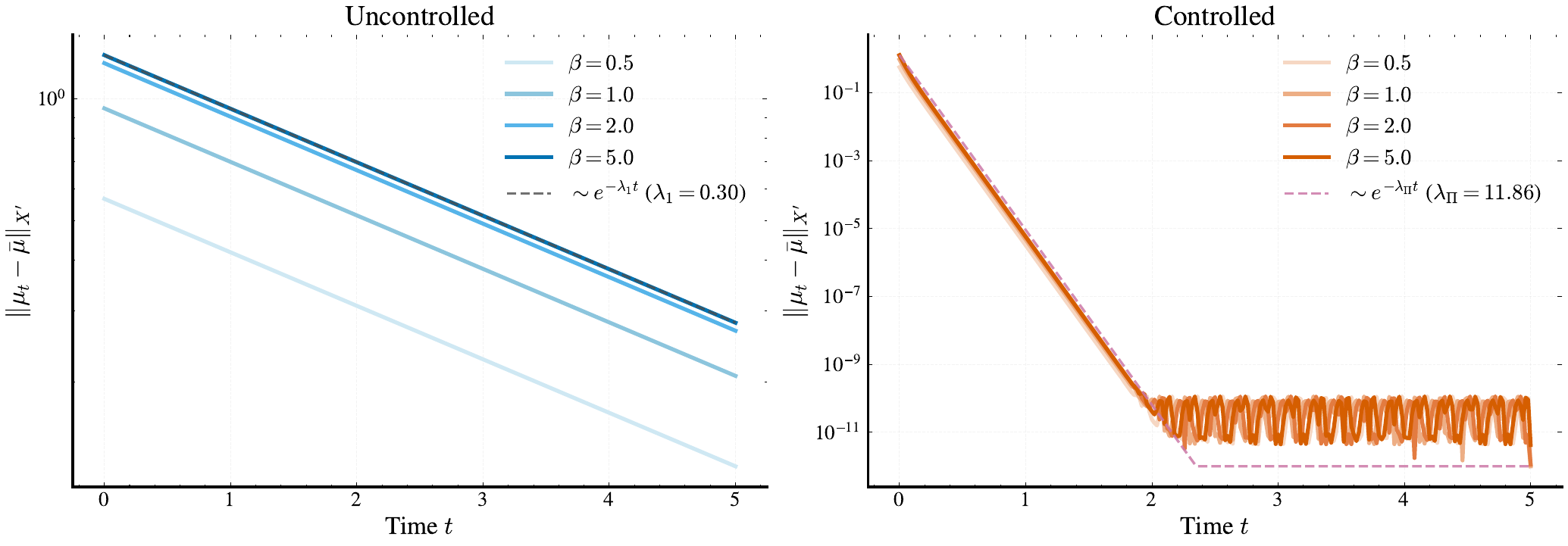}
    \caption{\textbf{Nonlinear nonlocal energy}.
    Decay of the $X'$-norm $\|\mu_t-\bar\mu\|_{X'}$ for initial densities $\mu_0 \propto \bar\mu e^{\beta \cos(2\pi x)}$ for $\beta \in \{0.5, 1.0, 2.0, 5.0\}$.
    We use $\sigma = 0.5$, $N_{\text{fine}} = 1024$, $N = 64$, and target rate $\delta = 5$.
    Left: uncontrolled dynamics.
    Right: closed-loop dynamics with all Hessian modes satisfying $\lambda_k \le \delta$ controlled.}
    \label{fig:nl_nonlocal}
\end{figure}

Following Section~\ref{sec:constrained-dynamics}, we consider the static moment constraint
\[
\int_{\T} \cos(2\pi x) \, \mu(dx) = 0,
\]
for which the unconstrained double-well equilibrium lies on the constrained set.
Thus, the Lagrange multipliers in Proposition~\ref{prop:constrained_projected_hessian} vanish, and the constrained operator $A_{\mathcal{M}}$ is associated with the Hessian form restricted to the constrained tangent space.
Figure~\ref{fig:constrained_static} shows that the projection removes the tunneling mode and increases the spectral gap of the constrained dynamics. 
A rank-$1$ feedback on the first active constrained mode then shifts this gap above the prescribed target rate.

\begin{figure}[!htbp]
    \centering
    \includegraphics[width=0.85\textwidth]{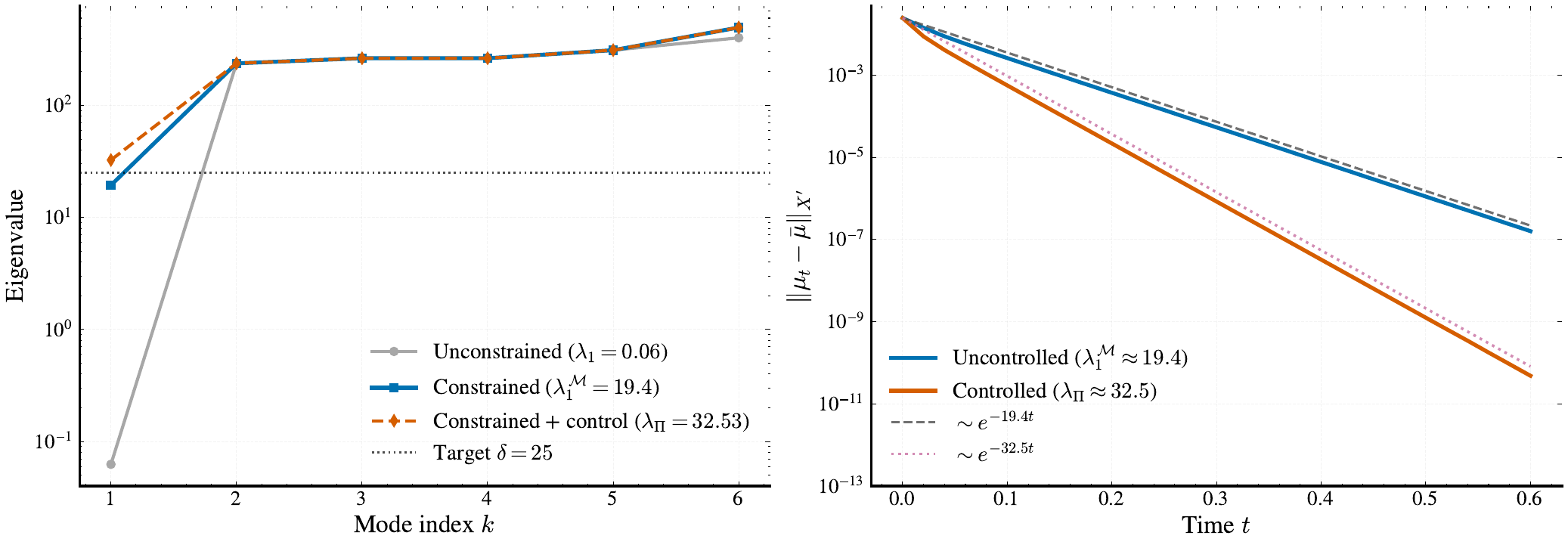}
    \caption{\textbf{Static constrained dynamics}. 
    Double-well potential $V(x) = -2\cos(4\pi x)$ on $\T$ with $\sigma = 0.5$ and constraint $\int_{\T}\cos(2\pi x) \, d\mu=0$.
    Simulation: $N_{\text{fine}} = 1024$, $N = 32$, $T = 0.6$, and $\Delta t = 10^{-3}$.
    Left: first unconstrained eigenvalues, active constrained eigenvalues of $A_{\mathcal{M}}$, and controlled constrained spectrum.
    The constraint removes the tunneling mode and lifts the projected constrained gap from $\lambda_1 \approx 0.06$ to $\lambda_1^{\mathcal{M}} \approx 19.4$.
    With target rate $\delta=25$ and one-mode Riccati weight $q_1 = 25$, the shifted rate is $\lambda_\Pi \approx 32.5$.
    Right: decay of $\|\mu_t - \bar\mu\|_{X'}$ for constrained uncontrolled and controlled dynamics.}
    \label{fig:constrained_static}
\end{figure}

\subsection{Empirical local attraction region}
\label{app:sec:empirical-basin}

Theorem~\ref{thm:local_stab} guarantees stability only for initial data sufficiently close to $\bar\mu$, so we scan a two-parameter family of initial perturbations in the span of two eigenfunctions of $A$.
To preserve positivity, we parameterize initial densities by
\[
\xi_0(r,\theta) = r(\cos\theta \phi_1 + \sin\theta \phi_2), \qquad \mu_0(r,\theta) \propto \bar\mu e^{\xi_0(r,\theta)}.
\]
For each $(r,\theta)$, the closed-loop dynamics are simulated and the observed decay rate $\kappa_{\mathrm{obs}}$ is estimated by linear regression of $\log\|\mu_t-\bar\mu\|_{X'}$ over the final part of the trajectory.
The run is classified as successful if $\kappa_{\mathrm{obs}} \ge \beta\lambda_\Pi$.
Figure~\ref{fig:basin_scan} reports the resulting empirical map.

\begin{figure}[!htbp]
    \centering
    \includegraphics[width=0.85\textwidth]{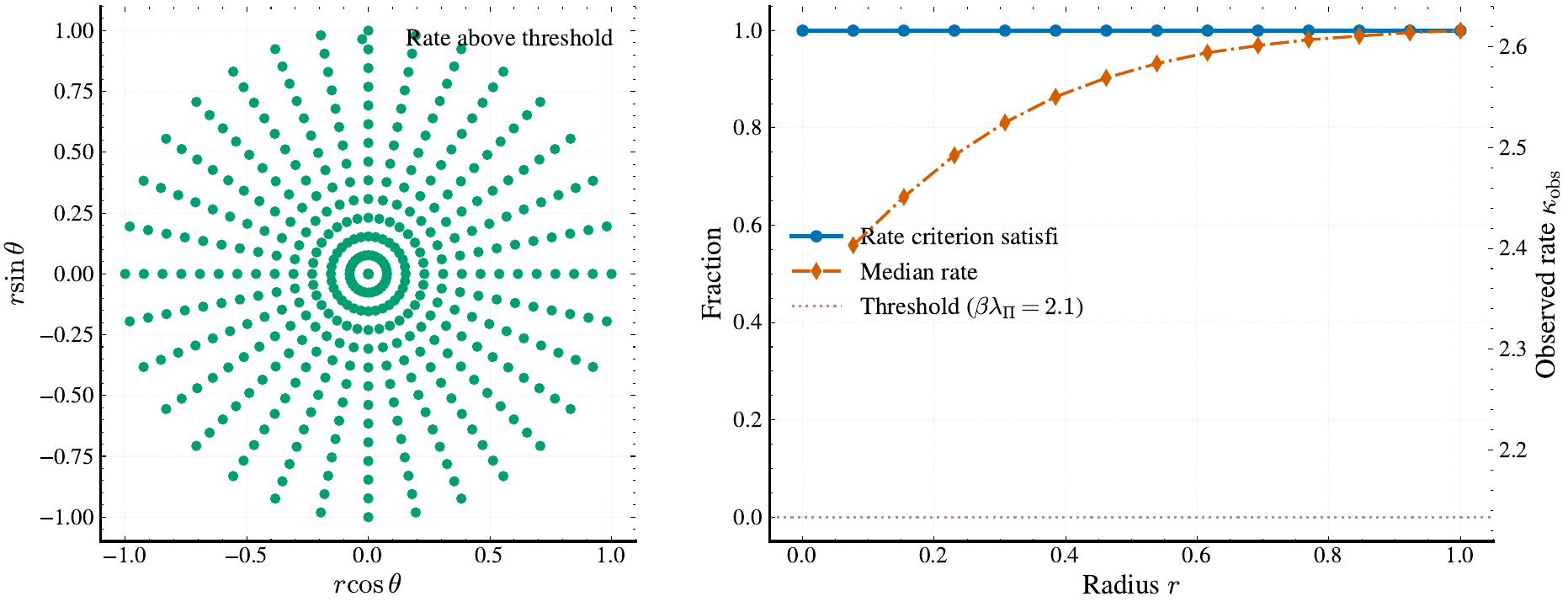}
    \caption{\textbf{Empirical attraction scan}. Closed-loop dynamics with $V(x) = -2\cos(4\pi x)$, $\sigma = 0.5$, and $\delta = 1$.
    Simulation: $N_{\text{fine}} = 512$, $N = 32$, $\Delta t = 5 \times 10^{-4}$, $T = 5$, scan $(r,\theta) \in [0,1] \times [0,2\pi)$ with $14 \times 32$ points, initialization $\mu_0 \propto \bar\mu e^{\xi_0}$, one-mode Riccati weight $q_1 = 1$, and threshold $\beta = 0.9$.
    Left: classification in the $(r \cos(\theta), r\sin(\theta))$ plane.
    Right: success fraction versus radius $r$, with median observed rate overlaid.}
    \label{fig:basin_scan}
\end{figure}

\end{document}